\documentclass[12pt]{imsart}
\usepackage{extsizes}

\usepackage[margin=.5in]{geometry}

\RequirePackage{amsthm,amsmath,amsfonts,amssymb}
\RequirePackage[numbers]{natbib}

\RequirePackage[colorlinks,citecolor=blue,urlcolor=blue]{hyperref}
\RequirePackage{graphicx}

\allowdisplaybreaks
\usepackage{bm}
\usepackage{amsthm}
\usepackage{amsfonts}
\usepackage{amssymb}
\usepackage{enumitem}
\usepackage{subcaption}

\usepackage{amsmath}
\usepackage{dsfont}
\usepackage{mathrsfs}
\usepackage{xcolor}
\usepackage{ wasysym }

\usepackage{tabularx}
\usepackage{float} 

\usepackage{orcidlink}

\startlocaldefs

\newcounter{dummy} 
\numberwithin{dummy}{section}

\theoremstyle{plain}
\newtheorem{thm}[dummy]{Theorem}
\newtheorem{lem}[dummy]{Lemma}
\newtheorem{claim}[dummy]{Claim}
\newtheorem{prop}[dummy]{Proposition}
\newtheorem{cor}[dummy]{Corollary}

\theoremstyle{definition}

\newtheorem{definition}[dummy]{Definition}
\newtheorem{assumption}[dummy]{Assumption}
\newtheorem{construction}[dummy]{Construction}
\newtheorem{remark}[dummy]{Remark}

\newcommand{\D}{\mathbb{D}}
\newcommand{\E}{\mathbb{E}}

\newcommand{\N}{\mathbb{N}}
\newcommand{\PR}{\mathbb{P}}
\newcommand{\R}{\mathbb{R}}
\newcommand{\Q}{\mathbb{Q}}
\newcommand{\bT}{\mathbf{T}}
\newcommand{\bbU}{{\mathbb{U}}}
\newcommand{\Z}{\mathbb{Z}}

\newcommand{\TYPE}{{\tt TYPE}}
\newcommand{\PAR}{{\tt PAR}}

\newcommand{\bzer}{\boldsymbol{0}}

\newcommand{\cD}{{\mathcal{D}}}

\newcommand{\T}{\mathcal{T}}
\newcommand{\U}{\mathcal{U}}

\newcommand{\cX}{\mathcal{X}}
\newcommand{\XX}{\mathcal{X}}
\newcommand{\YY}{{\mathcal{Y}}}
\newcommand{\Zz}{\mathcal{Z}}

\newcommand{\sG}{{\mathscr{G}}}

\newcommand{\F}{\mathscr{F}}

\newcommand{\sT}{{\mathscr{T}}}

\newcommand{\M}{\mathfrak{M}}

\newcommand{\eps}{\varepsilon}

\newcommand{\tr}{{\tt t}}
\newcommand{\tred}{{\tt t}_{\tt red}}

\newcommand{\red}{{\tt red}}

\newcommand{\hgt}{{\operatorname{hgt}}}

\newcommand{\bX}{\boldsymbol{X}}

\newcommand{\bh}{\boldsymbol{h}}
\newcommand{\bN}{\boldsymbol{N}}

\newcommand{\bM}{\boldsymbol{M}}
\newcommand{\be}{{\mathbf{e}}}

\newcommand{\dotms}{\dotsm} 

\newcommand{\Leb}{\operatorname{Leb}}

\newcommand{\Var}{\operatorname{Var}}
\newcommand{\bZ}{{\mathbf{Z}}}
\newcommand{\bY}{{\mathbf{Y}}}

\newcommand{\UU}{{\mathbb{U}}}
\newcommand{\bbT}{{\mathbb{T}}}
\renewcommand{\PR}{\mathbb{P}}

\newcommand{\weakarrow}{\overset{d}{\Longrightarrow}}

\newcommand{\fl}[1]{\left\lfloor #1 \right\rfloor}

\theoremstyle{definition}

\usepackage{verbatim}

\newcommand{\FORGET}{{\tt \tiny FORGET}}

\newcommand{\dghp}{\textup{d}_{\textup{GHP}}}
\newcommand{\dghI}{\dgh^{(\infty)}}
\newcommand{\dghpI}{\dghp^{(\infty)}}

\newcommand{\dgh}{\textup{d}_{\textup{GH}}}

\newcommand{\di}{{\textup{d}}}

\newcommand{\sF}{\mathscr{F}}

\usepackage{color}
\definecolor{BLUE}{RGB}{41,86,143}
\definecolor{RED}{RGB}{178,31,53}

\newcommand{\bb}[1]{\mathbb{#1}}
\newcommand{\bo}[1]{\ensuremath{{\bf #1 } }}

\newcommand{\diag}{{\operatorname{diag}}}

\newcommand{\bp}{{\bf p}}
\newcommand{\bq}{{\bf q}}

\newcommand{\fK}{{\mathfrak{M}_{\tiny \mbox{mp}}}}

\newcommand{\fMd}{{\mathfrak{M}^{\ast d}}}

\endlocaldefs

\begin{document}

\begin{frontmatter}

\title{Multitype L\'evy trees as scaling limits of multitype Bienaym\'e-Galton-Watson trees}
\runtitle{Multitype L\'evy trees as scaling limits of MBGW trees}

\begin{aug}
 \orcid{0009-0008-9599-623X}               

\author[A]{\fnms{Osvaldo}~\snm{Angtuncio Hern\'andez}\ead[label=e1]{osvaldo.angtuncio@cimat.mx}}
\and
\author[B]{\fnms{David}~\snm{Clancy Jr.}\ead[label=e2]{dclancy@math.wisc.edu}}

\address[A]{Departamento de Probabilidad y Estad\'istica, Centro de Investigaci\'on en Matem\'aticas\printead[presep={,\ }]{e1}}

\address[B]{Department of Mathematics, University of Wisconsin\printead[presep={,\ }]{e2}}
\end{aug}

\begin{abstract}
We establish sufficient mild conditions for a sequence of multitype Bienaym\'e-Galton-Watson trees, conditioned in some sense to be large, to converge to a limiting compact metric space which we call a \emph{multitype L\'{e}vy tree}. More precisely, we condition on the size of the maximal subtree of vertices of the same type joined by the root to be large.
While we employ a different conditioning, our result can be seen as a generalization to the multitype setting of the continuum random trees defined by Aldous, Duquesne and Le Gall in 
[Ald91a,Ald91b,Ald93,DLG02].
Our main result is an invariance principle for the convergence of such trees, by gluing single-type L\'{e}vy trees together in a method determined by the limiting spectrally positive additive L\'{e}vy field, as constructed by Chaumont and Marolleau 
[CM21].

Our approach is an improvement of a result about the convergence in the Gromov-Hausdorff-Prohorov topology, of compact marked metric spaces equipped with vector-valued measures, which are then glued via an iterative operation. 
To analyze the gluing operation, we extend the techniques developed by S\'enizergues 
[Sen19,Sen22]
to the multitype setting.  
While the single-type case exhibits a more homogeneous structure with simpler dependency patterns, the multitype case introduces interactions between different types, leading to a more intricate dependency structure where functionals must account for type-specific behaviors and inter-type relationships.
\end{abstract}



\end{frontmatter}

\tableofcontents


\section{Introduction}

In this article, we extend to the multitype setting the classical result that Bienaym\'e–Galton–Watson (BGW) trees conditioned to have size $n$ converge to random real trees. More precisely, we show that a multitype Bienaym\'e–Galton–Watson (MBGW) tree conditioned in some sense to be large, converges to a compact metric measure space that we call the \emph{multitype L\'evy tree}. To the best of our knowledge, this is the first result in which the limiting random real tree is described while simultaneously keeping track of the finite number of types in the limit.

The convergence theory for BGW trees conditioned on having a fixed size was pioneered by Aldous \cite{Aldous.91, Aldous.91a, Aldous.93} in the context of the Brownian continuum random tree, and subsequently generalized to the L\'evy setting by Duquesne, Le Gall, and Le Jan \cite{LL.98a, LL.98b, MR1954248}.
Such works have started an important field of research in probability.
For multitype trees, less is known. Miermont \cite{Miermont.08}, Berzunza \cite{MR3748121}, and de Raph\'{e}lis \cite{deRaphelis.17} show that various MBGW trees converge under appropriate re-scaling and after forgetting the type information to \textit{single-type} $\alpha\in(1,2]$-stable continuum random trees. 
More explicitly, 
in \cite{Miermont.08}, Miermont shows that irreducible MBGW trees, under certain criticality assumptions on the mean matrix of the offspring distribution and a finite variance assumption, converge to (scaled) Brownian continuum random trees. More precisely, let $\vec \varphi = (\vec \varphi^1,\dotsm,\vec \varphi^d)$ be a fixed collection of offspring distributions with at least one $i\in[d]:=\{1,2,\dotsm,d\}$ satisfying $\vec \varphi^i (\{(k_1,\dotsm,k_d): \sum_{j} k_j = 1\})<1$. Denote the mean matrix by $M = (m_{ij})_{i,j\in[d]}$, with $m_{ij} = \E[\xi_{j}^i]$ where $\vec{\xi}^i \sim \vec \varphi^i$ and let $Q^{i} = (q^i_{jk})_{j,k\in[d]}$ be the matrix with $q^i_{jk}: = \E[\xi^i_j\xi^i_k]$. Miermont shows the following.\footnote{To be more precise, Miermont avoids the continuum random tree formalism and instead states his results in terms of scaling limits of height processes. Theorem 1 in \cite{Miermont.08}, together with standard arguments relating height processes to real trees \cite{MR2203728}, yields the stated result.}
\begin{thm}[Miermont {\cite[Theorem 1]{Miermont.08}}]\label{thm:Miermont}
Suppose that $M$ is irreducible, $q_{jk}^i<\infty$ for all $i,j,k\in[d]$, and that the Perron-Frobenius eigenvalue $\varrho$ of $M$ satisfies $\varrho = 1$. Let $\vec{a} = (a_1,\dotsm,a_d)$ be the left eigenvector of $M$ with eigenvalue $1$ normalized so that $\sum_{i} a_i = 1$. Let $T_{n,r}$ be a $\vec \varphi$-MBGW tree, started from an individual $\rho$ having type $i_0$, conditioned on having size at least $nr$ for some $r>0$. 
Let $\vec{\mu}_n$ be the counting measure on the tree and $d_n$ the graph distance on $T_{n,r}$. 
Then, there exists some explicit constant $\sigma\in (0,\infty)$ such that in the Gromov-Hausdorff-Prohorov topology 
\begin{equation*}
    \left(T_{n,r},\rho, \frac{\sigma}{\sqrt{n}}  d_n, \frac{1}{n}\vec{\mu}_n \right) \weakarrow \left( \sT_r, \rho, d_{\sT_r}, \vec{a}\mu\right)
\end{equation*}
where $(\sT_r,\rho, d_{\sT_r},\mu)$ is a Brownian continuum random tree conditioned on having size at least $r$.
\end{thm}

Some of the results of Miermont were extended to infinitely many types by de Raph\'elis in \cite{deRaphelis.17} and a stable analog of Theorem \ref{thm:Miermont} was obtained by Berzunza \cite{MR3748121}. There are some additional assumptions in \cite{MR3748121} which are rather technical because each $\vec \varphi^i$ need not be in the domain of attraction of the same stable law nor even have the same stable index, and so we refer the reader to \cite{MR3748121} for details. 
We remark that in the aforementioned papers, the limit does not keep track of the types of the individuals; that is, the limit is a 'single-type tree'.

As the precise statement of our main result requires some terminology to be introduced, we informally state our main theorem here (for a formal description see Theorem \ref{thm:MAIN} below).
A simulation of the MBGW tree is presented in Figure \ref{multitipeTreeGDS32000V3}.

\begin{figure}
\centering
     \includegraphics[width=.6\textwidth]{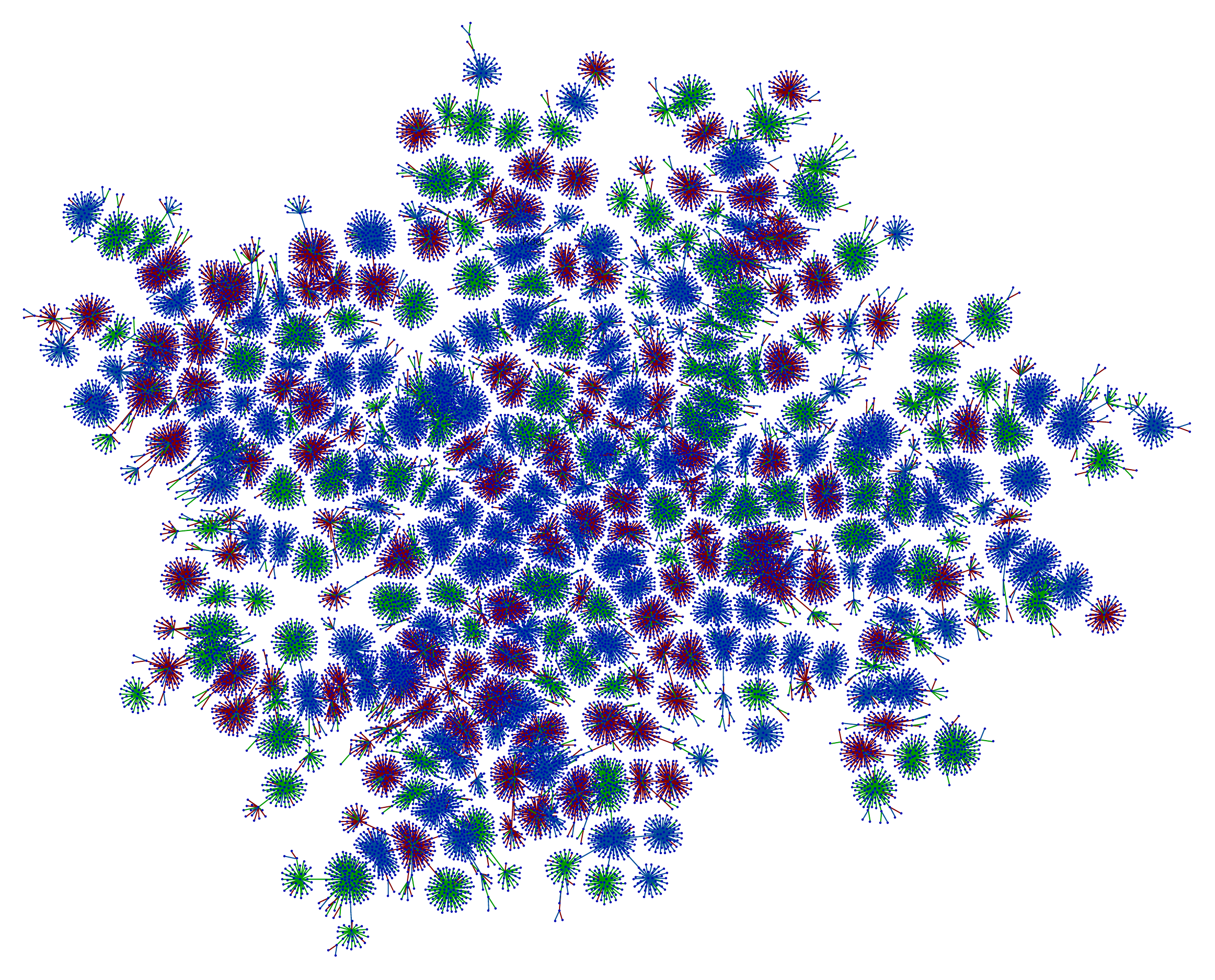}
\caption{We show a simulation of a multitype tree with $d = 3$ and 32000 vertices.}
\label{multitipeTreeGDS32000V3}
\end{figure}

\begin{thm}[Informal statement of Theorem \ref{thm:MAIN}]
Suppose that $T_n$ is a sequence of MBGW trees with $d$ many types conditioned to be large in some sense, and with offspring distribution $\vec \varphi_n$, for $n\in \mathbb{N}$. 
Let $\vec{\mu}_n$ denote the $\R^d_+$-valued measure whose $i^{\text{th}}$ coordinate, $\langle \vec{\mu}_n,\vec{\epsilon}_i\rangle$, represents the counting measure of type $i$ vertices for each $i\in [d]$. 
Let $d_n$ be the graph distance on $T_n$. 
Then, under appropriate assumptions on $\vec{\varphi}_n$, there are real valued sequences $a_n\to\infty$, $b_n^j\to\infty$ for each $j\in [d]$ and a limiting compact metric space $\sT = (\sT,\rho, d_\sT,\vec{\mu})$ with a root $\rho$, a distance $d_\sT$, and an $\R_+^d$-valued measure $\vec{\mu}$ such that, for the diagonal matrix $B_n := \textup{diag}(b_n^1,\dotms,b_n^d)\in \R^{d\times d}$, we have
\begin{equation*}
    \left( T_n, \rho_n, a_n^{-1} d_n, B_n^{-1} \vec{\mu}_n\right) \weakarrow \sT
\end{equation*} in the Gromov-Hausdorff-Prohorov topology.
\end{thm}

To obtain our main theorem, we adapt and improve upon several recent results, notably the convergence of metric measure spaces via decorated spaces, as developed by S\'enizergues \cite{Senizergues.19, Senizergues.20}, and the theory of spectrally positive additive L\'evy fields (spaLfs),  introduced by Chaumont and Marolleau \cite{MR4193902,MR4575008}.

On the one hand, decorated spaces provide a convenient framework of constructing larger metric spaces from smaller spaces. On the other hand, the constructions of Chaumont and Marolleau \cite{MR4193902, MR4575008} allow us to extract single-type information
from a multitype field encoding of a continuous state branching process. To the best of our knowledge, this is the first work that combines these two techniques to establish convergence of metric spaces with finitely many types. 
We believe this methodology can be further developed in other models, by demonstrating the convergence of multitype objects through the gluing of single-type components. 
Within the branching processes framework, we provide the first general result of this type, as previous works addressed the convergence of the tree without keeping track of the types. 

\subsection{Outline of proof technique}
Our first step is to show that a MBGW tree can be naturally interpreted as a discrete decorated space. 
Loosely speaking, one starts with a root, say type $i_0$, and attach a maximal connected subtree of its same type, constructed according to the MBGW offspring distribution. 
Each vertex in this subtree may in turn have children of the other types. 
For each such child, say of type $j\neq i_0$, we attach a new maximal connected subtree of type $j$, and then proceed inductively. 
Thus, to view the MBGW tree as a decorated space, one needs each maximal connected subtree of every type and a procedure to glue them.

The theory of decorated spaces can be viewed as a generalization of the stick-breaking construction of the CRT (see \cite{MR1085326,MR3317158}), but, instead of cutting and pasting segments from $\mathbb{R}_+$, one works with  marked metric measure spaces. That is, one starts with a base metric space containing infinitely many marks. Onto each of these marks, one ``glues'' another metric space which itself contains infinitely many marks. This process continues ad infinitum and the ultimate metric space is the Cauchy completion of the ``limit'' of the gluing procedure. Roughly speaking, in order to show convergence in the Gromov-Hausdorff-Prohorov (GHP) topology of decorated spaces, one needs to show (1) the constituent marked metric spaces converge in the GHP topology along with the marks, and (2) establish uniform control of the size of all the metric spaces added after finitely many gluing operations. This theory was largely developed in \cite{Senizergues.19,Senizergues.20}, however, the formulation found therein is not applicable to our setting. Here we rely on more direct assumptions that we can verify in our setting. For example, in \cite{Senizergues.20} it is assumed that the measures on the decorations are probability measures that ``push the mass to the leaves'' in some sense, whereas our assumptions deal directly with the mass of the decorations. 
Our approach circumvents the need to construct and define a 'multidimensional height process', which appears intractable. The trade-off is that the multitype tree must instead be analyzed through the maximal connected subtrees of each type—that is, the exploration proceeds via these subtrees rather than the usual generational structure.

In the continuum setting, we show that the previous construction can be extended to yield a multitype L\'evy tree, thus formalizing its definition.
This extension relies on two key ingredients: (1) a precise description of each maximal connected subtree (a L\'evy tree) of a given type together with the exploration of such subtrees, and (2) a specification of the marks at which subtrees of the other types are attached.
For the first, we build on the classical constructions of the single-type case \cite{MR1954248, LL.98a, LL.98b}, together with the more recent analysis of spaLf's and multidimensional first hitting times developed in  \cite{MR4193902,MR4575008}. 
Here, we develop some fluctuation-theoretic properties of spaLf's  that are crucial in the multitype framework, showing that maximal connected components can be described through recursive equations naturally arising from the analysis of multidimensional hitting times. 
This yields a simple way to explore the subtrees and their masses.
For the second ingredient, we adapt convergence techniques from random graph theory to encode and control the way maximal connected subtrees are attached. Specifically, in the tree exploration process, we glue subtrees according to their ordered size, using a Poisson point process to govern the attachment mechanism.

\section{Main Results}

Since the formal definition of the multitype L\'evy tree—and consequently our main theorem—relies on several technical concepts and on our new construction, the following subsections are devoted to introducing these ingredients step by step. We begin by defining glued decorations and establishing a general convergence result for them, which is Theorem \ref{thm:graphConv1}. We then present in Subsections \ref{subsectionEncodingMTsAsWalks} and \ref{subsectionIntroMBGWTreesAsGluedDecorations} the coding of multitype trees via random walks and excursions, which is used to construct MBGW trees as glued decorations. We extend this construction to the continuum setting in Subsection \ref{subsectionMLevyTrees}. After briefly recalling the notions of L\'evy processes, their height processes, excursion measures, and single-type marked L\'evy trees, we show in Subsection \ref{subsectionContinuumDecoration} how to construct a continuum decoration.
After describing the discrete and continuous construction of the glued decorations, we then formally state our main result in Subsection \ref{refMainResult}.

\subsection{Convergence of Glued Decorations} 

We assume the reader is familiar with the Gromov-Hausdorff-Prohorov (GHP) topology (refer to Section \ref{ssec:GHP} for the standard definitions). 

S\'{e}nizergues \cite{Senizergues.20} has developed a program to show the convergence in the GHP topology of a sequence of metric measure spaces obtained by gluing smaller metric measure spaces in a very general fashion. We start with the Ulam tree $\bbU$ defined by 
\begin{equation}\label{eqn:bbU}
\bbU := \bigcup_{n\ge 0} \N^n,\qquad \text{where}\qquad \N = \{1,2,\dotsm\}\text{  and 
} \N^0 = \{\emptyset\}.
\end{equation}
There is a natural genealogical order $\preceq$ on $\bbU$ defined by $\bp\preceq \bq$ if and only if $\bp$ is a prefix of $\bq$, i.e. $\bp:=(p_1,p_2,\dotsm, p_n)$ and $\bq := (p_1,p_2,\dotsm, p_n, q_{n+1},q_{n+2},\dotsm,q_{n+m})$ for some $m\ge 0$. We will write $|\bp| = n$ if $\bp\in \N^n$. Given any $\bp = (p_1,\dotsm,p_n)$ and $\bq = (q_1,\dotsm, q_k)$, we will denote by $\bp\bq$, the concatenation of $\bp$ and $\bq$, as the element $(p_1,\dotsm, p_n,q_1,\dotsm,q_k)$. If $\bp\neq \emptyset,$ we will denote by $\PAR(\bp) = (p_1,\dotsm, p_{n-1})$ the parent of $\bp$.

Let $\fMd$ denote the collection of compact rooted metric measure spaces equipped with finite $\R^d_+$-valued measures, and let $\mathfrak{M}^{\ast d}_{\tiny \mbox{mp}}$ denote the collection of compact rooted metric measure spaces with \emph{infinitely many marked points} and equipped with a finite $\R_+^d$-valued measure (see Subsection \ref{subsubsectionIMPMM} for a precise definition of this space).

A \textit{decoration} $\cD:\bbU\to \mathfrak{M}_{\tiny \mbox{mp}}^{*d}$ is a map that associates to each $\bp\in \bbU$ a \emph{subdecoration} $\cD(\bp): = \left(D_\bp, \di_\bp , \rho_\bp,\vec{\mu}_\bp, (x_{\bp;\ell})_{\ell\ge 1}\right)$, for an arbitrary collection of spaces $(\cD(\bp);\bp \in \bbU)$ in $\mathfrak{M}_{\tiny \mbox{mp}}^{*d}$. Note that in \cite{Senizergues.20}, the convergence result only handles marked metric measure spaces equipped with probability measures.

Following \cite{Senizergues.20}, we construct a rooted metric measure space $\sG^*(\cD)$ from the decoration. In words, it is defined as the quotient space  $\left(\bigsqcup_{\bp\in\bbU} D_\bp\right)/\sim$
 where $\sim$ is the smallest equivalence relation such that for all $\bp\in\bbU,\ell\in \N$ we have $  D_{\bp} \ni x_{\bp;\ell} \sim \rho_{\bp\ell}\in D_{\bp\ell}$ (recall that $\bp \ell$ the concatenation of two elements in $\bbU$). There is a natural metric we can place on $\sG^*(\cD)$ which we delay until Section \ref{sssec:metricDecorations}. If we denote by $\vec{\epsilon}_j$ the $j$-th standard basis vector in $\R^d$ and $\langle  \vec{\mu}_\bp, \vec{\epsilon}_j\rangle (\cdot)$ is the $j$-th coordinate measure of $\vec{\mu}_\bp$, we can define a measure $\vec{\mu}$ on $\sG^*(\cD)$ by
\begin{equation}\label{eqn:measuredef}
\langle\vec{\mu},\vec{\epsilon}_j\rangle(A) := \sum_{\bp\in \bbU} \langle  \vec{\mu}_\bp, \vec{\epsilon}_j\rangle (A)\in[0,\infty],\qquad \forall\ j\in [d] ,\ \forall \textup{ Borel set }A \subseteq \sG^*(\cD)
\end{equation}which, by Carath\'eodory's extension theorem, exists and is unique. Here we are a bit sloppy and associate with each measure $\vec{\mu}_\bp$ on $D_\bp$ its pushforward under the canonical maps $D_\bp \hookrightarrow \bigsqcup_{\bq} D_\bq \twoheadrightarrow \sG^*(\cD)$ where the former is the inclusion map and the latter is the quotient map.

\begin{definition}[Decoration, subglued decoration and glued decoration]\label{definitionDecorationGluedMetricspaceGluedDecoration}
A \emph{subglued decoration} $\sG^*(\cD)$ of a decoration $\cD$, is defined by 
\begin{equation*}
\sG^*(\cD): = \left(\big(\bigsqcup_{\bp\in \bbU} D_\bp \big)/\sim,d_{\sG^*(\cD)},\rho_\emptyset,\vec \mu\right),
\end{equation*}where $\sim$ is the smallest equivalence relationship which makes 
\begin{equation}\label{defIdentifyingInfinitelyMarkedPoinsWithRootsOfTheSubsequenceDecorations}
x_{\bp;\ell}\sim \rho_{\bp \ell}\qquad \mbox{ for every } \bp\in \bbU \mbox{ and } \ell\in \N,
\end{equation}where the root $\rho_\emptyset$ is the same as the root of the subdecoration $\cD(\emptyset)$, and $d_{\sG^*(\cD)}$ is a  distance in $\sG^*(\cD)$. 
Finally, from a subglued decoration we construct a \emph{glued decoration} $\sG(\cD)$, as  the Cauchy completion of the former.
We denote by $d_{\sG(\cD)}$ and $\vec \mu$ its distance and measure. 
\end{definition}
A priori, we should not expect that the glued decoration $\sG(\cD)$ is compact nor that its measure is finite almost surely (or even locally). In order to guarantee both these properties we impose the following assumption. We recall that the \emph{diameter} of a metric space $(X,d)$ is defined as  $\textup{diam}(X):=\sup_{x,y\in X}d (x,y)$. 

\begin{assumption}\label{ass:uniformBound}  The decoration $\cD:\bbU\to \mathfrak{M}^{\ast d}_{\tiny \mbox{mp}}$ satisfies
\begin{enumerate}
\item\label{ass:uniformBound1} For each $j\in [d]$, the summation $\displaystyle\sum_{\bp\in \bbU} \langle \vec{\mu}_\bp,\vec{\epsilon}_j \rangle(D_\bp) <\infty.$
\item\label{ass:uniformBound2} For each $h\ge 1$ and every $\eps>0$ there exists a finite set $I(h,\eps)$ such that
\begin{equation*}
\sup_{\substack{\bp\notin I(h,\eps),\\ |\bp| = h}} \operatorname{diam}(D_\bp)\le \eps. 
\end{equation*}

\item\label{ass:uniformBound3} For all $\eps>0$, there exists an $h = h(\eps)$ sufficiently large, such that for each $\bp\in \bbU$ with $|\bp|> h$ and each $y\in D_\bp\hookrightarrow\sG(\cD)$ there exists $\bq\in \bbU$ with $|\bq|\le h$ and some $z\in D_\bq$ such that $d_{\sG(\cD)}(y,z)\le \eps$.
\end{enumerate}
\end{assumption}

Let us explain informally what these conditions are doing. Condition (1) guarantees that the $[0,\infty]^d$-valued measure constructed in \eqref{eqn:measuredef} is actually $\R_+^d$-valued. Conditions (2)+(3) together imply compactness of $\sG(\cD)$. Condition (3) helps control the Cauchy completion operation $\sG^*(\cD)\mapsto \sG(\cD)$ while Condition (2) helps control the growth of metric balls around the root.
The following proposition is easy to prove. 
\begin{prop}\label{prop:copmact1}
If $\cD$ satisfies Assumption \ref{ass:uniformBound}, then $\sG(\cD)\in \fMd$.
\end{prop}

To state the next theorem, we define the \emph{total mass} of a subdecoration $D_\bp$ as  $\|\vec{\mu}_\bp\|(D_{\bp}):=\sum_{j=1}^d \langle \vec{\mu}_\bp,\vec{\epsilon}_j \rangle(D_\bp)$.
A similar definition holds also for general measure spaces.  
The following limit theorem shows that a \textit{uniform} analog of Assumption \ref{ass:uniformBound} is sufficient for convergence of a sequence of glued decorations. 
As is now standard, we prove the convergence in the GHP topology. 

\begin{thm}\label{thm:graphConv1} 
Suppose that $(\cD^n;n\ge 1)$ and $\cD$ are a collection of decorations on $\bbU$, where $\cD^n(\bp): = \left(D^n_\bp, \di^n_\bp , \rho^n_\bp,\vec{\mu}^n_\bp, (x^n_{\bp;\ell})_{\ell\ge 1}\right)$ for every $n\in \bb{N}$. Assume that:
\begin{enumerate}[label=\textbf{H.\arabic*}]
\item \label{ass:1thm1} For each $\bp\in \bbU$, $\cD^n(\bp)\to \cD(\bp)$ in the infinitely marked GHP topology. 
\item\label{ass:2thm1}For every $h\ge 1$, and for   each $\eps>0$, there exists some finite collection $I(h,\eps)\subset \N^h= \{\bp\in \bbU:|\bp| =h\}$  and some $n_{h,\eps}$ sufficiently large such that {for all }$n\ge n_{h,\eps}$, we have
\begin{equation*}
    \sum_{\substack{ \bp\notin I(h,\eps)\\ |\bp| = h}} \|\vec{\mu}_\bp^n\|(D^n_{\bp}) \le\eps \qquad\text{and}\qquad \sup_{\substack{\bp\notin I(h,\eps)\\ |\bp|=h }} \textup{diam}(D^n_{\bp}) \le \eps.
\end{equation*}

\item\label{ass:3thm1} For every $\eps>0$, there exists some $h = h(\eps)$ sufficiently large and $n_{h,\eps}$ sufficiently large such that for all $n\ge n_{h,\eps}$
\begin{equation}\label{eqn:massH3}
    \sum_{\bp: |\bp|>h} \|\vec{\mu}_\bp^n\|(D^n_{\bp}) \le \eps.
\end{equation}Furthermore, {for all $n\ge n_{h,\eps}$}, for all $\bp\in \bbU$, and for all $y\in D^n_{\bp}$, there is some $\bq$ with $|\bq| \le h$ and some $z\in D^n_{\bq}$ such that $d_{\sG(\cD^n)}(y,z)\le \eps$.
\end{enumerate}
Then $\cD$ satisfies Assumption \ref{ass:uniformBound} and we have the convergence of glued decorations $\sG(\cD^{n}) \to \sG(\cD)$ in the GHP topology.
\end{thm}

The above assumptions can be interpreted as follows. Assumption \ref{ass:1thm1} is about the convergence of each subdecoration. 
Assumption \ref{ass:2thm1} says that for any height $h$ in $\bbU$, almost all subdecorations growing at height $h$, except for a finite collection, have a total contribution to the \emph{mass} and \emph{diameter} that becomes negligible as $n\to \infty$. 
Assumption \ref{ass:3thm1} says that the total mass of all subdecorations growing after height $h$ becomes negligible, for some $h$ large enough; and that subdecorations growing after height $h$, are \emph{close enough} to all decorations growing before height $h$.

We remind the reader that we delay the discussion of the  metric on $\sG(\cD)$ until Section \ref{ssec:GHP}.
Theorem \ref{thm:graphConv1} is an extension of an analogous result in \cite{Senizergues.20} wherein the author deals with probability measures and includes {a} technical assumption that the measures on the decorations are ``pushed to the leaves'' in some sense. S\'{e}nizergues mentions the technical assumption placed on the measures can be dropped under an appropriate change in assumptions, which we provide in Theorem \ref{thm:graphConv1}.

\subsection{Encoding Multitype Trees with $\Z^d$-valued Walks}
\label{subsectionEncodingMTsAsWalks}

Let us now explain the depth-first encoding of multitype rooted planar trees in terms of a collection of $d$ many $\Z^d$-valued walks. This is analogous to constructions found in \cite{MR3449255, Hernandez:2020, CKL.22}. Before describing the encoding of a multitype tree, we recall the definition of the Ulam tree, and the \L ukasiewicz path (also called depth-first walk) for a single-type planar tree. We use here the usual terminology for trees (see \cite{MR850756}).

\subsubsection{Coding a single-type tree}\label{sec:SingleEncode} A plane tree $\tr$ is a non-empty finite subset of $\bbU$ such that $\emptyset\in \tr$; if $v\in \tr\setminus\{\emptyset\}$, then $\PAR(v)\in \tr$; and if $v\in \tr$, there is a finite $\chi(v)\in\{0,1,\dotsm\}$ such that $vj\in \tr$ if and only if $1\le j \le \chi(v)$. 
Here  $\chi(v):=\chi_\tr(v)$ denotes the number of children that the vertex has.
We let $\bbT$ be the collection of all finite plane trees.

Consider a plane tree $\tr$ of size $n=\#\tr$ rooted at the vertex $v_1 = \emptyset$. Label the vertices in a \emph{depth-first order}: starting at the root $v_1$ and proceeding clockwise through the vertices and labeling them $v_1,\dotsm, v_n$ using the lexicographical order. 

We can encode the entire tree using the \emph{\L ukasiewicz path} $X_\tr = (X_\tr(m) ; m = 0,1,\dotsm, n)$, where
\begin{equation*}
X_\tr(m): = \sum_{k=1}^m (\chi_\tr(v_k)-1),\qquad m = 0,1,\dotms, n.
\end{equation*} Note that $X_{\tr}(0) = 0$. 
It is easy to see \cite{MR2203728} that $X_\tr(m)\ge 0$ for $m = 0,1,\dotms, n-1$, $X_\tr(n) = -1$, and $X_\tr(m+1)-X_\tr(m)\ge -1$ for all $m \ge 0$. Conversely, given any walk $X$ with the previous properties, one can construct a plane tree $\tr$ such that $X = X_\tr$. 

In {the} sequel we will frequently concatenate the \L ukasiewicz paths (or more accurately the graphs of the paths) of an ordered collection of finite trees $(\tr_\ell;\ell\ge 1)$ in order to get a walk $(X(m); m \geq 0)$ that encodes the entire forest. That is, let $v_1,v_2,\dotsm$ be the enumeration of the vertices in $(\tr_\ell;\ell\ge 1)$ where $v_1,\dotms, v_{\#\tr_1}$ is the depth-first ordering of the vertices in $\tr_1$, similarly,  $v_{\#\tr_1+1},\dotsm, v_{\#\tr_1+\#\tr_2}$ is the depth-first ordering of the vertices in $\tr_2$, etc. Set
\begin{equation*}
X(m): = \sum_{k=1}^m (\chi(v_k)-1) ,\qquad m \geq 0.
\end{equation*}Consequently, each tree is encoded by an \textit{excursion} of the $\Z$-valued walk $X$ above its running minimum. That is, if $\tau(\ell) = \min\{m: X(m) = -\ell\}$ for $\ell\geq 1$, then the tree $\tr_\ell$ is encoded by the walk 
\begin{equation*}X_{\tr_\ell}(m) = X(m + \tau(\ell-1)) + (\ell-1)\qquad\textup{ for }\qquad m=0,\dotsm, \tau(\ell)-\tau(\ell-1).
\end{equation*}

Denote by $\hgt(v)$ the \emph{height} of a vertex $v$, i.e. its distance to the root. Here we use the usual graph distance on the tree $\tr$. We define the corresponding height function of the tree $\tr$, denoted by $H_\tr = (H_{\tr}(m);m = 0,1,\dotms,\#\tr-1)$, as $H_{\tr}(m): = \hgt(v_{m+1}).$ Note that we can recover $H_\tr$ as a function of $X_\tr$ by
\begin{equation}\label{eqn:HeightDiscFunctional}
H_\tr(m) = \#\{k\in \{0,\dotms, m-1\}:X_{\tr}(k) = \min_{r\in \{k,\dotms ,m\}} X_\tr(r)\},
\end{equation}for $m\in \{1,2,\dotms,\#\tr-1 \}$
See, for example, \cite[Proposition 1.2]{MR2203728}. {The height function can be easily extended to the concatenated} $X$ encoding a forest.

\begin{remark} \label{rmk:introRemark1}
In the sequel we will frequently say that the walk $X$ is made by concatenating the walks $(X_{\tr_\ell};\ell\geq 1)$ as opposed to $X$ is made using the \L ukasiwicz encoding from the concatenation of the depth-first orderings of the vertices in $(\tr_\ell;\ell\geq 1)$. Similarly, we say that a height process $H$ is made by concatenating the excursions $(H_{\tr_\ell};\ell\geq 1)$. 
\end{remark}

\subsubsection{Multitype encoding}\label{subsectionMultitypeEnconding}

We define a \textit{multitype tree} $\tr$ with $d$ types as a rooted planar tree equipped with a type function, written as $\TYPE$ which assigns to each vertex $v\in \tr$ a specific type in the set $[d]$. 
If $\TYPE(v) = i$ we will use interchangeably the phrases ``$v$ is type $i$'', ``$v$ is of type $i$'' and ``$v$ has type $i$''. We say the tree $\tr$ is rooted at a type $i$ vertex if its root $\rho_\tr$ has type $i$.

Following \cite[p. 2728]{MR3449255}, we say that $\tr'\subseteq\tr$ is a type $i$ subtree if it is a maximally connected subgraph of $\tr$ such that all its vertices have type $i$. We root type $i$ subtrees at the unique element of minimal height. We also define the reduced tree $\tred$ of $\tr$ as the plane tree obtained by collapsing all vertices in each type $i$ subtree into a single type $i$ vertex but retaining all the edges between the typed subtrees (of different types). That is if $\tr_i$ and $\tr_j$ are two type $i$ and $j$ subtrees of $\tr$ and there is some $u\in \tr_i$ and $v\in \tr_j$ such that the parent of $v$, $\PAR(v)$, is $u$ then in the reduced subtree $\tred$ the vertex corresponding to $\tr_i$ has a child which corresponds to the subtree $\tr_j$. 
The tree $\tred$ is also a multitype tree where each vertex $v_\red\in \tred$ is given the type of each vertex in the corresponding collapsed tree. 

Analogously to the single-type case above, we let $\vec \chi(v) := (\chi^1(v),\dotsm, \chi^d(v))$ denote the vector where $\chi^i(v)$ is the number of type $i$ children of vertex $v$.

Consider the reduced tree
$\tred$ {of} a fixed multitype tree $\tr$. The reduced tree inherits its planar structure from $\tr$. Label all the type $i\in[d]$ vertices in $\tred$ by $(w^i_\ell;\ell = 1,2,\dotsm)$ in \textit{breadth-first order}. Recall that the breadth-first order of a plane (single-type) tree, is a labeling of its vertices in the following way: label the root $w_1$, label its $\chi(w_1)$ children $w_2,\dotsm, w_{\chi(w_1)+1}$, then label the $\chi(w_2)$ children by $w_{\chi(w_1)+2},\dotms, w_{\chi(w_1)+\chi(w_2)+1}$ etc. 
That is, going from generation zero to the last generation, label the individuals of each generation from left to right. 
The vertices $(w_\ell^i;\ell\ge 1)$ are the subsequence of the type $i$ vertices in the breadth-first list constructed from $\tred$.

By construction each vertex $w\in \tred$ corresponds to a subtree $\tr'_w$ (of some type in $[d]$), seen as a subtree of $\tr$. 
Fix an $i\in [d]$ and $\ell\geq 1$. 
For any vertex $w^i_\ell$, denote by $\tr_\ell^i$ its corresponding type $i$ subtree. To code the tree $\tr^i_\ell$, label its vertices in depth-first order as $v^i_1,v^i_2,\dotsm, v^i_q$ for $q = \#\tr^i_\ell$. 
In Figure \ref{figTreeWithBFOAndDFO}, we show a multitype tree together with its breadth-first order and depth-first order.
For each $j\in [d]$ and $m=0,\dotsm , q$, we write
\begin{equation*}
X^j_{\tr^i_\ell} (m): = \sum_{k=1}^m (\chi^{j}(v^i_k) - 1_{[i=j]}).
\end{equation*}
Observe that for each $j\neq i$, the term $X^{j}_{\tr^i_\ell}$ is non-decreasing and $X^i_{\tr^i_\ell}$ is the \L ukasiewicz path of the rooted planar (single-type) tree $\tr^i_\ell$. 
See Figure \ref{figTreeTEmptysetAndCoding} for an example of the previous coding.

\begin{figure}
\centering
     \includegraphics[scale=0.8]{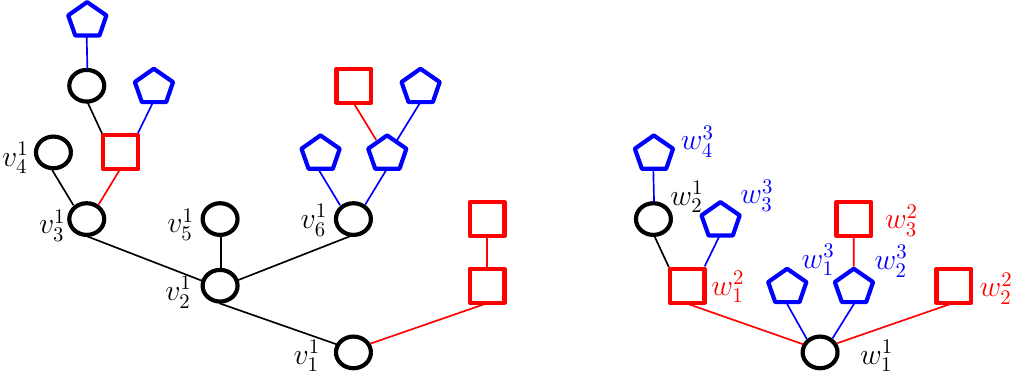}
\caption{We show a multitype tree with $d = 3$, where type one is represented as black circles, type two as red squares, and type three as blue pentagons. On the left, the tree together with the depth-first order of the subtree $\tr^1_1$ (the first subtree of type one) is depicted. On the right, we show the corresponding reduced tree together with its breadth-first order.}
\label{figTreeWithBFOAndDFO}
\end{figure}

\begin{figure}
\centering
     \includegraphics[scale=0.6]{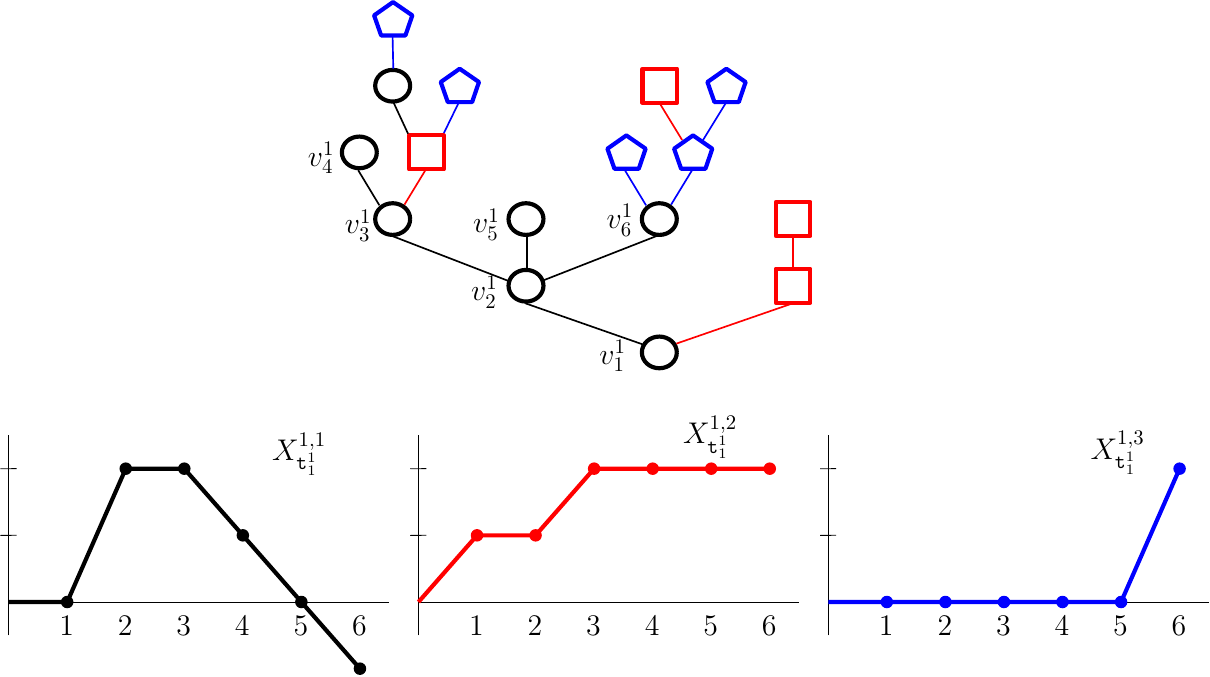}
\caption{We depict the paths $\big(X^{1}_{\tr^1_1},X^{2}_{\tr^1_1},X^{3}_{\tr^1_1}\big)$ of the subtree $\tr^1_1$.}
\label{figTreeTEmptysetAndCoding}
\end{figure}

\begin{remark}
It is well-known that using height processes $H_\tr$, which rely on the depth-first order, provides a robust way of establishing limit theorems for the metric space structure of a (single-type) BGW tree. See Chapter 2 of \cite{MR1954248} for some general results, as well as \cite{MR2203728} for a gentle introduction. This may lead the reader to wonder why we even involve any breadth-first order in our construction. 
The reason why we do this is because the construction of the (discrete) metric space $\sG^*(\cD)$ involves identifying points in the metric space $D_\bp$ of the decoration with the root in the metric space $D_{\bp \ell}$. That is, the glued metric space is constructed by gluing metric spaces $D_\bp$ indexed, layer-by-layer, by the Ulam-Harris tree $\bbU$. This layer-by-layer indexing is essentially what the breadth-first ordering is doing on the reduced tree. 
\end{remark}

Let $\vec X^i$ be the concatenation of the $\Z^d$-valued walks $(X^1_{\tr^i_\ell},\dotsm,X^d_{\tr^i_\ell})$ (recall Remark \ref{rmk:introRemark1}). It is shown in \cite{MR3449255} that the map that sends a finite rooted planar multitype tree $\tr$ to the collection of walks $(\vec X^i;i\in [d])$ is bijective. This bijection is essentially re-established in Section \ref{sec:discreteDecorations}.
This bijective correspondence should not be surprising because one can construct for fixed $\ell$, the $\Z^d$-valued walk $ (X^1_{\tr^i_\ell},\dotsm,X^d_{\tr^i_\ell})$ encoding $\tr^i_\ell$ by considering the excursion intervals of $X^{i,i}$. Indeed, let $\tau^i(k) = \min\{m: X^{i,i}(m) = -k\}$ for any $k\in \N_0$, then for $m=0,\dotsm, \tau^i(\ell)-\tau^i(\ell-1)$ we have
\begin{equation*}
X^{j}_{\tr^{i}_\ell} (m) = X^{i,j}(m+ \tau^i(\ell-1)) - X^{i,j}(\tau^i(\ell-1)).
\end{equation*}Such formulations will appear regularly in the sequel.

\subsubsection{Multitype  Bienaym\'e-Galton-Watson trees}\label{subsectionMBGWIntro}

\begin{definition}[Bienaym\'e-Galton-Watson Trees]\label{definitionBGWtree}
Let $\varphi$ and $\vec \varphi^1,\dotms,\vec \varphi^d$ be probability measures on $\{0,1,\dotms\}$ and $\{0,1,\dotms\}^d$, respectively. Define $\vec \varphi: = (\vec \varphi^1,\dotms,\vec \varphi^d)$. 
We say that a rooted planar tree $T$ is a BGW tree with offspring distribution $\varphi$ if each vertex $v$ gives birth to a random number of children with common law $\varphi$, independently of the vertices up to its current generation, i.e. $\chi(v) \overset{d}{\sim}\varphi$, and if $v$ has height ${\rm hgt}(v)$, then $\chi(v)$ is independent of $(\chi(w);{\rm hgt}(w)\leq {\rm hgt}(v), w\neq v)$. We say that a random multitype tree $T$ is a MBGW tree with offspring distribution $\vec \varphi$ (or simply a $\vec \varphi$ MBGW tree), started from a type $i_0$ individual, if the root is type $i_0$ and each vertex $v$ type $i$, has child vector $\vec \chi(v)\overset{d}{\sim} \vec \varphi^i$ independently of all other vertices up to its current generation. 
\end{definition}

In order for the trees above to be finite a.s. we need to impose a (sub)criticality assumption which we now explain. Let $\vec{\xi}^i:=(\xi^{i,1},\dotsm,\xi^{i,d})\overset{d}{\sim} \vec \varphi^i$ and write $M\in \R^{d\times d}$ the square matrix $M: = (\E[{\xi^{i,j}}]; i,j\in[d])$, hence, we assume each mean exists. We say that $M$ is \emph{irreducible} if for every $i,j\in [d]$, there is some $n$ such that $M^n_{i,j}>0$. The following is a standard result in branching processes.
\begin{prop}[{\cite[Chapter V]{AN.72}}] Suppose that $M$ is irreducible and $\PR\big(\sum_{j=1}^d \xi^{i,j} = 1\big) < 1$ for some $i\in[d]$. Then the corresponding MBGW tree is a.s. finite if and only if the Perron-Frobenius eigenvalue $\varrho$ of $M$ satisfies $\varrho\le 1$.     
\end{prop}
We will call the corresponding offspring distribution $\vec \varphi$ subcritical if $\varrho<1$, critical if $\varrho = 1$ and supercritical if $\varrho>1$. A fundamental assumption that we will make is: 
\begin{assumption} \label{ass:subcrit}
$\vec \varphi $ is (sub)critical.
\end{assumption}

A fundamental property of both single-type and multitype BGW trees is that all subtrees are encoded by individual excursions of the processes $X$ and $\vec{X}^i$ for $i \in [d]$, respectively; furthermore, these excursions are i.i.d.
It is also not hard to see that the resulting walks $ \vec X^i = \big((X^{i,1}(k),\dotsm,X^{i,d}(k)); k\ge 0\big)$, can be expressed as the $\Z^d$-valued random walks
\begin{equation}\label{eqnRandomWalks}
\vec X^{i}(m): = \sum_{k=1}^m \left(\vec{\xi}^{i}_k-\vec{\epsilon}_i\right),\qquad (\vec{\xi}^i_k;k\ge 1)\overset{i.i.d.}\sim\vec \varphi^i,
\end{equation} for $m\geq 0$.
As above, let $\tau^{i}(m) = \min\{k: X^{i,i}(k) = -m\}$. We define the $i^\text{th}$ \textit{excursion measure} of $\vec X^{i}$, denoted by $\vec N^{i,\vec \varphi}: = \vec N^{\vec \varphi^i}$, as
\begin{equation}\label{eqnCodingSubtreesOfMBGWWitRandomWalk}
\vec N^{i,\vec \varphi} := \textup{Law} \left(\big(\vec X^i(m); m = 0,1,\dotsm,\tau^i(1)\big) \right).
\end{equation}

\subsection{Representing Multitype Trees as Glued Decorations}\label{subsectionIntroMBGWTreesAsGluedDecorations}

Hopefully the picture is becoming clear on how we establish convergence of the MBGW branching tree. 
Recall the definition of a glued decoration, given in \eqref{definitionDecorationGluedMetricspaceGluedDecoration}. 
Building upon the representation of multitype trees outlined in the previous section, we now develop an algorithmic construction of these structures as glued decorations.
Consequently, by applying the convergence criteria for glued decorations established in Theorem \ref{thm:graphConv1}, we derive the scaling limit for MBGW trees conditioned on their size.

In Section \ref{sec:discreteDecorations} we construct a decoration $\cD^{i_0,\vec \varphi}$ which recreates a MBGW tree $T$ starting from a single type $i_0$ individual using the excursion measures $(\vec N^{i,\vec \varphi})_{i\in [d]}$ and the connection between reduced subtrees and the walks. We briefly describe its construction here. We do this in two steps.

We first construct the (law) of a single value of the subdecoration $\cD^{i_0,\vec \varphi}(\bp)$, for an arbitrary $\bp\in \bbU$ and an initial type $i_0\in [d]$ of the root. Recall that this is an element of $\mathfrak{M}^{\ast d}_{\tiny \mbox{mp}}$, meaning that we will construct an element $\cX = (\cX,\rho,d_g, \vec{\mu}, (x_\ell)_{\ell\ge 1})$ where $(\cX,d_g)$ is a compact metric space with the graph distance $d_g$, $\rho\in \cX$ is a root, $\vec{\mu}$ is a vector-valued measure on $\cX$ with values in $\R_+^d$ for every Borel subset of $\cX$ (in particular for the $j^\text{th}$ coordinate measure $\langle \vec{\mu},\vec{\epsilon}_j\rangle$, it holds that $\langle \vec{\mu},\vec{\epsilon}_j\rangle(\cX)\in[0,\infty)$) and $(x_\ell)_{\ell\ge 1}$ is a collection of distinguished points in $\cX$. To specify the points, we will need a random permutation so we write $\mathfrak{S}_n$ as the symmetric group on $n$ letters.
\begin{construction}[Single subdecoration of type $i$ under $\vec N^{i,\vec \varphi}$] \label{construction:Discretesingletype}Generate the $\Z^{d}$-valued random walk $\vec X^{i}$ according to $\vec N^{i,\vec \varphi}$. Let $\tr$ be the (random and a.s. finite single-type) rooted planar tree with \L ukasiewicz path $X^{i,i}$, and denote by $v_1$ its root. Let $\widetilde{\tr}$ denote the {rooted} planar tree constructed by adding a single additional vertex $v_0$ to the tree $\tr$ and joining $v_0$ to $v_1$, defining $v_0$ as the new root. Let $v_1,v_2,\dotsm, v_{\tau^{i}(1)}$ denote the depth-first labeling of the vertices of $\tr$.

\setlength{\leftmargini}{.15in}

\begin{itemize}
   \item Let $\cX = \widetilde{\tr}$, $d_g$ be the graph distance on $\widetilde\tr$ and root the space $\cX$ at $\widetilde\rho := v_0$.
   \item Define the measure $\vec{\mu}^{(i)}$ on subsets of $\cX$ by
   \begin{equation*}
       \langle \vec{\mu}^{(i)},\vec{\epsilon}_j \rangle (A) := \#\{v_m\in A: m\ge 1\}1_{[j=i]}
   \end{equation*}{for any $j\in [d]$,} so that the $i^\text{th}$ coordinate of the measure is the counting measure on the type $i$ vertices on $\tr\subset \cX$ and the remaining coordinates are null.
   \item Construct the special points $(x_\ell)_{\ell\ge 1}$ as follows. Fix $j\neq i$, and let $\vec L_j$ be the tuple which contains the vertex $v_m$, for $m\geq1$, exactly $X^{i,j}(m)-X^{i,j}(m-1)$ many times. That is $v_m$ is contained with multiplicity given by the number of type $j$ children that the vertex would have in a MBGW tree:
   \begin{equation*}
       \vec L_j: = (a_1^j,\dotsm, a_{\# \vec L_j}^j): = \big( \underset{\substack{ X^{i,j}(1)-X^{i,j}(0)\\\textup{many times}}}{\underbrace{v_1,v_1,\dotsm, v_1}},\dotsm,\underset{\substack{ X^{i,j}(\tau^{i}(1))-X^{i,j}(\tau^{i}(1)-1)\\\textup{many times}}}{\underbrace{v_{\tau^{i}(1)},v_{\tau^{i}(1)},\dotsm, v_{\tau^{i}(1)}}} \big).
   \end{equation*}Given $\#\vec L_j = X^{i,j}(\tau^{i}(1)) = n$, let $\pi^j\in \mathfrak{S}_{n}$ be a uniform permutation of $\#\vec L_j$ many letters, conditionally independent of $\vec X^i$.
   Now set for any $\ell\in \mathbb{N}$
   \begin{equation}\label{eqnDefinitionOfInfinnitelyMarkedPointInSubdecoration}
       x_{\ell}: = \begin{cases}
           a^j_{\pi^j(m+1)} & \quad\ell = j + md,\  \mbox{for some }\ j \neq i, \ m=0,\dotsm, \#\vec L_j-1\\
           \widetilde\rho&\quad \textup{else.}
       \end{cases}
   \end{equation}
\end{itemize} 
The resulting infinitely marked metric measure space $(\cX,\widetilde\rho,d_g,\vec{\mu}^{(i)}, (x_\ell)_\ell)$ will be called a \textit{subdecoration of type $i$}, and we write $\cX\sim \vec N^{i,\vec \varphi}$ to denote its distribution (with a slight abuse of notation, we use the same notation as for the excursion measure).
\end{construction}
Let us make some remarks about the construction above. 
\begin{remark}
{If we translate the subdecoration to a subtree in a MBGW tree, the points $x_1,x_{1+d},x_{1+2d},x_{1+3d},\dotsm$ will be a uniform permutation of all the vertices having children of type $1$, counted with multiplicity. They will keep track of where to paste all type 1 subtrees. The points $x_2,x_{2+d},x_{2+2d},x_{2+3d},\dotsm$ will be a uniform permutation of all the vertices having children of type $2$, counted with multiplicity, keeping track of where to paste all type 2 subtrees; and so on.  }
\end{remark}
\begin{remark}
In a MBGW tree, the type $j$ children of the vertices $v$ in a type $i$ subtree constructed from $\vec X^{i}$, will correspond to points $x_\ell$ in {the} subdecoration of type $i$, where $\ell \equiv j\mod d$. This allows us to keep track of what type a particular infinitely marked metric measure space is.
\end{remark}
\begin{remark}
We view the root $\widetilde \rho$ of $\cX$ as a cemetery point for the infinite marks. Hence, whenever $x_\ell=\widetilde \rho$, 
we will only glue the trivial single-point space $\bzer := (\{\widetilde \rho\}, \widetilde \rho, d_{\widetilde \rho}, \vec{0}, (\widetilde \rho)_{\ell\ge 1})$, where $\vec 0$ is the null measure.
\end{remark}
\begin{remark}
The presence of the uniform permutation is so that we can order the type $i$ subtrees that we glue to a particular subtree by their size as opposed to any ordering relating to planarity. This is because when taking limits, we need to be able to describe the limit of each subdecoration in the sequence of glued decorations (Assumption \ref{ass:1thm1} in Theorem \ref{thm:graphConv1}). However, the ``left-most'' tree is equal in law to the first excursion of a random walk, and if the random walk converges to a L\'evy process with infinite variation, then the first excursion converges to the zero function. The natural way to fix this issue is to order the corresponding subtrees by their size.
Similar analysis also appears in the work on Markovian growth-fragmentations \cite{Bertoin.17} (the top of page 1085 therein) and for studying \textit{heavy subtrees} of BGW trees conditioned to have size $n$ is studied in \cite{MR3916322}.
\end{remark}

With these remarks in mind, we can construct the full decoration, using an excursion measure which we denote by $\vec M^{i_0,\vec \varphi}$. It is convenient to introduce a quantity beforehand. 
Suppose that $D_\bp \subset \bbU\cup\{\widetilde\rho\}$ is a subdecoration of type $i$, with $\widetilde\rho$ some distinguished point not contained in $\bbU$. Denote by $\vec{\#}^{(i)}D_\bp \in\{0,1,\dotsm, \infty\}^d$ the vector with $j^{\text{th}}$ entry, for $j\neq i$, given by
\begin{equation}\label{defVecCardinalityJOfASubdecoration}
\begin{split}
\vec{\#}^{(i)}_jD_\bp  &:= \#\{\ell: \ell\equiv j \mod d, \mbox{ where } x_\ell\neq \widetilde\rho\}
= \#\{m: x_{j+(m-1)d}\neq\widetilde \rho\},
\end{split}
\end{equation}and 
\begin{equation}
\begin{split}
	\vec{\#}^{(i)}_iD_\bp &:= \#D_\bp.
\end{split}
\end{equation}Observe that for each $i\in [d]$, the subdecoration $\cX\sim \vec N^{i,\vec \varphi}$ of type $i$ almost surely satisfies $\vec{\#}^{(i)}\cX\in \{0,1,\dotsm\}^d$ where $\widetilde \rho$ is the point $v_0$ in Construction \ref{construction:Discretesingletype} above.
Also, given $j\neq i$, observe that $\vec{\#}^{(i)}_j\cX$ has the same meaning as $X^{i,j}(\tau^{i}(1))$ but without the need to describe the random walk $\vec X^{i}$.
\begin{construction}[Decoration $\cD^{i_0,\vec \varphi}$] \label{construction:DiscreteDecoration}
Construct $\cD = \cD^{i_0,\vec \varphi}:\bbU\to \mathfrak{M}^{\ast d}_{\tiny \mbox{mp}}$ inductively as follows:
\setlength{\leftmargini}{.15in}
\begin{itemize}
\item For $\bp = \emptyset \in \bbU$, generate
$\cD(\emptyset) \sim \vec N^{i_0,\vec \varphi}.$ 
\item For $h\geq 0$, conditionally given $(\cD(\bp); |\bp|\le h)$, independently for each $\bp\in \bbU$ such that $\bp = (p_1p_2p_3\dotsm p_h)$, i.e. $|\bp| = h$, assuming $\cD(\bp)$ is a type $i$ subdecoration, proceed as follows:
\begin{itemize}
    \item For each $j\in [d]$ with $j\neq i$, let $(\cX^{ j}_{\bp,m}; m =1,2,\dotsm, \vec\#^{(i)}_j D_{\bp} )$ denote i.i.d. samples of subdecorations of type $j$ with law $\vec N^{j,\vec \varphi}$ and let $(\cX^{j}_{\bp,(m)}; m =1,2,\dotsm, \vec\#^{(i)}_j D_{\bp} )$ denote their rearrangement in decreasing order of cardinality (i.e. $\#\cX^{j}_{\bp,(1)}\geq \#\cX^{j}_{\bp,(2)}\geq \dotsm $) with ties broken arbitrarily (and measurably). Extend this list by setting $\cX^{j}_{\bp,(m)} = \bzer$ for all $m> \vec{\#}^{(i)}_j D_
    {\bp}$. 
    \item For any $\ell\in \mathbb{N}$, choose $j\in [d]$ with $j\neq i$, and  $m\in \mathbb{N}$ such that $\ell = j + (m-1)d$, and set
    \begin{equation}\label{defSubdecorationsChildrenOfASubdecorationLabeledInDecOrder}
        \cD(\bp \ell): =   \cX^{ j}_{\bp,(m)}.
    \end{equation}For $\ell \equiv i \pmod d$, we set $\cD(\bp \ell) := \bzer$.
\end{itemize}
\end{itemize}
\end{construction}

Having constructed the decoration, now we explicitly construct its glued decoration (recall Definition \ref{definitionDecorationGluedMetricspaceGluedDecoration}). 
We use Construction \ref{construction:DiscreteDecoration}, together with the construction of the points $(x_\ell)_\ell$ in \eqref{eqnDefinitionOfInfinnitelyMarkedPointInSubdecoration} and the identification of the root and marked point $x_{\bp;\ell}\sim \widetilde \rho_{\bp\ell}$ as in \eqref{defIdentifyingInfinitelyMarkedPoinsWithRootsOfTheSubsequenceDecorations}. 
Given $\bp\in \bbU$, a subdecoration $\cD(\bp)$ type $i$, and given $\vec \#^{(i)}_j \cD(\bp)$ for $j\neq i$, construct independent subdecorations $(\cX^{j}_{\bp,(m)})_{m\in [\vec \#^{(i)}_j \cD(\bp)]} $ type $j$ ordered in decreasing order of cardinality.
Finally, for any $j\neq i$, glue all subdecorations $(\cX^{j}_{\bp,(m)})_{m\in [\vec \#^{(i)}_j \cD(\bp)]}$ to $\cD(\bp)$ by identifying each marked point $x_{j+md}\in \cD(\bp)$ with the root of the $(m+1)^{\text{th}}$ subdecoration $\cX^{j}_{\bp,(m)}$ type $j$.
Following the notation in Definition \ref{definitionDecorationGluedMetricspaceGluedDecoration}, we denote by $\sG(\cD^{i_0,\vec \varphi})$ the glued decoration. 
Now we define its excursion measure. 

\begin{definition}[Excursion measure of the glued decoration $\sG(\cD^{i_0,\vec \varphi})$]\label{remarkConstructionOfGluedDecoration}
Let $\vec{\varphi}$ be the offspring distribution of some MBGW tree and let $\cD^{i_0,\vec{\varphi}}$ be the decoration constructed above. We denote the \textit{excursion measure of the glued decoration} $\sG(\cD^{i_0,\vec{\varphi}})$ as the measure $M^{i_0,\vec{\varphi}}$.
\end{definition}

In Section \ref{sec:discreteDecorations} we prove the following:
\begin{lem}[Multitype Bienaym\'e-Galton-Watson trees as Glued Decorations] \label{lem:decoration} Suppose Assumption \ref{ass:subcrit} holds.
Let $T$ be a MBGW tree started from a single type $i_0$ individual, having offspring distribution $\vec \varphi$. Join a single vertex $\widetilde\rho$ to $T$ which is connected to just the root vertex $\emptyset\in T$, call the resulting space $\widetilde{T}$. View $\widetilde{T}$ as a metric measure space with the graph distance, rooted at $\rho$, and equipped with the vector-valued measure $ \langle\vec{\mu},\vec{\epsilon}_j\rangle (A) = \#\{v\in A: \TYPE(v) = j; v\neq \widetilde \rho\}.$ 
Then as rooted metric measure spaces 
\begin{equation}\label{eqnGluedDecorationSameLawAsMBGWTree}
    \sG(\cD^{i_0,\vec \varphi})\overset{d}{=} \widetilde{T}.
\end{equation}
\end{lem}

The above lemma does not immediately allow us to apply Theorem \ref{thm:graphConv1} to the conditioned MBGW trees because, as already discussed, the first excursion coding $\cD(\emptyset)$ will converge to the $0$ function. To overcome this, we condition so as to guarantee there are enough individuals at the ``base'' of the tree to obtain a proper convergence. More explicitly, in order to prove \ref{ass:1thm1} holds for $\bp=\emptyset$, we assume that $\#\cD(\emptyset)\geq r_n$ and let $r_n\to \infty$.  After proving the lemma above, we show in Subsection \ref{subsectionGluedDecorationConditionedWithALargeFirstSubtree} that the same result holds but conditioning with the size of the first subdecoration $\cD(\emptyset)$ on the left-hand side of \eqref{eqnGluedDecorationSameLawAsMBGWTree}, and the subtree of type $i_0$ connected with the root on the right-hand side of \eqref{eqnGluedDecorationSameLawAsMBGWTree}.

\subsection{Multitype L\'{e}vy trees and the GHP Convergence}\label{subsectionMLevyTrees}
In order to describe the scaling limits, we need to first introduce the limiting stochastic processes that we consider.

\subsubsection{L\'{e}vy processes}
For each $i\in[d]$, consider a L\'{e}vy process $\vec \bX^i = (\vec \bX^i(t);t\ge 0)$ in $\R^d$ with coordinate functions $\vec \bX^i=(X^{i,1},\dotsm, X^{i,d})$ such that $X^{i,i}$ is spectrally positive (i.e. possesses no negative jumps) and for each $j\neq i$, the process $X^{i,j}$ is a subordinator (i.e. has non-decreasing paths). We denote by $\Psi_i$ its Laplace exponent defined by (see \cite{Bertoin.96})
\begin{equation*}
\E\left[\exp(-\langle \vec{\lambda} , \vec \bX^i(t) \rangle \right] = \exp\left( t\Psi_i( \vec{\lambda})\right), \qquad t\ge0,\quad \vec{\lambda} = (\lambda_1,\dotms, \lambda_d)\in \R^d_+.
\end{equation*}
In our setting, the exponents $\Psi_i$ take the form \cite{MR3689968,MR4193902}
\begin{equation*}
\Psi_i(\vec \lambda) := -\sum_{j=1}^d a_{i,j} \lambda_j + q_i \lambda_i^2 - \int\limits_{\R_+^d} \left(1-e^{-\langle \vec{\lambda}, \vec{x}\rangle} - \lambda_ix_i	 1_{[\|x_i\|\le 1]} \right)\pi_i(d\vec{x}),
\end{equation*}where for each $i\neq j$ it holds that $a_{i,j}\ge 0$, and for each $i\in [d]$ it holds that $a_{i,i}\le 0$, $q_{i}\ge 0$, and $\pi_i$ is a Radon measure on $\R_+^d$ such that $\pi_i(\{0\}) = 0$ and
\begin{equation}\label{eqn:ass1Laplace}
\int\limits_{\R_+^d} \Big[(1\wedge \|\vec{x}\|^2) + \sum_{j\neq i} (1\wedge x_j) \Big] \pi_i(d\vec{x})<\infty.
\end{equation}
It is also convenient to consider the Laplace transform of the coordinates $X^{i,i}$, which we define by
\begin{equation*}
\E[e^{-\lambda X^{i,i}(t)}] = \exp(t\psi_i(\lambda))
\qquad\textup{with}\qquad 
\psi_i(\lambda) = -a_{i,i}\lambda +q_i \lambda^2 + \int_{0}^\infty (e^{-\lambda r}-1+1_{r<1}\lambda r) \widetilde{\pi}_i(dr),
\end{equation*}
where we impose the condition that for all $i\in[d]$
\begin{equation}   \label{eqn:ass2Laplace}
\int_0^\infty (r\wedge r^2)\widetilde{\pi}_i(dr)<\infty \textup{ and }  \textup{ either } q_i>0\quad  \text{or}\quad \int_0^1 r\widetilde{\pi}_i(dr) = +\infty.
\end{equation}
The last condition implies that $X^{i,i}$ has infinite variation.

\subsubsection{Height Process}

It is well known that under the assumption 
\begin{equation}\label{eqn:Greys1type}
\int_1^\infty \frac{1}{\psi_i(\lambda)}\,d\lambda<\infty
\end{equation}
there exists a continuous process $H^i = (H^i(t);t\ge 0)$ which is a modification (in the sense of \cite[Lemma 1.2.2]{MR1954248}) of the limit in probability
\begin{equation*}
H^{i,\circ}(t): = \liminf_{k\to\infty} \int_0^t 1\{X^{i,i}(s) < \delta_k+I^i(s,t)\}\,ds,\qquad I^i(s,t): = \inf_{s\le r\le t}\{ X^{i,i}(r)\}
\end{equation*}
for $\delta_k\downarrow 0$.
When $q_i>0$, the process $H^i$ can be constructed as $H^i(t) := q_i^{-1} \Leb\left\{ I^i(s,t): s\le t\right\}.$ See also Theorem 1.4.3 and equations (14) and (15) in \cite{MR1954248}. The height process $H^i$ codes the genealogy of a continuous state branching process with branching mechanism $\psi_i$ \cite{MR1954248}.

Let us enumerate the excursion intervals of $H^i$ (and $X^{i,i}-I^i$, for $I^i(t) = I^i(0,t)$) as
\begin{equation*}
\{t: H^i(t)>0\} = \{t: X^{i,i}(t)>I^i(t)\} = \bigcup_{\ell\in \mathbb{Q}_+} (g^i_\ell,d^i_\ell),
\end{equation*}were $(g^i_\ell,d^i_\ell)$ is the interval containing $\ell$. 
We define $\zeta_\ell^i: = d_\ell^i-g_\ell^i$ as the length of the excursion interval containing $\ell$. 
An important property is that the process $H^i$ on the interval $(g_\ell^i,d_\ell^i)$ depends only on the process $X^{i,i}-I^i$ on $(g_\ell^i,d_\ell^i)$  \cite[Chapter 1]{MR1954248}. Let us write \begin{align*}
H_{\ell}^i(t) &= H^i(g_\ell^i+t\wedge \zeta_\ell^i), &\vec \bX^i_\ell(t) &= \vec \bX^i(g_\ell^i+t\wedge \zeta_\ell^i) - \vec \bX^i(g_\ell^i)
\end{align*}We define the point process on $\R_+\times \D(\R_+,\R^d_+)\times C(\R_+,\R_+)$ by
\begin{equation*}\label{eqnPointProcessOfExcursionsHavingExcursionMEasure}
\vec{ \mathcal{N}}^i(dt,d\vec \be,dh) = \sum_{\ell*} \delta_{(-I^i(g_\ell), \vec \bX^i_\ell, H^i_\ell)}(dt,d\vec \be,dh),
\end{equation*}where the sum ranges over all $\ell\in \mathbb{Q}_+$ that count each excursion once.
The point process $\vec {\mathcal{N}^i}$ is a Poisson random measure with intensity $\Leb\otimes \vec \bN^i$, and we call $\vec \bN^i$ the \textit{excursion measure} of $(\vec \bX^i,H^i)$. 
We will frequently write $(\vec \be,h)$ as a generic excursion, and we will write $\zeta = \zeta(h) = \sup\{t: h(t)>0\}$ as the duration of this excursion.

\subsubsection{Single-type (marked) L\'{e}vy trees}

We recall the definition of a real tree coded by a continuous excursion. Let $f:[0,\zeta]\to \R_+$ denote an arbitrary continuous function with $f(0) = f(\zeta) = 0$, and consider the (pseudo-)distance $d_f$ by
\begin{equation*}
d_f(s,t) = f(s)+f(t) - 2 \inf_{ s\wedge t\le r\le s\vee t} f(r).
\end{equation*} We note that for most of the limiting functions $f$ that we consider, $f(t)>0$ for all $t\in(0,\zeta)$; however, the discrete contour processes can take the value $0$ for some $t\in(0,\zeta)$ with strictly positive probability, but for the sake of readability, we do not explicitly distinguish that case here.
The function $d_f$ satisfies the triangle inequality, is symmetric, but fails to distinguish points. We can still view $f$ as the contour process of the real tree $\T_f$ defined by $\T_f = [0,\zeta]/\sim,$
where $\sim$ is the smallest equivalence relation such that $s\sim t$ if $d_f(s,t) = 0$. One can see \cite{MR2203728}, that $(\T_f,d_f)$ is a metric space. We root the tree $\T_f$ at the element $\rho = [0]_\sim$, i.e. at the equivalence class of $0$.

For an excursion $h$ of $H^i$, we can construct a continuum random tree (a L\'{e}vy tree) $\T_{h} = (\T_{h}, \rho, d_h,\vec{\mu})$. Here $d_h$ is defined analogously as the (pseudo-)distance $d_f$ above, but using the function $h$. The vector-valued measure $\vec{\mu}$ has coordinate measures defined as follows. Let $p_{h}:[0,\zeta]\to \T_{h}$ be the canonical quotient map and define
\begin{equation}\label{eqnMeasureMuOnLevyTree}
\langle \vec{\mu},\vec{\epsilon}_j\rangle = \begin{cases}
    (p_h)_{\#}\Leb|_{[0,\zeta]} &\quad j=i\\
    0 & \quad \textup{else}
\end{cases}, 
\end{equation} where recall that $\vec{\epsilon}_j$ is the canonical vector with $j^{\text{th}}$ entry one and the others zero; here we write $g_\#\nu = \nu\circ g^{-1}$ for the \emph{pushforward} of a measure $\nu$ under a (measurable) map $g$; and $\nu|_{[0,\zeta]}$ is the restriction to $[0,\zeta]$.

We will also need to construct the infinitely marked single-type L\'{e}vy tree analogous to Construction \ref{construction:Discretesingletype}. Recall that if $f:[0,\zeta]\mapsto \R_+$ is a non-decreasing c\`adl\`ag function with $f(0) = 0$ and $f(\zeta)>0$, then the Stieltjes measure $\frac{1}{f(\zeta)}f(dx)$ defines a probability measure on $[0,\zeta]$ such that a random variable $R\sim \frac{1}{f(\zeta)}f(dx)$ satisfies
\begin{equation*}
\PR(R\le s) = \frac{f(s\wedge \zeta)}{f(\zeta)} \qquad\forall s\ge 0.
\end{equation*}
\begin{construction}[Infinitely marked single-type L\'{e}vy tree under $\vec \bN^i$] \label{construction:continuum1} Generate an excursion $(\vec \be,h)$ sampled from the excursion measure $\vec \bN^i$ and let $\zeta$ denote its duration. Conditionally given $(\vec \be,h)$, generate independent random times $R_{\ell}\in[0,\zeta]$ {for $\ell=j+md\geq 1$} with laws
\begin{equation}\label{eqnDefinitionR_jLevyTree}
\PR(R_{j+md}\le s|\vec \be) = \begin{cases}
\displaystyle \frac{e^j(s\wedge \zeta)}{e^j(\zeta)} &\quad  e^j(\zeta)>0\\
1&\quad  \textup{else}
\end{cases}\qquad\textup{ for all }s\in[0,\zeta],
\end{equation} where $\vec \be = (e^1,\dotms, e^d)$. 
In particular, $R_{i+md}  = 0$ for all $m$. 
We let $\rho = p_h(0)$, for $p_h$ the canonical quotient map as above. Define the special marked points $x_{\ell} = p_h(R_{\ell})$ for all $\ell\ge 1$. We let $\vec{\mu}$ denote the measure constructed in \eqref{eqnMeasureMuOnLevyTree}. Finally, we set
\begin{equation}\label{eqn:infinitelyPointedLevyTree}
\T_{h} = \left(\T_{h}, \rho, d_h, \vec{\mu},(x_{\ell})_{\ell\ge 1}\right).
\end{equation} We call this the marked L\'{e}vy tree, and we will denote it by $\T^*_{h}$ if there is possibility of confusion with the unmarked tree. 
\end{construction}

\begin{remark}
In \eqref{eqnDefinitionR_jLevyTree} above, we only need to condition on the excursion $\vec \be$ instead of the pair $(\vec \be,h)$ since $h$ can be constructed given $\vec \be$. Also, the tree $\T_h\in \mathfrak{M}^{\ast d}_{\tiny \mbox{mp}}$ depends on $e^j$ for $j\neq i$, which can not be recovered directly from $h$ nor $e^i$. 
\end{remark}
\begin{remark}
At the end of the proof of Lemma \ref{lem:inductionProof}, we show how the marks constructed from a MBGW tree converge to the ones presented in \eqref{eqnDefinitionR_jLevyTree}, thus, justifying the choice of such a law. 
\end{remark}

\subsection{The continuum decorations}\label{subsectionContinuumDecoration}

We now describe how to construct a multitype L\'{e}vy tree, $\sT$, rooted at a type $i_0$ individual under the excursion measures $(\vec \bN^{1},\dotsm, \vec \bN^{d})$. 
For that, we define $\vec \Psi:=( \Psi_1,\dotsm, \Psi_d)$.

\begin{construction}[Decoration  $\cD^{i_0,\vec \Psi}_r$]\label{cons:continuum2} Fix $i_0\in [d]$ and $r>0$. Let $\cD(\emptyset) = \T_{h_\emptyset}^*$ as in Construction \ref{construction:continuum1} where $(\vec \be_\emptyset,h_\emptyset)\sim \vec \bN^{i_0}( -| \zeta\ge r)$.
Inductively {over $h\geq 1$}, suppose that we have constructed $\cD(\bp)$ for all $|\bp|\le h$. Conditionally given $(\cD(\bp);|\bp|\leq h)$ and independently for any $\bp\in \mathbb{U}$ such that $|\bp| = h$, assume $\cD(\bp)$ is the tree $\T^*_{h_\bp}$ where $(\vec \be_\bp,h_\bp)$ is an excursion sampled according to the excursion measure $\vec \bN^i$, for some $i\in[d]$. 
Denote by $\zeta_\bp$ the length of $h_\bp$. 
We now construct $(\cD(\bp \ell);\ell\ge 1)$. 
Given $\vec\be_{\bp}$, for $j\neq i$ let $\vec {\mathcal{N}}^j_{\bp}$ denote an independent Poisson random measure on $[0,\zeta_\bp]\times \D(\R_+,\R_+^d)\times C(\R_+,\R_+)$ with intensity
\begin{equation*}
\E[\vec {\mathcal{N}}_\bp^j(dt,-)|\vec \be_\bp] ={e^j_\bp(dt)} \times \vec \bN^j_\bp(-) \qquad j\neq i,\quad e^j_\bp(\zeta_\bp)>0,
\end{equation*}
The Poisson random measure $\vec{\mathcal{N}}_\bp^j$ encodes all the ``type $j$'' subdecorations $D_{\bp \ell}$ where $\ell\equiv j\mod d$ and that are glued to $D_\bp$. Therefore, if $e^j_\bp(\zeta_\bp)>0$, we will enumerate the (countably infinitely many) atoms of $\vec {\mathcal{N}}^j_\bp$ as $(\vec \be_{\bp;j+md},h_{\bp;j+md})$ for $m\ge 0$ in the a.s. unique way (see Lemma \ref{lem:distinctExcursionLengths}) such that $\zeta_{\bp;j}> \zeta_{\bp;j+d}>\dotsm>0.$ If $e^j_\bp(\zeta_\bp) = 0$ we just set the atoms as $(0,\boldsymbol{0}, \boldsymbol{0})$ where $\boldsymbol{0}$ are the zero functions in the respective space (in particular this is the case when $j = i$). Doing this for all $j\in[d]$, gives us a collection of functions
\begin{align*}
(\vec{\be}_{\bp \ell}, h_{\bp \ell})_{\ell\ge 1}:=	(\vec{\be}_{\bp;\ell},h_{\bp;\ell})_{\ell\ge 1}
\end{align*}
some of which are the zero function. Recalling  Construction \ref{construction:continuum1}, we define
\begin{align*}
\cD(\bp \ell) = \T_{h_{\bp \ell}} = \left(\T_{h_{\bp \ell}}, \rho, d_{h_{\bp \ell}}, \vec{\mu}, (x_{\bp \ell; r})_{r\ge 1}\right)
\end{align*}
where the right-hand side is defined as in \eqref{eqn:infinitelyPointedLevyTree}. 

After recursively building  $\cD(\bp)$ for all $\bp \in \bbU$, we define the decoration $\cD^{i_0,\vec \Psi}_r:=(\cD(\bp);\bp \in \bbU)$.

\end{construction}
We are now ready to define the multitype L\'evy tree $\sT$.  
\begin{definition}[Continuum glued decoration $\sG(\cD^{i_0,\vec \Psi}_r)$, multitype L\'evy tree $\sT$]\label{constructionMultitypeLevyTreeGluedDecoration}
For each $r>0$, we let $\bM^{i_0,\vec \Psi}_r$ denote the law of $\sG(\cD_r^{i_0,\vec{\Psi}})$.

We say that $\bM^{i_0}_{r}:=\bM^{i_0,\vec \Psi}_r$ is the law of the \textit{multitype L\'{e}vy tree} $\sT = \sG(\cD^{i_0,\vec{\Psi}}_r)$, with {branching mechanisms} $(\Psi^j;j\in[d])$ and rooted at a type $i_0$ individual contained in a subtree of size at least $r$. 
Note that the above defines a glued decoration, formed by the Cauchy completion of the subglued decoration constructed from $\cD^{i_0,\vec \Psi}_{r}$ for all $r\in \mathbb{R}_+$, just as in Definition \ref{definitionDecorationGluedMetricspaceGluedDecoration}. 
\end{definition}

\subsection{Convergence of conditioned  Multitype Bienaym\'e-Galton-Watson trees}

We now state our main convergence results. 
Recall the definition of the law $M^{i_0,\vec \varphi_n}$ of the glued decoration in \ref{remarkConstructionOfGluedDecoration}.
Throughout this section we will consider a sequence of decorations $(\cD^{n}_{r_n};n\ge 1)$, conditioned on the first subdecoration of type $i_0\in [d]$ connected to the root, being large, that is
\begin{equation}\label{eqnDecorationsWithBigSubtreeBig}
\sG(\cD^{n}_{r_n}):=\sG(\cD^{n,i_0,\vec \varphi_n}_{r_n})\sim 	M^{i_0,\vec \varphi_n}(-| \#\cD^n(\emptyset) \ge r_n)
\end{equation}
for some (sub)critical offspring distribution $\vec \varphi_n$ and some $r_n\to\infty$ at an appropriate speed (to be specified later). Here we use a sequence of $d$-dimensional random walks $(\vec X^i_n;i\in [d])$ with step distribution $\vec \varphi_n=(\vec \varphi^1_n,\dotsm, \vec \varphi^d_n)$, as constructed in \eqref{eqnRandomWalks}. The following four assumptions will be needed.

\begin{enumerate}[label=\textbf{A.\arabic*}]
\item \label{ass:a1} There exist sequences $a_n\to\infty$, $b^j_n{\to \infty}$ with $b^j_n/a_n\to\infty$ for each $j\in[d]$, and such that for all $i\in[d]$
\begin{equation*}
    \left(\frac{a_n}{b^j_n} X^{i,j}_n (\fl{b^i_nt}) ;t\ge 0, j\in[d]\right) \weakarrow \left(\vec \bX^i(t);t\ge 0 \right),
\end{equation*}
for a L\'{e}vy process $\vec \bX^i${.}
\item \label{ass:a2} For any $i\in [d]$, the L\'{e}vy process $\vec \bX^i$ has Laplace exponent $\Psi_i$ satisfying \eqref{eqn:ass1Laplace}, and the L\'{e}vy process $X^{i,i}$ has Laplace exponent $\psi_i(\lambda) = \Psi_i(\lambda \vec{\epsilon}_i)$ satisfying \eqref{eqn:ass2Laplace}.
\item \label{ass:a3} The Laplace exponent $\psi_i$ of $X^{i,i}$ satisfies \eqref{eqn:Greys1type}.
\item \label{ass:a4} For each $i\neq j$, $X^{i,j}$ is strictly increasing and the Perron-Frobenius eigenvalue $\lambda$, of $(\E[{X^{i,j}}(1)];i,j\in[d])$, is non-positive: $\lambda\le 0$. 
\end{enumerate}

Assumption \ref{ass:a1} can be equivalently formulated in terms of the offspring distributions through the convergence of multidimensional triangular arrays (see, e.g., \cite{MR1943877}, Chapter VII.3). 

\begin{remark}
{
Requiring the functional convergence of rescaled depth-first walks (also called \L ukasiewicz paths) to L\'evy processes, as formulated in \ref{ass:a1}--\ref{ass:a4}, has become a standard framework for establishing scaling limits of branching processes. This approach contrasts with the results in \cite{MR2469338, MR3606739, MR3748121}, where the convergence of depth-first walks was not assumed. In the single-type setting, assumptions analogous to ours are employed in \cite{MR3098685, MR0362529, MR2225068, mijatović2025criticalbranchingprocessesimmigration}, while for the multitype case, the convergence of multidimensional branching processes is established within this framework in \cite{MR827331, MR3689968, clancy2021encodingmultitypegaltonwatsonforests, MR5009763}. This methodological distinction parallels the developments in stochastic blockmodels: for instance, the scaling limits in \cite{BBSW.14} are characterized by a single Brownian motion with parabolic drift, whereas those in \cite{KL.21, CKL.22} necessitate multiple independent Brownian motions or L\'evy-type processes.
}
\end{remark}

\subsubsection{Known Convergence Results}\label{sec:associatedBranchingProcesses}
It will be convenient to describe certain future assumptions in terms of branching processes, so we establish the corresponding notation here. We let $\vec Z_n = \left((Z^1_n(m),\dotms, Z^d_n(m));m = 0,1,\dotsm\right)$ denote a MBGW branching process starting from $\lfloor b^i_n/a_n\rfloor $ many individuals 
type $i$, for all $i\in[d]$.  Each vertex of type $i\in[d]$ gives birth to $\xi^{i,j}_n$ many vertices of type $j$ where $\vec{\xi}^i_n=(\xi^{i,1}_n,\dotsm, \xi^{i,d}_n)\sim \vec \varphi^i_n$. We also let $(Y^i_n(m);m=0,1,\dotms)$, $i\in[d]$, {be} a single-type BGW branching process with offspring	 distribution $\widetilde{\varphi}^{(i)}_n(k):=\vec \varphi^i_n (\Z^{i-1}_+\times\{k\}\times \Z^{d-i}_+)$ and starting from $\lfloor b_n^i/a_n\rfloor $ many individuals. Note that $\widetilde{\varphi}_n^{(i)}$ is {the} $i^\text{th}$ marginal law of $\vec \varphi^i_n$. Obviously we can couple $\vec Y_n$ and $\vec Z_n$ such that for all $m\ge 0$ and $i\in[d]$ it holds $Y^i_n(m)\le Z^i_n(m)$. 

Using the discrete Lamperti transform, which describes the generation sizes of the population (see \cite{MR3098685,MR3689968, MR3449255}), we can write $\vec Z_n$ and $\vec Y_n$ as solutions to
\begin{align*}
Z_n^j(m+1)  &= \left\lfloor \frac{b_n^j}{a_n} \right\rfloor  + \sum_{i=1}^d X^{i,j}_n \left( \sum_{\ell=0}^m Z_n^i(\ell) \right),&
Y_n^j(m+1)  &= \left\lfloor \frac{b_n^j}{a_n} \right\rfloor  + X^{j,j}_n \left( \sum_{\ell=0}^m Y_n^j(\ell) \right),
\end{align*} 
although this does not extend to the coupling discussed above.

The weak convergence of $H_n^i$ to $H^i$ has been established in \cite[Theorem 2.3.1, Corollary 2.5.1]{MR1954248}, which we now recall in the next proposition.
\begin{prop}\label{prop:dlHeight}
Suppose \ref{ass:a1}, \ref{ass:a2}, \ref{ass:a3} and
\begin{enumerate}[label=\textbf{A.5}]
    \item \label{ass:a5} For each $i\in[d]$ and $\delta>0$ it holds that
    $
     \liminf_{n} \PR(Y^i_n(\fl{a_n\delta}) = 0) > 0
    $, where $Y^i_n$ is the type $i$ branching process defined above, starting from $\fl{b_n^i/a_n}$ many individuals. 
\end{enumerate}
Then jointly for all $i\in[d]$ and jointly  with the convergence in \ref{ass:a1}, we have
\begin{equation*}
    \left(a_n^{-1} H_n^i\left(\fl{b^i_nt} \right);t\ge 0 \right) \weakarrow H^i,
\end{equation*} where $H^i_n$ is the height process constructed from $X^{i,i}_n$ in \ref{ass:a1}, and associated to $Y^i_n$.
\end{prop}

\begin{remark}
Note that the rescaling for the multitype BGW process is
\[
\Big(\frac{a_n}{b^{i}_n}Z^i_n(\fl{a_n t});t\geq 0\Big). 
\]That is, in our setting, the height of a \emph{typical} point in the tree is of the order $a_n$, whereas the number of individuals type $i$ at a typical height, is of the order $\fl{b_n^i/a_n}$. 
\end{remark}

\subsection{Main Result}\label{refMainResult}

Recall from Definition \ref{constructionMultitypeLevyTreeGluedDecoration}, the law $\bM^{i_0}_r$. 
Our main theorem states the convergence to such multitype L\'evy tree.

\begin{thm}[Convergence of conditioned MBGW trees to the conditioned multitype L\'evy tree]\label{thm:MAIN}
Suppose Assumptions \ref{ass:a1}--\ref{ass:a4} hold and also that 
\begin{enumerate}[label=\textbf{A.6}] 
\item \label{ass:A6} For every $\delta>0$, $\liminf_{n\to\infty} \PR\left(\vec Z_{n}(\fl{a_n\delta}) = \vec{0} \right) > 0$
{where $\vec{Z}_n$ is a MBGW process starting from $\fl{b^i_n/a_n}$  individuals type $i$ for each $i\in [d]$, with offspring distribution $\vec{\varphi}_n$.}
\end{enumerate}
Suppose $r>0$ is such that $\vec \bN^{i_0}(\zeta = r) = 0$, and let $r_n = \fl{b_n^{i_0} r}$.
Let $T_{n,r_n}$ be a $\vec \varphi_n$ MBGW tree started from a root $\rho_n $ with type $i_0\in [d]$, conditioned on the event that the type $i_0$ subtree containing the root has size at least $r_n$. Let $\langle \vec{\mu}_n,\vec{\epsilon}_j\rangle$ be the counting measure on the type $j$ vertices for every $j\in [d]$, and $d_n$ the graph distance on $T_{n,r_n}$.
Then, in the Gromov-Hausdorff-Prohorov topology
\begin{equation*}
\left(T_{n,r_n}, \rho_n, a_n^{-1} d_n, B_n^{-1} \vec{\mu}_n \right) \weakarrow \sT_r,
\end{equation*}
where $B_n = \diag(b_n^1,\dotsm,b_n^d)$ is a diagonal matrix and $\sT_r\sim \bM^{i_0}_r.$
\end{thm}

Taking $d=1$ in the theorem, yields the convergence of a single-type tree conditioned to have mass at least $r$, which can be deduced by standard arguments (see also Lemma \ref{lem:longExc}).

{\begin{remark}\label{remarkA6ImpliesA5}
We emphasize that although we utilize the convergence of the single-type height functions in Proposition \ref{prop:dlHeight}, our main theorem does not establish the convergence of the overall tree's height function. Doing so would yield a limit independent of the types, as done in  \cite{MR2469338,MR3748121}.
Also, note that Assumption \ref{ass:A6} is a natural generalization of Assumption \ref{ass:a5} made in the single-type case to prove convergence of the tree \cite[Theorem 2.3.1]{MR1954248}. In fact, by a coupling argument between a multitype BGW tree and the single BGW trees corresponding to the reduced subtrees, it is easy to see that \ref{ass:a5} is implied by \ref{ass:A6}.
\end{remark}
}

\begin{remark}
Our primary contribution is a novel construction that extends the scope of previous models by yielding a limit where each point retains a well-defined type. 
While this framework conceptually generalizes earlier approaches as in \cite{Miermont.08,MR3748121}, by keeping the multitype structure, it is not a direct generalization in the sense of weakening their hypotheses. Our methodology requires entirely different techniques, and our assumptions do not imply one another, but rather cover distinct regimes of offspring distributions.
\end{remark}

As the collection of compact $\bb{R}$-trees in the Gromov-Hausdorff topology form a closed set (\cite[Theorem 1]{MR2221786}), we obtain the following corollary.
\begin{cor} The limit $\sT_r$ in Theorem \ref{thm:MAIN} is a compact real tree.     
\end{cor}

\subsection{Discussion, Overview and Applications}

\subsubsection{Relation to random graph models}

Understanding the geometry of large random combinatorial structures has motivated a large amount of probabilistic research over the last several decades. A vital tool for these works are the metric space scaling limits of BGW trees 
initiated with Aldous' work in the early 90's \cite{Aldous.93,Aldous.91, Aldous.91a} on the (Brownian) continuum random tree and then the L\'{e}vy extensions of Le Gall, Le Jan, and Duquesne \cite{MR1954248, LL.98a, LL.98b}. 
Using these techniques, in particular the convergence of the height process, many random graphs models are known to have continuum tree-like metric space limits \cite{ABBG.12, ABBGM.17, BBSW.14, BDW.21, CKG.23}. However, apart from \cite{BBSW.14} the random graph works just mentioned 
are restricted to so-called \textit{rank-one} models wherein the probability that two vertices $i,j$ share an edge is $\PR(i\sim j) \approx w(i)w(j)$ for some weight function $w$. One reason why there is a restriction to rank-one models is that at criticality these random graphs are close to BGW trees, perhaps with a few additional edges included. 

The work \cite{BBSW.14} begins to describe the \textit{basin of attraction} of the (rank-one) Erd\H{o}s-R\'{e}nyi random graph and shows that the stochastic blockmodel under certain assumptions lies in this basin. The stochastic blockmodel, introduced by social scientists in \cite{HLL.83}, is a random graph where each vertex $v$ is assigned a type in $[d]$ for some fixed $d\in \mathbb{N}$, and edges are included independently with probabilities depending solely on the types of the incident vertices. Even though this model is not rank-one in general, the authors of \cite{BBSW.14} show that under certain conditions this model behaves like the rank-one Erd\H{o}s-R\'{e}nyi model. 
This is analogous to the works of Miermont \cite{Miermont.09}, Berzunza \cite{MR3748121}, and de Raph\'{e}lis \cite{deRaphelis.17} wherein the authors show that various MBGW trees converge under appropriate re-scaling and after forgetting the type information to \textit{single-type} $\alpha\in(1,2]$-stable continuum random trees. With this said, there are many applications of random graphs that lie squarely outside of the rank-one regime. One notable example related to cluster analysis is community detection, wherein one tries to determine the underlying block structure of a random graph by just examining the unlabeled graph. We refer the reader to Abbe's survey \cite{Abbe:2017} for more details and some results.

Mathematically, some things are quite well understood for general models of random graphs. For example, the emergence of a giant connected component  is established in \cite{BJR.07} under very mild and general conditions. A key part of the analysis in \cite{BJR.07} is a comparison of the inhomogenous random graph with a multitype branching process. We also refer the reader to \cite{vanderHofstad.24}. 

On the other hand, there are relatively few results about the sizes of the connected components at criticality, let alone their geometry. Recently, Konarovskyi and Limic \cite{KL.21} established a new critical regime for the stochastic blockmodel wherein they establish a scaling limit for the size of the connected components which are described via a random functional of several independent standard multiplicative coalescents \cite{Aldous.97}. Therein, the authors do not distinguish by the type of the vertex in the limit; however, the results of Konarovskyi, Limic and DC \cite{CKL.22} establish a scaling limit of the vector of component masses for the degree-corrected stochastic blockmodel introduced in \cite{Karrer:2011}. In \cite{Wang.23}, Wang establishes various results about the critical bipartite Erd\H{o}s-R\'{e}nyi random graph including the limiting size and geometry of the connected components. In \cite{clancy2024componentsizesrank2multiplicative,clancy2024nearcriticalbipartiteconfigurationmodels}, DC identified the limiting sizes of more general bipartite random graph models. Using different techniques, Haig and Wang \cite{haig2025randombipartitegraphsiid} established metric space scaling limits for one of the models studied in \cite{clancy2024componentsizesrank2multiplicative}.

The next step in describing the geometry of large critical stochastic blockmodels would be to identify the scaling limits of critical multitype BGW trees. This is with the eventual goal of constructing \emph{continuum random graphs} via a quotienting of a multitype CRT, perhaps after some Girsanov transformation. This has been the standard approach for single-type graphs since \cite{ABBG.12}.

\subsubsection{Relation to the CRT}

The CRT, besides being the limit of BGW trees conditioned to have a given size, also arises as the limit of several other classes of trees \cite{MR3050512,MR2829313}, and, as noted earlier, in the multitype setting when types are ignored \cite{MR2469338,MR3748121,addarioberry2025scalinglimitsmultitypebienayme}. More surprisingly, it also emerges as the scaling limit of various classes of graphs and maps \cite{MR2438817,MR3335010,MR3573291,MR3382675,MR3342658}. We expect that some of these models can be extended to the multitype setting adapting our framework.
Single-type limits of multitype trees are also obtained by Haas and Stephenson in \cite{HS.21} using a general framework of multitype Markov branching and multitype fragmentation trees.

The CRT also appears in the context of random walks on graphs and the uniform spanning tree on graphs. 
For example, it is the limit of the uniform spanning tree on some graphs \cite{2004math.....10430P,MR2496437}, as well as the stationary measure of some processes taking values on trees on graphs \cite{MR2221786,hernandez2025scalinglimitaldousbroderchain}. 
On this line of research, there is also a deep connection with multitype branching processes as shown in \cite{MR3876899}, where it is proved that the UST on dense graphs converges locally to a multitype branching process. 

Finally, the recent results of the first author, together with Pardo and Harris \cite{hernandez2025coalescentstructuremultitypecontinuoustime}, gives us the law of the subtree generated by a sample of $k$ individuals in a MBGW tree. 
The next step in this direction is to relate such a result, with the subtree generated by a sample of $k$ individuals in the multitype L\'evy tree, which is a standard topic in the convergence of trees \cite{MR3522292}.

\subsubsection{Recursive Constructions of CRTs}

Going back to Aldous' original construction of the Brownian CRT \cite{Aldous.91}, many authors have considered the stick-breaking construction of CRTs. In these works one starts with the positive real line $[0,\infty)$ and applies marks $ 0 = t_0<t_1<t_2<\dotsm$ according to some Poisson process(es). One then starts with the tree $\sT_1 = [0,t_1)$ and then inductively constructs $\sT_{n+1}$ from $\sT_{n}\sqcup [t_n,t_{n+1})$ by identifying $t_{n}$ with some point $s_n\in[0,t_n)\subset \sT_{n}$. Formally, $\sT_{n+1}$ is the quotient (metric) space $(\sT_n \sqcup [t_n,t_{n+1}))/ \sim$ where $t_n\sim s_n$. The precise rule for choosing $s_n$ depends on the particular model at hand which we do not detail. We refer to \cite{Aldous.91, AP.00, GH.15, CH.17} for some examples, as well as \cite{Senizergues.19} for a general framework. The continuum random tree $\sT$ is then the Cauchy completion of $\bigcup \sT_n$.
Another convergence of glued  metric spaces under a recursive construction, is given in \cite{MR4819750}. 
It is worth mentioning the recent work \cite{bertoin2025selfsimilarmarkovtreesscaling}, where the authors rescale a sequence of self-similar Markov trees and keep track of the (non-countably) types in the limit, using also decorations.
One application to the study of decorated spaces can be found in \cite{MR4780507}, where a random walk on a decorated BGW tree can be found.

A different method was introduced by Rembart and Winkel \cite{RW.18} where they consider iterative construction using ``strings of beads." We keep the construction a little vague so as to not introduce unnecessary technicalities and refer to the reader to \cite{RW.18} for precise statements. A string of beads is a triple $([0,L], (X_i;i\ge 1), (P_i; i\ge 1))$ where $L = \sup_{i: P_i>0} X_i >0$, $\{X_i: P_i>0\}$ are distinct points and $P_1\ge P_2\ge\dotsm\ge 0$ with $\sum_i P_i = 1$. They analyze a ``bead crushing'' algorithm wherein one starts with a random string of beads $\xi = ([0,L], (X_i),(P_i))$ and for each $i$ such that $P_i>0$ one takes an independent copy $\xi_i\overset{d}{=} \xi$ and takes (a disjoint copy of) the interval $[0,L_i]$ corresponding to the string $\xi_i$ and attaches the $0\in \xi_i$ to location $X_i\in [0,L]$ and scales the metric by $P_i^\beta$ for some $\beta$. (When $\beta = 0$ this scaling does nothing). One continues this process ad infinitum. The main theorems in \cite{RW.18} give sufficient conditions for this operation to be stable in the Gromov-Hausdorff topology.

Our Construction \ref{cons:continuum2} is similar to this construction; however there are several notable differences. First, we replace the interval $[0,L]$ in the string of beads with a continuum random tree distributed ``under'' an excursion measure. Second, the points $X_i$ are not distinct but instead atoms of a Poisson random measure whose intensity measure can be atomic. Finally, we do not glue i.i.d. copies of the original random tree onto our tree, but instead glue trees with a different distribution. 

\subsection{Applications}

Let us conclude this section by offering several concrete examples which our general
results cover and, in particular, some of them not currently carried by the literature.

\subsubsection{Convergence to a multitype tree with linear drift in the off-diagonal and Brownian motions in the diagonal}
We provide an example that satisfies \ref{ass:a1}--\ref{ass:a4} and \ref{ass:A6} (recall from Remark \ref{remarkA6ImpliesA5} that this is enough to show Assumption \ref{ass:a5}). Fix any matrix $A = (\alpha_{i,j})_{i,j\in[d]}$ with $\alpha_{i,i}<0$ for all $i\in[d]$ but $\alpha_{i,j}>0$ for $i\neq j$. Suppose also that $\sum_{j} \alpha_{i,j} < 0$. Let $\Lambda^{(n)} = (\lambda^{(n)}_{i,j})_{i,j\in [d]}$ be the matrix $ \Lambda^{(n)} = I_{d\times d} + n^{-1/2} A^{(n)}$ where $A^{(n)}=(\alpha^{(n)}_{i,j})_{i,j\in [d]}\in \R^{d\times d}$ is any sequence such that $A^{(n)}\to A$ as $n\to\infty$. 
Suppose also that $\sum_{j} \alpha^{(n)}_{i,j} < 0$ for every $n\in \N$ and $i\in [d]$. 
We claim that if
\begin{align*}
    \xi^{i,j,(n)}_{\ell} \sim \operatorname{Poi}(\lambda_{i,j}^{(n)})
\end{align*}are independent Poisson random variables for every $i,j\in [d]$ and $\ell,n\in \mathbb{N}$, then \ref{ass:a1}--\ref{ass:a4} and \ref{ass:A6} are satisfied with $a_n = \sqrt{n}$ and $b_n^i = n$ for every $i\in [d]$. We start with \ref{ass:a1}--\ref{ass:a4} and use a classical result of Skorohod \cite{MR94842}. Namely, it suffices to show that if $X^{i,j}_n$ is a random walk with increments $(\xi^{i,j,(n)}_{\ell}-\bo 1_{[i=j]};\ell\in\mathbb{N})$, recall \eqref{eqnRandomWalks}, then
\begin{align}
    n^{-1/2} X^{i,j}_n(n) \weakarrow \begin{cases}
        \alpha_{i,j} & i\neq j\\
        \mathcal{N}(\alpha_{i,i},1) & i=j,
    \end{cases}
\end{align}where $\mathcal{N}(\alpha_{i,i},1)$ are independent normal random variables of mean $\alpha_{i,i}$ and variance one, for every $i\in [d]$.  
Note that for $i\neq j$ we have $X_n^{i,j}(n)\overset{d}{=}\operatorname{Poi}(\lambda_{i,j}^{(n)}n)$, while for $i=j$ we have $X_n^{i,i}(n)+n \overset{d}{=}\operatorname{Poi}(\lambda_{i,i}^{(n)}n)$. Thus, for all $i,j\in [d]$ we obtain
\begin{align*}
\E[n^{-1/2}X_n^{i,j}(n)] = \alpha_{i,j}^{(n)}  \sim \alpha_{i,j}\qquad\textup{and}\qquad \Var(n^{-1/2} X_n^{i,j}(n)) = \lambda_{i,j}^{(n)} = \bo 1_{[i=j]} +o(1).
\end{align*}
By Theorem 2.7 in \cite{MR94842} and the Lindeberg-Feller central limit theorem we have, jointly over $i,j\in [d]$ that
\begin{align*}
    \left(n^{-1/2} X_n^{i,j}(n t) ; t\ge 0\right)\weakarrow X^{i,j}\textup{ where } X^{i,j}(t)= 
        \alpha_{i,j}t + \mathbf{1}_{[i=j]}B_i(t)
\end{align*} for independent standard Brownian motions $(B_i;i\in [d])$. Conditions \ref{ass:a1}--\ref{ass:a4} follow easily. 
For \ref{ass:a4}, recall that for the non-negative matrix $\mathbb{E}(X^{i,j}_n(1))_{i,j\in [d]}$, its Perron–Frobenius eigenvalue $\lambda_n$ satisfies the inequality
\[
\lambda_n\leq \max_{i\in [d]} \sum_{j\in [d]} \mathbb{E}(X^{i,j}_n(1))\leq n^{-1/2} \max_{i\in [d]} \sum_{j\in [d]} \alpha^{(n)}_{i,j}<0,
\]by the Gershgorin circle theorem. 

We now turn to \ref{ass:A6}.
For this, we begin by noting that
\begin{equation}\label{eqnExampleConvergenceToBrownianInDiagonal}
    \sum_{j=1}^d \xi^{i,j,(n)}_1 \le_{\textup{s.t.}} \operatorname{Poi}(1)\qquad\textup{for all }i\in [d],\ n\in\mathbb{N},\ 
\end{equation}
where $X\le_{\textup{s.t.}}Y$ is shorthand for $X$ is stochastically dominated by $Y$, i.e. $\PR(X>t)\le \PR(Y>t)$ for all $t\in \R$. From here, it is not too hard to couple the $d$-type BGW branching process 
\begin{align*}
    {Z}_n^j(k) = \lfloor \frac{b^{j}_n}{a_n} \rfloor + \sum_{i=1}^d X^{i,j}_n\circ C_n^i(k-1) \qquad\textup{and}\quad C_n^j(k) = \sum_{h=0}^k Z_n^j(h)
\end{align*}
with a single-type BGW branching process $Z_n^*$ starting from $\lfloor d\sqrt{n}\rfloor$ many individuals and having $\operatorname{Poi}(1)$ offspring distribution such that
\begin{align*}
    \sum_{j=1}^d Z_n^j(k) \le Z_n^*(k)\qquad\textup{ for all }k\ge 0.
\end{align*} As \cite[Theorem 2.3.2]{MR1954248} implies
\begin{align*}
    \liminf_{n} \PR(Z_n^*(\lfloor \delta \sqrt{n}\rfloor) = 0)>0
\end{align*} it easily follows that 
\begin{align*}
    \liminf_{n} \PR(\vec{Z}_n (\lfloor \delta \sqrt{n}\rfloor ) = \vec{0})>0.
\end{align*}

%
\subsubsection{Convergence to a multitype tree with linear drift in the off-diagonal and stable processes in the diagonal}\label{exampleLinearDriftPlusStables}

Fix any matrix $A = (\alpha_{i,j})_{i,j\in[d]}$ with $\alpha_{i,i}<0$ for all $i\in[d]$ but $\alpha_{i,j}>0$ for $i\neq j$. Suppose also that $\sum_{j} \alpha_{i,j} < 0$. 
Let $A^{(n)}=(\alpha^{(n)}_{i,j})_{i,j\in [d]}\in \R^{d\times d}$ be any deterministic sequence such that $A^{(n)}\to A$ as $n\to\infty$.
Suppose also that $\sum_{j} \alpha^{(n)}_{i,j} < 0$ for every $n\in \N$ and $i\in [d]$. 
We will use the parameters $a_n = n^{1-1/\alpha}$ and $b_n^i = n$ for every $i\in [d]$.

Consider a set of i.i.d. $\alpha$-stable processes $X^{i,i}=(X^{i,i}(t);t\geq 0)$ for all $i\in [d]$, and $\alpha\in (1,2)$, such that $\E(X^{i,i}(1))=0$. 
Let $W_\alpha$ be a random variable in the domain of attraction of an $\alpha$-stable law, taking values on $\{-1,0,1,2,\ldots\}$ and with mean zero. 

To achieve the stochastic domination as in \eqref{eqnExampleConvergenceToBrownianInDiagonal}, we couple the diagonal and off-diagonal terms for each row. For any $n\in\mathbb{N}$, let $P_{n} = \sum_{j \ne i} X_{n}^{i,j}(1)$, where we choose $(X_{n}^{i,j}(1))_{j \neq i}$ to be mutually independent Poisson random variables with $X_{n}^{i,j}(1) \sim \textup{Poi}(n^{1/\alpha-1}\alpha_{i,j}^{(n)})$. Thus, $P_{n} \sim \textup{Poi}(\lambda_n)$ with $\lambda_n = n^{1/\alpha - 1} \sum_{j \ne i} \alpha_{i,j}^{(n)}$. 
Because $\lambda_n \to 0$ and $\mathbb{P}(W_\alpha \ge -1) = 1$, we have $\mathbb{P}(P_{n} \ge k) \le \mathbb{P}(W_\alpha + 1 \ge k)$ for all $k \ge 0$ whenever $n$ is sufficiently large. Thus, we can couple $P_{n}$ and $W_\alpha$ such that $P_{n} \le W_\alpha + 1$ almost surely.

Under this joint coupling, we explicitly define the diagonal term as
\begin{equation}
X_{n}^{i,i}(1) = W_\alpha - P_{n} - B_n \mathbf{1}_{[W_\alpha \ge P_{n}]},
\end{equation}
where, conditionally on $W_\alpha$ and $P_n$, $B_n \sim \textup{Ber}(p_n)$ is an independent Bernoulli random variable. We choose $p_n$ to absorb the remaining required drift:
\[
\mathbb{E}\big[B_n \mathbf{1}_{[W_\alpha \ge P_{n}]}\big] = p_n \mathbb{P}(W_\alpha \ge P_{n}) := - n^{1/\alpha - 1} \sum_{j \in [d]} \alpha_{i,j}^{(n)} > 0.
\]
Because $\mathbb{P}(W_\alpha \ge P_{n}) \to \mathbb{P}(W_\alpha \ge 0) > 0$, we have $p_n \to 0$ as $n \to \infty$, making $B_n$ a well-defined Bernoulli variable with $p_n \in (0,1)$ for all large $n$.

By construction, $X_{n}^{i,i}(1)$ takes values in $\{-1, 0, 1, 2, \dots\}$. Indeed, if $W_\alpha < P_{n}$, the coupling $P_{n} \le W_\alpha + 1$ forces $P_{n} = W_\alpha + 1$, giving $X_{n}^{i,i}(1) = W_\alpha - P_{n} - 0 = -1$. If $W_\alpha \ge P_{n}$, then $X_{n}^{i,i}(1) \ge W_\alpha - P_{n} - 1 \ge -1$. 

Furthermore, the sum over all types yields:
\[
\sum_{j=1}^d X_{n}^{i,j}(1) = X_{n}^{i,i}(1) + P_{n} = W_\alpha - B_n \mathbf{1}_{[W_\alpha \ge P_{n}]} \le W_\alpha \quad \textup{a.s.}
\]
This exact pathwise inequality implies the required stochastic domination $\sum_{j=1}^d X_{n}^{i,j}(1) \le_{\textup{s.t.}} W_\alpha$, completely satisfying Assumption \ref{ass:A6}.

It remains to verify conditions \ref{ass:a1}--\ref{ass:a4}. For $j \ne i$, the variables $X_{n}^{i,j}(1)$ are independent Poisson r.v.'s, so $\mathbb{E}[n^{-1/\alpha}X_{n}^{i,j}(n)] = n^{1-1/\alpha} \cdot n^{1/\alpha-1} \alpha_{i,j}^{(n)} = \alpha_{i,j}^{(n)} \sim \alpha_{i,j}$. For their variances, after scaling the sum of $n$ steps by $n^{-1/\alpha}$, the variance vanishes as $O(n^{-1/\alpha}) \to 0$. Thus, $(n^{-1/\alpha}X_{n}^{i,j}(\lfloor nt\rfloor))_{t \ge 0} \Rightarrow (\alpha_{i,j} t)_{t \ge 0}$.

For the diagonal term, observe that $X_{n}^{i,i}(1) = W_\alpha - R_{n}$, where $R_{n} = P_{n} + B_n \mathbf{1}_{[W_\alpha \ge P_{n}]}$. Notice that $R_n \ge 0$ is bounded by $P_n + 1$. We have $\mathbb{E}[R_n] = \lambda_n + (-\lambda_n - n^{1/\alpha - 1} \alpha_{i,i}^{(n)}) = -n^{1/\alpha - 1} \alpha_{i,i}^{(n)}$.
Furthermore, $\mathbb{E}[R_n^2] \le \mathbb{E}[2P_n^2 + 2B_n^2] \le 2(\lambda_n + \lambda_n^2) + 2p_n = O(n^{1/\alpha - 1})$.
Thus, the sum $\sum_{l=1}^n R_{n,l}$ of i.i.d. variables distributed as $R_n$, has expectation $n \mathbb{E}[R_n] \sim -n^{1/\alpha} \alpha_{i,i}$ and variance $n \textup{Var}(R_n) = O(n^{1/\alpha})$. 
Thus, the variance of $n^{-1/\alpha}\sum_{l=1}^n R_{n,l}$ vanishes as $O(n^{-1/\alpha}) \to 0$. Hence, it converges in probability to the deterministic drift $-\alpha_{i,i} t$.
Therefore, taking $(W_{\alpha, \ell};\ell\in \mathbb{N})$ i.i.d. distributed as $W_\alpha$, and using Slutsky's theorem, 
\[
(n^{-1/\alpha} X_{n}^{i,i}(\lfloor nt \rfloor))_{t \ge 0} = \Big(n^{-1/\alpha} \sum_{\ell=1}^{\lfloor nt \rfloor} W_{\alpha, \ell} - n^{-1/\alpha} \sum_{\ell=1}^{\lfloor nt \rfloor} R_{n, \ell} \Big)_{t \ge 0} \Rightarrow (X^{i,i}(t) + \alpha_{i,i} t)_{t \ge 0}.
\]

Because the limits of the off-diagonal terms are purely deterministic, joint convergence over all $j \in [d]$ holds automatically by Slutsky's theorem, despite the finite-$n$ coupling:
\[
\Big(n^{-1/\alpha}X_{n}^{i,j}(\lfloor nt\rfloor); t\ge 0 \Big)_{j \in [d]} \Rightarrow (Y^{i,1}, \dots, Y^{i,d}), 
\]
where $Y^{i,j}(t)=X^{i,i}(t)\mathbf{1}_{[i=j]}+\alpha_{i,j}t$. This establishes \ref{ass:a1}--\ref{ass:a4}.

\subsubsection{Convergence to a multitype tree with linear drift in the off-diagonal and L\'evy processes in the diagonal}

The preceding example can be generalized as follows. Suppose that $(\widetilde Y_n^i(k);k=0,1,\dotsm)$ are defined by 
\begin{align*}
    \widetilde Y_n^i(k): =Y_n^i(k)-k= \sum_{\ell=1}^k (\chi_\ell^{i,(n)}-1)
\end{align*} for a collection of i.i.d. random variables $(\chi_\ell^{i,(n)};\ell\ge 1)$ with values in $\{0,1,\dotsm\}$ and also independent over $i\in[d]$. Moreover, suppose that for some constants $a_n,b_n\to\infty$ with $b_n/a_n\to\infty$ that 
\begin{align*}
    \left(\frac{a_n}{b_n}\widetilde Y_n^i(b_nt) ;t\ge 0\right) \weakarrow \left(Y^i(t);t\ge 0\right)
\end{align*}
for a L\'{e}vy process $Y^i$ that satisfies
\begin{align*}
    \E[\exp(-\lambda Y^i(t))] = \exp(t\psi_i(\lambda))
\end{align*}
where $\psi_i$ satisfies \eqref{eqn:ass2Laplace}--\eqref{eqn:Greys1type} and $\E[Y^i(1)] \le 0$. 

Suppose additionally that for all $i\in [d]$ and all $n\in \mathbb{N}$ we have
\begin{align*}
    \chi_1^{i,(n)}\le_{\textup{st}}\xi_1^{\ast,(n)}
\end{align*} for some dominating random variables $\xi_1^{\ast,(n)}$ such that the BGW branching process $Z_n^*$ starting from $\lfloor b_n/a_n \rfloor $ many individuals with offspring distribution $\xi_1^{\ast,(n)}$ satisfies \ref{ass:a5}. 
Suppose that $(\alpha_{i,j})_{i,j\in[d]}$ is a $d\times d$ matrix with $\alpha_{i,j}\ge 0$ for $i\neq j$ and $\sum_{j} \alpha_{i,j}=0$ for all $i$. We claim that if $(\vec{p}_n^{i})_{i\in[d]}$ is a family of probability vectors such that $a_n(\vec{p}_n^i-\vec{\epsilon}_i) \to (\alpha_{i,j};j\in[d])$ then the process $\vec  X_n^i:=(X^{i,j}_n;j\in [d])$, defined by  
\begin{align*}
	\vec  X_n^i(k) = \sum_{\ell=1}^k (\vec{\xi}_\ell^{\,i,(n)} - \vec{\epsilon}_i), \quad\textup{where}\quad \vec{\xi}_\ell^{\,i,(n)}|(\widetilde{Y}^j_n;j\in[d]) \sim \operatorname{Multi}\left(\chi_\ell^{i,(n)},\vec{p}^{i}_n\right), 
\end{align*}
satisfies \ref{ass:a1}--\ref{ass:a4}, \ref{ass:A6}. Here $\operatorname{Multi}(n,\vec{p})$ is a multinomial random variable and $\operatorname{Multi}(0,\vec{p})$ is the point mass at $\vec{0}$. In this case the corresponding Laplace transform is given by 
\begin{align*}
    \Psi_i(\vec{\lambda}) = \psi_i(\lambda_i) - \sum_{j=1}^d \alpha_{i,j}\lambda_j.
\end{align*}
Clearly, \ref{ass:a2}--\ref{ass:a4} will hold. Also, \ref{ass:A6} follows as in the previous example, since $\|\vec{\xi}_\ell^{\,i,(n)}\|_1=\chi_\ell^{\,i,(n)}\le_{\textup{st}}\xi_1^{\ast,(n)}$.
To show \ref{ass:a1} holds, from the form of $\Psi_i$ above, we see that the off-diagonal terms $X^{i,j}(t) = \alpha_{i,j}t$ are deterministic. By Slutsky’s theorem, \ref{ass:a1} follows from showing the marginal convergence of $X_{n}^{i,j}$, which is further implied by showing 
\begin{equation}\label{eqnExample2ConvergenceOffDiagonal}
    \frac{a_n}{b_n} X_n^{i,j}(b_n) \weakarrow \alpha_{i,j}+Y^i(1)\bo 1_{[i=j]}.
\end{equation}
This easily follows from the next lemma.

\begin{lem}
Let $Y_n$ be any sequence of non-negative integer-valued random variables such that
\begin{align*}
    \frac{a_n}{b_n}\left(Y_n-b_n\right)\weakarrow Y
\end{align*}
where $a_n,b_n/a_n\to\infty$. 
Fix $i\in [d]$.
Consider for each $n\ge 1$, $(p^j_n;j\in[d])$ a probability distribution on $[d]$ and suppose that there is some $(\alpha_j;j\in[d])$ such that $a_n(p^j_n-\bo 1_{[i=j]})\to \alpha_j\ge 0$ for all $j\neq i$.
For $j\in [d]$, conditionally given $Y_n$, let 
\begin{align*}
\vec{X}_n = (X_n^1, \dots, X_n^d) \sim \operatorname{Multi}(Y_n, \vec{p}_n), 
\end{align*}
a multinomial random vector, where $\operatorname{Multi}(0,p)$ is identically $\bo 0$. 
Then, for all $j\in [d]$
    \begin{align*}
	    \frac{a_n}{b_n} (X^j_n-b_n\bo{1}_{[i=j]})\weakarrow \alpha_j+Y\bo 1_{[i=j]}.
    \end{align*}
\end{lem}
\begin{proof}
Recall that as $x\downarrow0$, we have $\log(1-x) = -(1+o(1))x$ and $1-e^{-x} = (1+o(1))x$. Hence, we have
\begin{align*}
    \E&\left[\exp\Big(-\lambda \frac{a_n}{b_n} \big(X^j_n-b_n\bo 1_{[i=j]}\big)\Big)\right] = \E\left[\left(1-p^j_n+e^{-\lambda a_n/b_n}p^j_n\right)^{Y_n}\right]e^{\lambda a_n\bo 1_{[i=j]}}\\
    &\hspace{3cm}=\E\left[\exp\left\{Y_n\log\left(1 - (1-e^{-\lambda a_n/b_n})p^j_n\right) \right\}\right]e^{\lambda a_n\bo 1_{[i=j]}}\\
    &\hspace{3cm}= \E\left[\exp\left\{ - \left(1+o(1)\right)Y_n\left(1-e^{-\lambda a_n/b_n}\right) p^j_n \right\} \right]e^{\lambda a_n\bo 1_{[i=j]}}\\
    &\hspace{3cm}=\E\left[\exp\left\{-(1+o(1)) \lambda \frac{a_n}{b_n} \big(Y_np^j_n-b_n\bo 1_{[i=j]}\big)\right\}\right].
\end{align*}On the one hand,  if $i\neq j$ then 
\begin{align*}
\frac{a_n}{b_n} Y_np^j_n= 
\frac{a_n}{b_n} (Y_n-b_n)p^j_n + a_n p^j_n \weakarrow \alpha_j,
\end{align*}since
\begin{align*}
    \frac{a_n}{b_n}(Y_n-b_n)p^j_n \weakarrow 0.
\end{align*} 
On the other hand,  if $i=j$ then $X_n^i = Y_n - \sum_{j\neq i}X_n^i$ and so by Slutsky's theorem
\begin{align*}
	\frac{a_n}{b_n}\left(	X_n^i-b_n\right) &= \frac{a_n}{b_n}\left(Y_n-b_n\right) - \frac{a_n}{b_n} \sum_{j\neq i}X_n^j\\
						 &\weakarrow Y - \sum_{j\neq i} \alpha_j = Y+\alpha_i
\end{align*}
where in the last equality we used the observation that $\sum_j \alpha_j = \lim_n a_n \sum_j(p_n^j-1_{[i=j]}) = 0$.
\end{proof}

We apply the previous lemma with $Y_n=\widetilde Y^i_n(b_n)+b_n$, $p^j_n=p^{
i,j}_n$, $\alpha_j=\alpha_{i,j}$ and $X^{j}_n=X^{i,j}_n(b_n)$. 
First note that for any $i,j\in [d]$, conditionally on $\big(\chi_\ell^{i,(n)};\ell\in [b_n]\big)$, we have
\[
X_n^{i,j}(b_n)\stackrel{d}{=}\sum_{\ell=1}^{b_n}{\rm Bin}(\chi_\ell^{i,(n)}, p_n^{i,j})-b_n\bo 1_{[i=j]}
\stackrel{d}{=}{\rm Bin}(\widetilde Y^i_n(b_n)+b_n, p_n^{i,j})-b_n\bo 1_{[i=j]},
\]where $\vec{p}_n^{i}=(p^{i,1}_n,\ldots, p^{i,d}_n)$. Thus \eqref{eqnExample2ConvergenceOffDiagonal} follows. 

\subsubsection{Convergence to a multitype tree with stable processes in the diagonal and Poisson subordinators in the off-diagonal}

In the preceding examples, the off-diagonal limits were restricted to deterministic linear drifts. We now construct an offspring distribution where the limiting branching process exhibits  continuous drift and  compound Poisson jump processes outside the diagonal. 

Fix an index $\alpha \in (1,2)$. Let $A = (\alpha_{i,j})_{i,j \in [d]}$ be a matrix of continuous drifts with $\alpha_{i,i} < 0$ and $\alpha_{i,j} > 0$ for $i \neq j$. For each $i \neq j$, let $c_{i,j} > 0$ dictate the macroscopic jump sizes, and let $\lambda_{i,j} \ge 0$ dictate their rates. 
Let $\tilde{\alpha}_{i,j} = \alpha_{i,j} + c_{i,j} \lambda_{i,j}$ for $i \neq j$, and $\tilde{\alpha}_{i,i} = \alpha_{i,i}$. We assume that the overall mean matrix $\widetilde{A} = (\tilde{\alpha}_{i,j})_{i,j \in [d]}$ is irreducible and that its row sums satisfy $\sum_{j=1}^d \tilde{\alpha}_{i,j} < 0$ for all $i \in [d]$. This guarantees that the Perron-Frobenius eigenvalue of $\widetilde{A}$ is strictly negative. Let $A^{(n)} = (\alpha_{i,j}^{(n)})$, $C^{(n)} = (c_{i,j}^{(n)})$, and $\Lambda^{(n)} = (\lambda_{i,j}^{(n)})$ be deterministic sequences converging to $A$, $C$, and $\Lambda$ respectively; the diagonal terms in $C^{(n)}$ and $\Lambda^{(n)}$ being irrelevant.

For each $i \in [d]$, let $W_\alpha^{(i)}$ be a random variable in the domain of attraction of an $\alpha$-stable law, taking values in $\{-1, 0, 1, 2, \dots\}$, such that $\mathbb{E}[W_\alpha^{(i)}] = 0$ and its upper tail satisfies $\mathbb{P}(W_\alpha^{(i)} \ge m) \le K_i m^{-\alpha}$ for all $m \ge 1$ and some constant $K_i > 0$. We assume $\mathbb{P}(W_\alpha^{(i)} \ge 1) > 0$. Let $(W_{\alpha, \ell}^{(i)}; \ell \ge 1)$ be an i.i.d.\ sequence distributed as $W_\alpha^{(i)}$.

We define the discrete random walk $(X_n^{i,j}(k);k\in\mathbb{N})$  using mutually independent families of random variables. For the off-diagonal terms ($j \neq i$), we combine a microscopic Poisson drift with a macroscopic Bernoulli jump:
\begin{equation}
X_{n}^{i,j}(\ell)-X_{n}^{i,j}(\ell-1) := D_{n,\ell}^{i,j} + J_{n}^{i,j} B_{n,\ell}^{i,j}\qquad \ell\geq 1,
\end{equation}
where $D_{n,\ell}^{i,j} \sim \operatorname{Poi}(n^{1/\alpha - 1} \alpha_{i,j}^{(n)})$ and $B_{n,\ell}^{i,j} \sim \operatorname{Ber}(\lambda_{i,j}^{(n)}/n)$ are mutually independent, and $X_{n}^{i,j}(0)=0$. The deterministic jump size is $J_n^{i,j} = \lfloor c_{i,j}^{(n)} n^{1/\alpha} \rfloor$. Since $\alpha \in (1,2)$, the exponent $1/\alpha - 1 < 0$, ensuring the mean of the Poisson drift vanishes as $n \to \infty$.

For the diagonal term ($j=i$), we define
\begin{equation}
X_{n}^{i,i}(\ell)-X_{n}^{i,i}(\ell-1) := W_{\alpha,\ell}^{(i)} - B_{n,\ell}^{i} \mathbf{1}_{\{W_{\alpha,\ell}^{(i)} \ge 1\}}\qquad \ell\geq 1,
\end{equation}
where $B_{n,\ell}^{i} \sim \operatorname{Ber}(p_n^i)$ is an independent microscopic Bernoulli drift, such that \linebreak $p_n^i: = \min\{1, -n^{1/\alpha-1} \alpha_{i,i}^{(n)} / \mathbb{P}(W_{\alpha,1}^{(i)} \ge 1)\}$, and $X_{n}^{i,i}(0)=0$. Since $\alpha_{i,i} < 0$, we have $p_n^i \in (0, 1)$ for large $n$. The mean of the diagonal increment is exactly $\alpha_{i,i}^{(n)} n^{1/\alpha-1}$.

We set the scaling sequences $a_n = n^{1 - 1/\alpha}$ and $b_n^j = n$ for all $j \in [d]$, making the random walk spatial scaling factor $a_n/b_n^j = n^{-1/\alpha}$.

Let us verify Assumptions \ref{ass:a1}--\ref{ass:a4}. We evaluate the Laplace transform of the scaled off-diagonal terms $j \neq i$ at $\theta \ge 0$, obtaining
\begin{align*}
\mathbb{E}\Big[\exp\Big(-\theta n^{-1/\alpha}  &X_{n}^{i,j}\big(\lfloor nt \rfloor\big)\Big)\Big] 
= \Big( \mathbb{E}\big[ e^{-\theta n^{-1/\alpha} D_{n,1}^{i,j}} \big] \mathbb{E}\big[ e^{-\theta n^{-1/\alpha} J_n^{i,j} B_{n,1}^{i,j}} \big] \Big)^{\lfloor nt \rfloor} \\
&= \exp\Big( \lfloor nt \rfloor n^{1/\alpha - 1} \alpha_{i,j}^{(n)} (e^{-\theta n^{-1/\alpha}} - 1) \Big) \Big( 1 - \frac{\lambda_{i,j}^{(n)}}{n} + \frac{\lambda_{i,j}^{(n)}}{n} e^{-\theta n^{-1/\alpha} \lfloor c_{i,j}^{(n)} n^{1/\alpha} \rfloor} \Big)^{\lfloor nt \rfloor}\\
& \to \exp\Big(-\theta \alpha_{i,j}t+  \lambda_{i,j}t (e^{-\theta c_{i,j}} - 1) \Big).
\end{align*}The latter characterizes convergence to a subordinator $Y^{i,j}(t) =  c_{i,j} N_{i,j}(t)+\alpha_{i,j} t$, where $N_{i,j}$ is a standard Poisson process of rate $\lambda_{i,j}$. Since $\alpha_{i,j} > 0$, the limit $Y^{i,j}$ is strictly increasing, satisfying \ref{ass:a4}. 

For the diagonal term, the scaled compensator $n^{-1/\alpha} \sum_{\ell=1}^{\lfloor nt \rfloor} B_{n,\ell}^{i} \mathbf{1}_{\{W_{\alpha,\ell}^{(i)} \ge 1\}}$ has mean $n^{-1/\alpha} \lfloor nt \rfloor p_n^i \mathbb{P}(W_{\alpha,1}^{(i)} \ge 1) \to -t \alpha_{i,i}$ and variance bounded by $\mathcal{O}(n^{-1/\alpha}) \to 0$. Thus, it converges in probability to $-t \alpha_{i,i}$. By Slutsky's theorem:
\[
\Big(n^{-1/\alpha}  X_{n}^{i,i}\big(\lfloor nt \rfloor\big)\Big)_{t \ge 0} \Rightarrow (X^{i,i}(t) + \alpha_{i,i} t)_{t \ge 0},
\]
where $X^{i,i}(t)$ is an $\alpha$-stable spectrally positive L\'evy process, satisfying \ref{ass:a2} and \ref{ass:a3}. Since all components are constructed using mutually independent families of variables, joint functional convergence to independent limits automatically fulfills \ref{ass:a1}.

Finally, we prove Assumption \ref{ass:A6}. 
We demonstrate that the total offspring generated by a type $i$ individual, $\Sigma_n^{(i)} :=\sum_{j \in [d]} X_{n,1}^{i,j}$, is uniformly stochastically dominated by a fixed heavy-tailed random variable with strictly negative mean. Dropping the negative diagonal compensator, we have the exact pointwise bound $\Sigma_n^{(i)} \le W_{\alpha,1}^{(i)} + \Sigma_{D,n} + \Sigma_{V,n}$, where $\Sigma_{D,n} := \sum_{j \neq i} D_{n,1}^{i,j}$ and $\Sigma_{V,n} := \sum_{j \neq i} J_n^{i,j} B_{n,1}^{i,j}$.

For any integer $k \ge 1$, using the union bound:
\begin{equation}
\mathbb{P}(\Sigma_n^{(i)} \ge k) \le \mathbb{P}\Big(W_{\alpha,1}^{(i)} \ge \lceil k/3 \rceil\Big) + \mathbb{P}\Big(\Sigma_{D,n} \ge \lceil k/3 \rceil\Big) + \mathbb{P}\Big(\Sigma_{V,n} \ge \lceil k/3 \rceil\Big).
\end{equation}
We analyze these three tail probabilities uniformly in $n$:
\begin{enumerate}
    \item By definnition of $W_{\alpha}^{(i)}$, we have $\mathbb{P}(W_{\alpha,1}^{(i)} \ge \lceil k/3 \rceil) \le C_1 k^{-\alpha}$ for some positive $C_1$.
    \item $\Sigma_{D,n}$ is Poisson distributed with a mean uniformly bounded by $\mu^* = \sup_n n^{1/\alpha-1} \sum_{j \neq i} \alpha_{i,j}^{(n)} < \infty$. Because the tail of a Poisson distribution decays super-exponentially, we can find a constant $C_2 > 0$ such that $\mathbb{P}(\Sigma_{D,n} \ge \lceil k/3 \rceil) \le C_2  k^{-\alpha}$.
    \item  For any $m \ge 1$, since $\Sigma_{V,n}$ is a sum of non-negative integers, by the union bound:
    \begin{equation}
    \mathbb{P}(\Sigma_{V,n} \ge m) \le \mathbb{P}(\Sigma_{V,n} \ge 1) \le \sum_{j \neq i} \mathbb{P}(B_{n,1}^{i,j} = 1) \le \frac{\lambda^*}{n},
    \end{equation}
    where $\lambda^* = \sup_n \sum_{j \neq i} \lambda_{i,j}^{(n)}$. Next, notice that $\Sigma_{V,n}$ is deterministically bounded by its maximum possible value $c^* n^{1/\alpha}$, where $c^* = \sup_n \sum_{j \neq i} c_{i,j}^{(n)}$. Therefore, if $m > c^* n^{1/\alpha}$, the probability $\mathbb{P}(\Sigma_{V,n} \ge m)$ is exactly zero. The event can only occur if $n \ge (m / c^*)^\alpha$. Thus, 
    \begin{equation}
    \mathbb{P}(\Sigma_{V,n} \ge m) \le \frac{\lambda^*}{n} \mathbf{1}_{\{ n \ge (m / c^*)^\alpha \}} \le \frac{\lambda^*}{(m / c^*)^\alpha} = \lambda^* (c^*)^\alpha m^{-\alpha}.
    \end{equation}
    Applying this to $m = \lceil k/3 \rceil $ yields $\mathbb{P}(\Sigma_{V,n} \ge \lceil k/3 \rceil)  \le C_3 k^{-\alpha}$ for some positive $C_3$. 
\end{enumerate}

Combining these, there exists a constant $M = C_1 + C_2 + C_3$ such that $\sup_n \mathbb{P}(\Sigma_n^{(i)} \ge k) \le M k^{-\alpha}$ for all $k \ge 1$. By scaling the continuous drifts $\alpha_{i,j}$, the jump rates $\lambda_{i,j}$, and the stable tail constant $K_i$ to be sufficiently small, we can make $M$ arbitrarily small.

To complete the stochastic domination, let $\delta = \inf_n \mathbb{P}(\Sigma_n^{(i)} = -1)$. Note that $\Sigma_n^{(i)} = -1$ if $W_{\alpha,1}^{(i)} = -1$, $B_{n,1}^i = 0$, $\Sigma_{D,n} = 0$, and $\Sigma_{V,n} = 0$. By independence, as $n \to \infty$, the probabilities of the last three events tend to $1$. Since $\mathbb{P}(W_{\alpha}^{(i)} = -1) > 0$, we have $\delta > 0$. Thus, $\mathbb{P}(\Sigma_n^{(i)} \ge 0) \le 1 - \delta$.

We define a fixed dominating random variable $\widetilde{W}^* \in \{-1, 0, 1, \dots\}$ by setting its tail probabilities: $\mathbb{P}(\widetilde{W}^* \ge 0) = 1 - \delta$, and $\mathbb{P}(\widetilde{W}^* \ge k) = \min(1-\delta, M k^{-\alpha})$ for all $k \ge 1$. Because $M k^{-\alpha}$ decays monotonically, this defines a valid, non-increasing tail function. Furthermore, the mean is $\mathbb{E}[\widetilde{W}^*] = \sum_{k=1}^\infty \min(1-\delta, M k^{-\alpha}) - \delta$. By making $M$ small enough, this sum is strictly less than $\delta$, making the mean strictly negative. Since $\Sigma_n^{(i)} \ge -1$ almost surely, we have $\mathbb{P}(\Sigma_n^{(i)} \ge k) \le \mathbb{P}(\widetilde{W}^* \ge k)$ for all $k \ge -1$, meaning $\Sigma_n^{(i)} \le_{\textup{s.t.}} \widetilde{W}^*$ uniformly in $n$. 
Because $\widetilde{W}^*$ possesses an $\alpha$-stable tail and a strictly negative mean, a single-type BGW process with this offspring distribution is subcritical and dies out with positive probability. 
Thus, as in the previous examples, the multitype branching process $\vec{Z}_n$ also goes extinct with positive probability, rigorously establishing Assumption \ref{ass:A6}.

\subsection{Overview of the Article}

We first use Lemma \ref{lem:decoration} to change from the conditioned MBGW tree to the conditioned glued (discrete) decoration.
After that, we use Theorem \ref{thm:graphConv1} to prove the convergence of the conditioned glued decorations to the conditioned glued decoration $\bM^{i_0}_r$, which by construction \ref{constructionMultitypeLevyTreeGluedDecoration} is the conditioned multitype L\'evy tree.

In Section \ref{sec:discreteDecorations}, we reinterpret the random-walk encoding of MBGW trees introduced by Chaumont and Liu \cite{MR3449255} in terms of decorations, thereby establishing Lemma \ref{lem:decoration}. The conditioned version of this result is presented in Corollary \ref{cor:DecorationSpalf}, where the MBGW tree is conditioned on the first subdecoration attached to the root being large, as in \eqref{eqnDecorationsWithBigSubtreeBig}. This conditioned tree is precisely the one that, after rescaling, converges to the multitype L\'evy tree.

In Section \ref{sec:DecorationOfUlamTree}, we introduce the necessary background on the Gromov–Hausdorff–Prohorov topology and provide the formal definitions of decorations. We then establish Theorem \ref{thm:graphConv1}, which concerns the convergence of glued decorations, following the approach pioneered by S\'enizergues \cite{Senizergues.20} for proving convergence of glued metric measure spaces.

Section \ref{sec:levy} recalls key properties of single-type L\'evy trees and L\'evy processes, together with the recent results of Chaumont and Marolleau \cite{MR4193902,MR4575008} on first hitting times of spaLf's.

Finally, in Section \ref{sec:convergence}, we show that the decorations $\cD^n$ encoding the metric measure spaces $T_n$ in Theorem \ref{thm:MAIN} satisfy, after suitable rescaling, Assumptions \ref{ass:1thm1}–\ref{ass:3thm1}, which are precisely those required for Theorem \ref{thm:graphConv1}. This theorem then yields the limiting decoration, explicitly described in Construction \ref{cons:continuum2}.

\section{Decorations in the Discrete}\label{sec:discreteDecorations}
In this section we describe MBGW trees as glued decorations of the Ulam tree, which will prove Lemma \ref{lem:decoration}. We refer the reader back to Section \ref{sec:SingleEncode} for an overview of how to encode a single-type plane tree using a single $\Z$-valued walk. 
First we show that a subdecoration, as described in Construction \ref{construction:Discretesingletype}, has the same law as a subtree in a MBGW tree together with added infinitely marked points. 
After that, we relate the decoration as in Construction \ref{construction:DiscreteDecoration} and the glued decoration as in Definition \ref{definitionDecorationGluedMetricspaceGluedDecoration}, with the whole MBGW tree. 

\subsection{Subdecorations and Infinitely Marked Points from Multitype Trees}

We recall that a multitype plane tree $\tr$ is a plane tree with a ``type'' function which assigns to
each vertex $v\in \tr$ its type $\TYPE(v)\in [d]$ for some $d<\infty$. All trees will be rooted at a vertex $\rho\in \tr$. Given a multitype tree $\tr$, we will say a tree $\tr'\subseteq \tr$ is a subtree of type $i$ if it is a maximally connected subgraph of $\tr$ whose vertices are all of type $i$. That means if $w\notin \tr'$ but $v\in \tr'$ then $w\sim v$ implies that $\TYPE(w)\neq i$. We root these subtrees at the unique element which is closest to the root in $\tr$. 

Recall from Subsection \ref{subsectionMultitypeEnconding} the definition of the reduced tree $\tred$ of a multitype tree $\tr$. 
We will say that a type $i$ subtree $\tr'\subseteq\tr$ is at \textit{reduced height} $k$ if $d(w_\red,\rho_\red) = k$ where $w_\red\in \tred$ is the vertex obtained by collapsing the tree $\tr'$, and $\rho_\red$ is the root of $\tr_\red$, and the distance is in the tree $\tred$. Recall Figure \ref{figTreeWithBFOAndDFO} for a visual representation.

\subsubsection{Induced labeling on a reduced subtree}\label{subsectionInducedLabelingOnAReducedSubtree} Recall from Subsection \ref{subsectionMultitypeEnconding} that the reduced subtree $\tred$ has associated a breadth-first order of its vertices $w_1,w_2,\dotsm, w_{\# \tred}$. Let us now describe how we further index the reduced subtree using the Ulam-Harris tree $\bbU$. First the root $\rho_\red\in \tred$ is labeled by $\emptyset\in \bbU$. Now for a vertex $w\in \tred$ which has been assigned label $\bp\in \bbU$, we give labels to its children as follows: let $w_{j,1},\dotsm , w_{j,K_j}$ be the $K_j$ many type $j$ children in $\tred$ of $w$ (where $j$ is different from the type of $w$). Order those labels as $w_{j,(1)},\dotsm , w_{j,(K_j)}$ corresponding to subtrees $\tr'_{j,(1)},\dotms, \tr'_{j,(K_j)}$ listed in order of \textit{decreasing cardinality} of the subtrees (with ties broken arbitrarily). 
Then these vertices are assigned labels in $\bbU$ as follows:
\begin{equation}\label{definitionLabelingOfVerticesInReducedTree}
w_{j,(m)}\textup{ is labeled by }\bp ( j+ (m-1)d),\qquad m\in [K_j].
\end{equation}Recall that $\bp ( j+ (m-1)d)$ refers to the concatenation of two elements of $\bbU$. 
In Figure \ref{figTreeAndReducedTreeUsingUlam} we show an example of such a labeling. 

\begin{figure}
\centering
\includegraphics[scale=0.7]{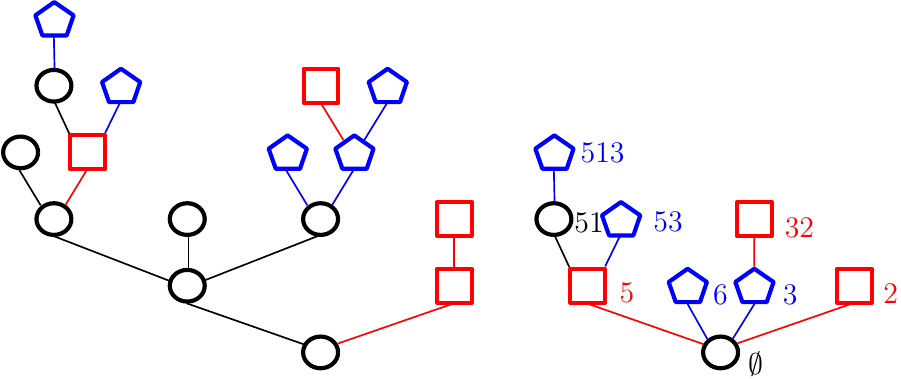}
\caption{The reduced subtree together with the labeling via the Ulam-Harris tree. Recall that type $1$ individuals are black, type $2$ are red, and type $3$ are blue.}
\label{figTreeAndReducedTreeUsingUlam}
\end{figure}

In the sequel we will make no difference between the vertex $w\in \tred$ and the corresponding label $\bp=\bp(w,\tred)$ (this also depends on the procedure for breaking ties) in the Ulam-Harris tree. This allows us to associate $\tred\subset\bbU$ which necessarily is a finite subtree, whenever $\tr$ if finite. Similarly, we will associate the vertex $w\in \tred$ with the corresponding typed subtree $ \tr'_\bp$. 
See Figure \ref{figTreeAndSubtreesUsingUlam} for an example. 

\begin{figure}
\centering
\includegraphics[scale=0.8]{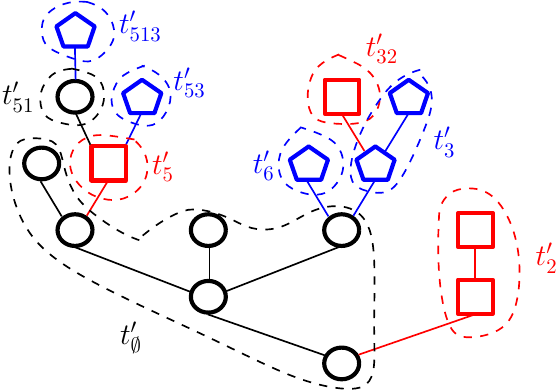}
\caption{The type $j$ subtrees at reduced height 0, 1, 2 and 3. Since the red square vertices are type $2$, the corresponding subtrees are assigned indexes $2,5,8,11,\dotsm\in \bbU$. Those type 2 subtrees are also labeled  $\tr'_{2,(1)}, \tr'_{2,(2)}, \tr'_{2,(3)}$. The blue vertices are type $3$ and so the corresponding subtrees are indexed by $3,6,9,\dotsm$. The type $2$ subtree in the right (in red) is labeled $\tr_2'$ because it is the largest type $2$ subtree at reduced height $1$.}
\label{figTreeAndSubtreesUsingUlam}
\end{figure}




\subsubsection{Subtrees as Metric Measure Space of Subdecorations}\label{subsubsectionSubtreesAsMetricMeasureSpaceOfSubdecorations}

The objective now is to define $\cD(\bp)$ for any $\bp\in \bbU$, using the previous labeling. 
A natural choice of the metric space $D_\bp$ of the subdecoration $\cD(\bp)$ is using the corresponding type $j$ subtrees $\tr_\bp'$; however, this will not work. Indeed, consider the following example. Let $\tr$ be a tree with two nodes $\{x,y\}$ where $x$ is of type 1 and $y$ is of type 2, and the tree is rooted at $x$. If $D_\emptyset =\{x\}$, $D_2 = \{y\}$ and $D_\bp = \bzer$ for all other $\bp$, then as a metric space $\sG(\cD) = \{pt\}$ just consists of a single point. This is because all the infinitely marked PMMs $\cD(\bp)$ have metric spaces $D_\bp$ that consist of a single point and, therefore, all the marks and roots must be the unique point on those metric spaces. Consequently, the defining equivalence relation $\sim$ in the definition of $\sG^*(\cD)$ identifies all the elements in $\sG^*(\cD)$ with each other. In short, if we use $\tr_\bp'$ then we lose all the edges of the form $\{\rho_{\bp,\ell}, x_{\bp;\ell}\}$.

The way we choose to overcome this slight problem is to \textit{plant} every subtree of each type as follows. 
Fix $\bp \in \bbU$.
On the one hand, assume the latter has associated a vertex $w\in \tred$ with type $i$ subtree $\tr'_\bp$. 
Denote by $\rho$ the root of $\tr'_\bp$. 
We write $\widetilde{\tr}'_\bp$ as the subtree $\tr'_\bp$ with a single additional vertex denoted by $\widetilde\rho$ which is connected only to the root $\rho$ of $\tr'_\bp$. The new tree $\widetilde{\tr}'_\bp$ is rooted at the new vertex $\widetilde{\rho}$. By convention we will say that the new vertex $\widetilde{\rho}$ does not have a type. On the other hand, if $\bp\in \bbU\setminus \tred$, that is, $\bp$ is not a label of a vertex in $\tred$, then $\tr'_\bp$ is empty, and we define $\widetilde \rho_\bp$ also as a single additional vertex with no type.
In this case the metric measure space will be $\widetilde D_\bp: = (\{\widetilde{\rho}_\bp\}, \widetilde{\rho}_\bp, d_g, \vec{\boldsymbol{0}})$, where $d_g$ is the graph distance. 
The corresponding subdecoration will be defined as $(\{\widetilde{\rho}_\bp\}, \widetilde{\rho}_\bp, d_g, \vec{\boldsymbol{0}},(\widetilde{\rho}_\bp)_{\ell\ge 1})$. Note that the infinitely marked points are defined to be all equal to the additional root in this case. We call such metric measure space and subdecoration trivial.

In the case $\bp\in \tred$ with corresponding subtree $\tr'_\bp$ type $i$, the measure $\vec{\mu}^{(i)}$ that we place on $\widetilde{\tr}'_\bp$ is
\begin{equation*}
\langle \vec{\mu}^{(i)},\vec{\epsilon}_j\rangle (A) = \#(A\cap \tr'_\bp)1_{[i=j]},\qquad j\in[d].
\end{equation*}
In particular, $\vec{\mu}^{(i)}(\{\widetilde\rho\}) = \vec{0}$.
Sometimes we also use the notation $\vec \mu_\bp$ instead of $\vec{\mu}^{(i)}$, which also specifies the dependency on $i$, when $\bp=\bq \ell$ and $\ell\equiv i \mod d$.
The metric measure space $\widetilde D_\bp$ of the subdecoration $\cD(\bp)$ is defined as $(\widetilde{\tr}'_\bp ,\widetilde{\rho},d_g, \vec{\mu}^{(i)})$ where $\widetilde \tr'_\bp$ was defined above. 
It only remains to describe the infinitely many marked points for non-trivial subtrees $\widetilde\tr'_\bp$.

\subsubsection{Infinitely Marked Points in Subtrees}\label{subsubsectionInfinitelyMarkedPointsInSubtrees}

Let $\bp\in \bbU$ and suppose that $\widetilde{\tr}'_\bp$ is a non-trivial type $i$ (rooted) subtree of $\tr$. The element $\bp\in \bbU$ necessarily corresponds to some vertex $\bp\in\tred$ and this vertex might have children in $\tred$. Necessarily, those offspring are type $j\neq i$. 

By construction (see \eqref{definitionLabelingOfVerticesInReducedTree}), the type $j$ children of $\bp$ in $\tred$ are of the form $\bp \ell$ where $\ell = j, j+d, j+2d,\dotsm, j+(K_j-1)d$ for some $K_j<\infty$. Moreover, each of the type $j$ children corresponds to the type $j$ subtree $\tr'_{\bp \ell}\subset \tr$ and this is rooted at $\rho_{\bp \ell}\in \tr_{\bp \ell}'$. 
Each vertex $\rho_{\bp \ell}$ has a parent in $\tr$ which also belongs to the type $i$ subtree $\tr_\bp'\subset\widetilde{\tr}_{\bp}'$. 
We call this parent $x_{\bp;\ell}$ and it is the $\ell^\text{th}$ marked point in the decoration $\cD(\bp)$. When $\ell = j+md$ for $m\ge K_j$, we set $x_{\bp;\ell} = \widetilde\rho_{\bp}$.
See Figure \ref{figTemptyset} for an example. 

\begin{figure}
\centering
\includegraphics[scale=0.85]{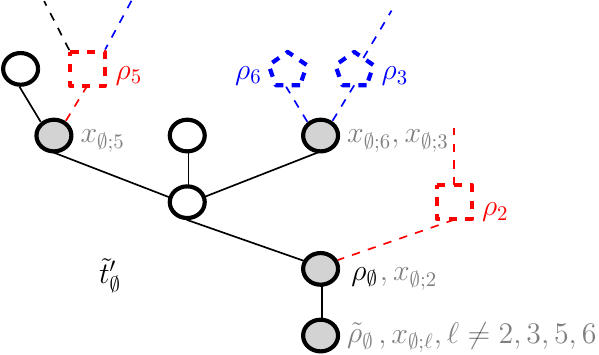}
\caption{The subtree $\widetilde{\tr}'_\emptyset$ together with the roots $(\rho_{\bp };\bp\in \UU)$ and the marks $(x_{\bp;\ell };\bp\in \UU)$. 
The marks in $\widetilde{\tr}_\emptyset' $ are depicted with filled gray to emphasize that they are the roots of the subtrees $\widetilde{\tr}_{\bp}'$.
}
\label{figTemptyset}
\end{figure}

\subsubsection{Subtrees in Multitype Bienaym\'e-Galton-Watson trees as Subdecorations}\label{subsubsectionSubtreesInMBGWTreesAsSubdecorations}

Recall Definition \ref{definitionBGWtree}. We will say that a tree $\widetilde T$ is a \textit{planted} MBGW tree rooted at a type $i_0$ vertex if $\widetilde T = \{\widetilde{\rho}\}\cup T$ for a MBGW tree $T$ and a single additional vertex $\widetilde{\rho}$ which shares an edge with $\rho\in T$. To be consistent we say $\TYPE(\rho) = i_0$.

Recall from \eqref{eqnCodingSubtreesOfMBGWWitRandomWalk}, that for every $i\in [d]$, the  type $i$ subtrees in $T$ can be coded by the excursion measures $\vec N^{i,\vec \varphi}$. 
Also, observe that since every vertex has an independent number of children, {for every $j$, }the type $j$ subtrees in $T$  growing from a type $i$ subtree (with $i\neq j$), are independent and identically distributed. 
Note that there are $K_j:=X^{i,j}(\tau^i(1))$ many of them (here we omit the dependency of the random walk $\vec X^i$ and the subtree). 
That means if $\widetilde{\tr}'_\bp$ has been constructed and is a type $i$ subtree such that there are $K_j\ge 0$ many points $x_{\bp; j+md}\in \widetilde{\tr}'_\bp \setminus\{\widetilde{\rho}_\bp\}$ (counted with multiplicity), each with corresponding subtree $\widetilde \tr'_{\bp(j+md)}$, then the subtrees
\begin{equation}\label{eqnSubtreesTypejOfaSubtreeTypei}
(\widetilde{\tr}'_{\bp (j+md)}; m = 0,1,\dotms, K_j-1)
\end{equation} are i.i.d. samples of a tree $\widetilde T^{(j)}$ re-arranged in decreasing order of size with ties broken arbitrarily (cf. Jagers' theorem on stopping lines \cite{MR1014449}). 
Moreover, conditionally on just the value $K_j$ they are \textit{independent} of $\widetilde{\tr}'_{\bp}$. In particular, conditionally given
\begin{equation*}
\left(\#\{m : x_{\bp;j+md} = v\}; v\in \tr'_{\bp}\right) = (\chi^j(v); v\in \tr'_\bp),
\end{equation*} where $\chi^j(v)$ is the number of type $j$ children of $v$ in $\tr$, 
the sequence 
\begin{equation}\label{eqnInfinitelyMarkedPointsInSubtreeAreUniform}
(x_{\bp; j},x_{\bp; j+d},\dotsm, x_{\bp; j+(K_j-1)d})
\end{equation}
is uniformly distributed among all sequences $(y^j_1,\dotms, y^j_{K_j})$ of vertices in $\tr'_\bp$ such that $\#\{m:y^j_m = v\} = \chi^j(v)$ for all $v\in \tr'_\bp$.
In other words, by independence of the subtrees in \eqref{eqnSubtreesTypejOfaSubtreeTypei} between them and between $\tr'_\bp$, ordering the subtrees of type $j$ by decreasing size, is independent of the label $v\in \tr'_\bp$.
Note that the above holds vacuously for $j = i$ as $K_i = 0$ in this situation. 

As in the construction of the glued decoration before Definition \ref{remarkConstructionOfGluedDecoration}, we have now all ingredients to describe a multitype BGW tree as a glued decoration.
Indeed, in Subsections \ref{subsectionInducedLabelingOnAReducedSubtree} and  \ref{subsubsectionSubtreesAsMetricMeasureSpaceOfSubdecorations} we have described how to label and define the subdecorations. 
In Subsections \ref{subsubsectionInfinitelyMarkedPointsInSubtrees} and \ref{subsubsectionSubtreesInMBGWTreesAsSubdecorations} we described how the infinitely marked points are constructed from a multitype tree, and that they are uniformly distributed.
To prove Lemma \ref{lem:decoration}, it remains to describe how we create the glued decoration of a MBGW tree, which is the main content on the proof of Lemma \ref{lem:decoration}. 
For a visual representation of the proof see Figure \ref{figDifferentCodingsOfTheTree}.

\begin{figure}
\centering
     \includegraphics[scale=.8]{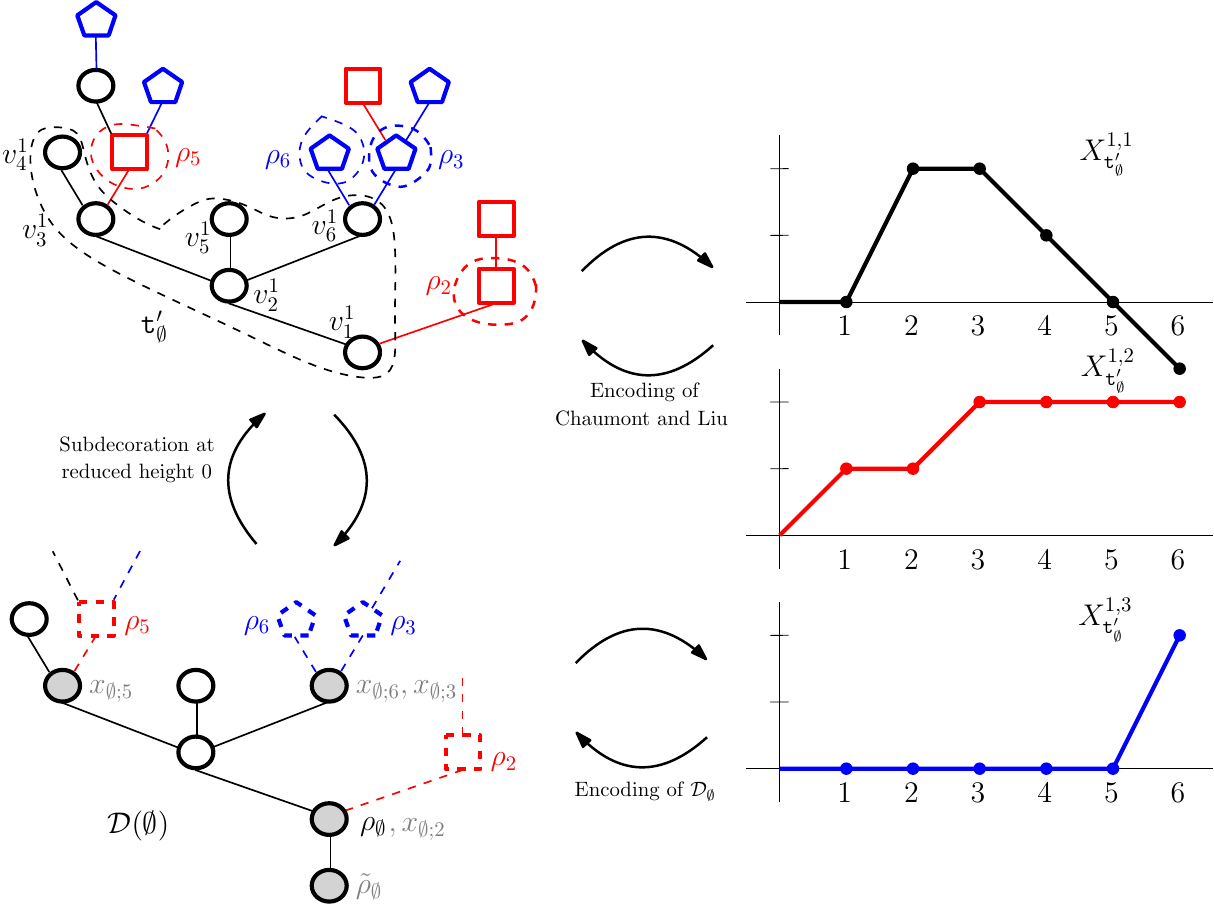}
\caption{In the top-left figure, we show the subtree $\tr'_\emptyset$, together with the positions in which the subtrees $(\tr'_{\ell};\ell\in \mathbb{N})$ growing from it, are pasted. Note that such subtrees have roots $(\rho_\ell;\ell\in\mathbb{N})$ which are encoded by vertices at reduced height one, as in Figure \ref{figTreeAndReducedTreeUsingUlam}. On the middle-right figure, we show the encoding of $\tr'_\emptyset$. Finally, on the bottom-left figure, we show the subdecoration $\cD(\emptyset)$, which contains the marks $(x_{\emptyset;\ell};\ell\in\mathbb{N})$, where the subtrees $(\tr'_{\ell};\ell\in \mathbb{N})$ are going to be pasted. Note that here, the subtrees shown in dotted lines, are not part of $\cD(\emptyset)$. }
\label{figDifferentCodingsOfTheTree}
\end{figure}

\subsection{Proof that Multitype Bienaym\'e-Galton-Watson trees are Glued Decorations} 

In this subsection, we prove  Lemma \ref{lem:decoration}.
In order to do so, we will compare the first step in the construction of the two spaces: the glued decoration  $\sG(\cD^{i_0,\vec \varphi})$ restricted to subdecorations $\cD(\bp)$ with $|\bp|\leq 1$, and the subtrees at reduced height at most one, in a MBGW tree.

For the former, recall Construction \ref{construction:Discretesingletype} of subdecorations, Construction \ref{construction:DiscreteDecoration} of a decoration and Definition \ref{remarkConstructionOfGluedDecoration} for the law of the glued decoration. 
First we start with a subdecoration $\cD(\emptyset)\sim \vec N^{i_0,\vec \varphi}$ together with its infinitely marked points $(x_\ell)_\ell$. 
For $j\neq i_0$, recall by \eqref{eqnDefinitionOfInfinnitelyMarkedPointInSubdecoration} that $(x_{j+(m-1)d}; m\in [\vec \#^{(i_0)}_j \cD(\emptyset)])$ are uniformly distributed conditionally given $\vec \#^{(i_0)}_j \cD(\emptyset)$. 
Recall also from \eqref{defVecCardinalityJOfASubdecoration} that $\vec \#^{(i_0)}_j \cD(\emptyset)=X^{i_0,j}(\tau^{i_0}(1))$. 
Then, we sample i.i.d. copies $\cX^{j}_{\emptyset,m}\sim \vec N^{j,\vec \varphi}$ for $m\in [\vec \#^{(i_0)}_j \cD(\emptyset)]$, conditionally independent of $\cD(\emptyset)$ given $\vec \#_j \cD(\emptyset)$. 
Finally, we order them in decreasing order of cardinality $(\cX^{j}_{\emptyset,(1)},\cX^{j}_{\emptyset,(2)},\dotsm )$ and glue the $m^{\text{th}}$ by its root, with the marked point $x_{j+(m-1)d}$. 
Thus, the constructed subdecorations and infinitely marked points are
\begin{equation}\label{eqnCompareSubdecorationsAndSubtreesFirstStep1}
\big(\cD(\emptyset),(x_{j+(m-1)d})_m,(\cX^{j}_{\emptyset,(m)})_m,m\in [\vec \#^{(i_0)}_j \cD(\emptyset)],j\neq i_0\big). 
\end{equation}

On the other hand, 
we construct the subtree $\tr'_\emptyset$ in the MBGW tree, corresponding to the root, and having type $i_0$. 
Note by \eqref{eqnCodingSubtreesOfMBGWWitRandomWalk} that such a subtree is also constructed from a 
\L ukasiewicz path $\vec X^{i_0}$ (equivalently, an excursion measure $\vec N^{i_0,\vec \varphi}$). 
Define $\widetilde \tr'_\emptyset$ as $\tr'_\emptyset$ planted in the additional vertex $\widetilde \rho_\emptyset$ connected to the root. 
Now, consider any vertex $w_k\in \tr'_\emptyset$ for some $k\in [\tau^{i_0}(1)]$. 
Fix $j\neq i_0$ and note that $w_k$ has  $X^{i_0,j}(k)-X^{i_0,j}(k-1)$ children of type $j$. 
Thus we sample i.i.d. subtrees $\tr'_{j,k,m}\sim \vec N^{j,\vec \varphi}$ for $m\in [ X^{i_0,j}(k)-X^{i_0,j}(k-1)]$, conditionally independent of $\tr'_\emptyset$ given $X^{i_0,j}(k)-X^{i_0,j}(k-1)$.
Connect their root to $w_{k}$ by an edge (equivalently, construct the planted subtree $\widetilde \tr'_{j,k,m}$ and glue its root to $w_k$).
Then, order all the subtrees $(\tr'_{j,k,m};m\in [ X^{i_0,j}(k)-X^{i_0,j}(k-1)], k\in [\tau^{i_0}(1)])$ by decreasing order of cardinality, and label them $(\tr'_{j,(m)};m\in [ X^{i_0,j}(\tau^{i_0}(1))])$.
Finally, from all the $X^{i_0,j}(\tau^{i_0}(1))$ subtrees  type $j$, assign label $j+(m-1)d\in \bbU$ to the root of $\tr'_{j,(m)}$, and assign label $x_{j+(m-1)d}$ to the vertex in $\tr'_\emptyset$ connected to this subtree by an edge. 
Recall from \eqref{eqnInfinitelyMarkedPointsInSubtreeAreUniform} that $(x_{j+(m-1)d};m\in [X^{i_0,j}(\tau^{i_0}(1))])$ are uniformly distributed. 
We collect the random planted subtrees and infinitely marked points in 
\begin{equation}\label{eqnCompareSubdecorationsAndSubtreesFirstStep2}
\big( \widetilde \tr'_\emptyset,(x_{j+(m-1)d}),(\widetilde \tr'_{j,(m)}),m\in [ X^{i_0,j}(\tau^{i_0}(1))],j\neq i_0\big). 
\end{equation}

Note that both \eqref{eqnCompareSubdecorationsAndSubtreesFirstStep1} and \eqref{eqnCompareSubdecorationsAndSubtreesFirstStep2} have spaces each equipped with the graph distance, and the measure is the counting measure on the type of the subspace and the remaining coordinates are null (that is, counting measure $\vec \mu^{(i_0)}$ for $\cD(\emptyset)$ and $\tr'_\emptyset$, and counting measure $\vec \mu^{(j)}$ for the subdecorations of type $j$ and subtrees of type $j$ growing from them, respectively). 
This implies that \eqref{eqnCompareSubdecorationsAndSubtreesFirstStep1} and \eqref{eqnCompareSubdecorationsAndSubtreesFirstStep2} have the same distribution. 

Continuing in this way, by independence, we can construct the glued decoration restricted to subdecorations $\cD(\bp)$ with $|\bp|\leq h$ and multitype BGW tree with subtrees at reduced height at most $h$, for any $h\in \N$ and prove that they have the same distribution. 
This proves the result. 
\qed

\subsection{Glued Decoration Conditioned with a Large First Subtree}\label{subsectionGluedDecorationConditionedWithALargeFirstSubtree}

After proving in Lemma \ref{lem:decoration} that the glued decoration $\sG(\cD^{i_0,\vec \varphi})$, with law $M^{i_0,\vec \varphi}$ as defined in \ref{remarkConstructionOfGluedDecoration}, is a MBGW tree, in this section we describe the glued decoration \emph{conditioned on the size of the tree} $\cD(\emptyset)$ being at least a given value. 
Namely, for any $r\in \N$ we work with a glued decoration $\sG(\cD_{r})$ with law $M^{i_0,\vec \varphi}(-| \#\cD(\emptyset) \ge r) = \PR\left(\sG(\cD^{i_0,\vec \varphi})\in - | \#\cD(\emptyset)\ge r\right)$, as defined in \eqref{eqnDecorationsWithBigSubtreeBig}. 
Recall also from \eqref{eqnCodingSubtreesOfMBGWWitRandomWalk} and the construction above such equation, that each type $j$ subtree $T^{(j)}$ is constructed according to an \textit{excursion measure} of $\vec X^{j}$, namely  $\vec N^{j,\vec \varphi}$.

Certainly we can condition on several different events, for example $\cD(\emptyset)$ having height at least $k$, its number of leaves, conditioning the total size of the tree to be $k$, or the number of individuals of type $j$ to be $k_j$, or even starting with a multitype BGW forest with sufficiently enough roots, but we leave establishing those limit theorems to the reader.  

\subsubsection{Connections between the Reduced Tree and Hitting times}

In order to apply our Theorem \ref{thm:graphConv1}, we need to establish some properties of MBGW forests which allow us to verify Hypotheses \ref{ass:2thm1}--\ref{ass:3thm1} therein, as Hypotheses \ref{ass:1thm1} is not hard to establish using Proposition \ref{prop:dlHeight}.

We use the decomposition of Chaumont and Liu \cite{MR3449255}. Let $\vec X^j$ denote the $\Z^d$-valued random walks from the previous section (or as in Subsection \ref{subsectionMBGWIntro}). Following the notation therein, for $i,j\in [d]$ with $i\neq j$ and $m\in \Z_+$ we define the processes
\begin{align*}
\overline{X}^{i,j}(m) = X^{i,j }\circ\tau^{i}(m),\qquad \tau^{i}(m)  = \min\{k: X^{i,i}(k) = -m\},
\end{align*} as before. 
Let us fix some $\vec{r}\in\Z^d_+\setminus\{\vec 0\}$ arbitrary. 
The \emph{multivariate first hitting/passage times} studied in \cite{MR3449255}, are given by 
\begin{equation}\label{eqn:minSolDiscrete}
\vec T(\vec r): = \min\left\{\vec{m}\in \Z^d_+: \sum_{i=1}^d X^{i,j}(m_i) = -r_j,\quad \forall j\textup{ s.t. }m_j<\infty\right\}.
\end{equation}
Set $ \vec T(\vec r)=\vec n=(n_1,\dotsm, n_d)$, and for any $j$ such that $n_j<\infty$, define  
\begin{equation*}
k_j: = -\min_{0\le m\le n_j} X^{j,j}(m) = -X^{j,j}(n_j).
\end{equation*}
Then Lemma 2.6 in \cite{MR3449255} implies that $\vec{k}=(k_1,\dotsm, k_d)$ is the minimal solution to \eqref{eqn:minSolDiscrete} for the process $\overline{X}^{i,j}$.
By the (sub)criticality assumption on $\vec \varphi$, we note that both $\vec{n}$ and $\vec{k}$ are in $\Z_+^d$ a.s. 

Recalling the construction of the minimal solution \eqref{eqn:minSolDiscrete} in \cite{MR3449255}, for any $j\in  [d]$, $\vec{r}\in\Z^d_+\setminus\{\vec 0\}$ and $h \geq 0$ we define
\begin{align}
R^{(0),j} &= r_j, & U^{(0),j} &= \min\{m: X^{j,j}(m)=-r_j\}\nonumber\\
R^{(h+1),j} &= r_j + \sum_{\ell\neq j} \overline{X}^{\ell,j}(R^{(h),\ell}) , &U^{(h),j} &= \tau^{j}(R^{(h),j}) =\min\left\{m: X^{j,j}(m) = - R^{(h),j}\right\}.\label{eqn:Rrecurse2}
\end{align}
Note that $R^{(h+1),j} = r_j+ \sum_{\ell\neq j} X^{\ell,j}(U^{(h),\ell})$.

\begin{lem}\label{lem:spalfAndDiscTree}
Suppose that $F$ is a MBGW forest with offspring distribution $\vec \varphi$ with $r_j$ many type $j$ roots for each $j\in [d]$, for $\vec{r}\in\Z^d_+\setminus\{\vec 0\}$. Let $F^{\red}$ denote the corresponding reduced forest.
Then jointly over $j\in[d]$ and $h\ge 0$
\begin{align*}
    \#&\{v\in F^{\red}: \TYPE(v) = j, \hgt(v)\le h \textup{ in }F^{\red}\} \overset{d}{=} R^{(h),j}\\
    \#&\{v\in F: \TYPE(v) = j, v\textup{ is at reduced height }\le h\} \overset{d}{=} U^{(h),j}.
\end{align*}
\end{lem}
\begin{proof}
To prove the first equality, observe that 
\begin{align*}
&\left(\overline{X}^{i,j}(m)-\overline{X}^{i,j}(m-1) ;j\neq i\right)_{m\ge 1}\\
&\quad\overset{d}{=} \left(\#\left\{\textup{type }j \textup{ children of the }m^\text{th}\text{ vertex type $i$ in }F^\red\right\}; j\neq i\right)_{m\ge 1}
\end{align*}
where recall from \ref{subsectionMultitypeEnconding} that the type $i$ vertices are labeled in a breadth-first manner in $F^\red$.
Since in the reduced forest none of the individuals type $j\in [d]$ have children of type $j$, we can write $\overline X^{j,j}(m):=-m$. 
Then, we have
\[
R^{(h),j}-R^{(h-1),j}=r_j+\sum_{\ell\neq j} \overline{X}^{\ell,j}(R^{(h-1),\ell})-R^{(h-1),j}=r_j+\sum_{\ell\in [d]} \overline{X}^{\ell,j}(R^{(h-1),\ell}).
\]Note that the latter is the number of individuals type $j$ with height $h$ in $F^{\red}$, by the Lamperti representation of MBGW forests in Lemma 3.5 of \cite{MR3444314}.  
Since $R^{(0),j}=r_j$, we have the desired equality.    

The second equality in distribution, follows from the fact that for a tree $\tr$ with \L ukasiewicz path $X_\tr$ we have (see (6.15) in \cite{MR2245368}) that $
\#\tr = \min\{m: X_\tr(m) = -1\}$.
Indeed, recall from \ref{subsectionMultitypeEnconding} that the type $j$ vertices are labeled in a depth-first order in $F$ (and the subtrees of type $j$ are enumerated in a breadth-first order). 
When $h=0$, the value $U^{(0),j}$ counts the number of individuals type $j$ in subtrees of type $j$ connected to the $r_j$ roots, which is precisely the number of individuals type $j$ at reduced height $0$. 
Note that such individuals are coded by excursions of $X^{j,j}$ which occur between $0$ and $U^{(0),j}$. 
Recursively, $U^{(h),j}$ counts the number of individuals type $j$ in subtrees of type $j$ connected to the subtrees of type $\ell\neq j$, which are at reduced height at most $h$. 
They are coded by excursions of $X^{j,j}$ which occur between $U^{(h-1),j}$ and $U^{(h),j}$. 
\end{proof}

\subsubsection{Relation between glued decoration conditioned with large first subtree and conditioned multitype BGW}

Recall from Lemma \ref{lem:decoration}, that we can view a glued decoration as a MBGW tree. 
recall also from Subsection \ref{subsectionMultitypeEnconding} that such a multitype tree, can be coded by a set of \emph{discrete random fields}. 
In the following result, we show how a conditioned glued decoration can be coded by conditioned discrete random fields.

For the next result, we use $\vec T(\vec \epsilon_{i_0})=(T^1(\vec \epsilon_{i_0}),\dotsm, T^d(\vec \epsilon_{i_0}))$, where $\vec T(-)$ is defined in \eqref{eqn:minSolDiscrete} and recall that $\vec \epsilon_{i_0}$ is the unit vector with $i_0^{\text{th}}$ entry equal to one.
The following corollary of Lemma \ref{lem:spalfAndDiscTree} will be useful when taking limits.

\begin{cor}[Coupling the conditioned glued decoration with conditioned discrete random fields]\label{cor:DecorationSpalf}
Fix an $i_0\in[d]$, an offspring distribution $\vec \varphi$, and the values $\vec{r}\in \Z^d_+$ with $r_{i_0}\geq 1$, $r\in \mathbb{N}$.  Suppose that $\sG(\cD_{r})\sim M^{i_0,\vec \varphi}(-|\#\cD(\emptyset)\ge r).$ Then we can couple the glued decoration $\sG(\cD_r)$ with the walks $(\vec X^{j}(m);m\leq T^j(\vec \epsilon_{i_0})),j\in[d])$, conditioned on $\tau^{i_0}(1) \geq r$, such that the subdecorations $(\cD(\bp \ell); |\bp \ell| = h, \text{ with } \ell\equiv j\mod d)$ are coded by excursions of $\vec X^{j}$ which occur between $U^{(h-1),j}$ and $U^{(h),j}$ for any $h\geq 1$, where $(U^{(h),j};j\in[d],h\ge 0)$ solves \eqref{eqn:Rrecurse2} with initial condition $\vec{r} = \vec{\epsilon}_{i_0}$. 
\end{cor}

\begin{remark}\label{rmrkCorollaryCouplingGluedAndRWs}
From the construction above, it is clear that we can also couple the conditioned decoration $\sG(\cD_{r})$, with two pairs of random walks as follows. 
The first random walk  $\big(\vec X^{i_0}_0(m); m = 0,1,\dotsm,\tau^{i_0}(1)\big)|\{\tau^{i_0}(1)\geq r\}$ will code the subtree of type $i_0$ connected to the root, together with the $X^{i_0,j}_0(\tau^{i_0}(1))$ subtrees of type $j\neq i_0$ growing from it. 
Now, consider an independent set of random walks $(\vec X^{j};j\in[d])$ starting at zero. 
Conditionally on $\big(\vec X^{i_0}_0(m); m = 0,1,\dotsm,\tau^{i_0}(1)\big)|\tau^{i_0}(1)\geq r$, for every $j\in [d]$ define 
\[
\widetilde{X}^j_n(-) = (X^{i_0,j}_0(\tau^{i_0}(1)) \vee 0) \vec{\epsilon}_j + \vec X^j(-).
\]Those random walks will code the forest growing from $\cD(\emptyset)$ if stopped at $\vec T(\vec {\bo 0})$ (the multidimensional first hitting time defined in \eqref{eqn:minSolDiscrete}). 
That is, all subdecorations $(\cD(\bp ); |\bp | \geq 1)$ are coded by $(\vec X^{j}(m);m\leq T^j(\vec{\bo 0}),j\in[d])$. 
This coupling will be used in Section \ref{sec:convergence}.
\end{remark}

\section{Multitype Decorations of the Ulam Tree}\label{sec:DecorationOfUlamTree}

\subsection{Preliminaries on Metric Measure Spaces}\label{ssec:GHP}

All metric measure spaces we consider in the sequel will be rooted (or pointed) at a point $\varrho\in \XX$, will be compact and equipped with a finite measure (or vector-valued measure). 
We will call pointed metric measure spaces (PMM) spaces to 4-tuples $\XX = (\XX,\varrho,\di,\mu)$. 
Thus, PMM spaces that we consider are assumed to be compact. We say that two PMM spaces $\XX = (\XX,\varrho,\di,\mu)$ and $\XX' = (\XX',\varrho',\di',\mu')$ are isomorphic if there exists an isometry $f:\XX\to \XX'$ such that $f(\varrho) = \varrho'$ and is measure preserving $f_\#\mu = \mu'$. Let $\M$ denote the collection of equivalence classes of PMM spaces.

It is now well-known that one can equip $\M$ with the Gromov-Hausdorff-Prohorov (GHP) topology. This topology is metrizable with the metric $\dghp$ and this turns $\M$ into a Polish space. See \cite{Khezeli.20} for more information. This topology extends the Gromov-Hausdorff distance, but we will define this distance in reference to the GHP metric. Define $\M_0$ as the space of PMM spaces equipped with the null measure. The space $\M_{0}\subset\M$ is equipped with the subspace topology which is simply Gromov-Hausdorff topology. Moreover it is a \textit{closed} subset. 
Note that there is a natural surjective contraction $\FORGET:\M\to \M_0$ sending a PMM $\XX = (\XX,\varrho,d,\mu)\mapsto \FORGET(\XX) = (\XX,\varrho,d, 0).$ We denote this metric on $\M_0$ by $\dgh$. We recall each of these below, but refer \cite{Khezeli.20} for more information.

\subsubsection{Hausdorff and Prohorov distances}

Suppose that $(E,d)$ is an arbitrary metric space. One can topologize the collection $F:=\{A:A\subseteq E\text{ is compact}\}$ with the Hausdorff distance:
\begin{equation*}
d^{E}_{\textup{H}}(A,B) = \inf\{\eps>0: A\subset B^\eps \text{ and }B\subset A^\eps\},\qquad A^\eps:=\bigcup_{x\in A} \{y\in A : d(x,y)<\eps\},
\end{equation*}for any $A,B\subseteq F$. 
The Hausdorff distance turns $F$ into a metric space. 

If $(E,d)$ is a Polish space and $\mu,\theta$ are two finite Borel measures on $E$, we can define the Prohorov distance between $\mu$ and $\theta$ as
\begin{equation*}
d_{\textup{P}}^E(\mu,\theta) = \inf\{\eps>0: \mu(A)\le \theta(A^\eps)+\eps\text{ and }\theta(A)\le \mu(A^\eps)+\eps \quad \forall A\text{ Borel}\}.
\end{equation*} The Prohorov distance turns the collection of finite (Borel) measures on $E$ into a metric space.

\subsubsection{Gromov-Hausdorff-Prohorov Topology} \label{sssec:GHP}

Given two elements $\XX,\YY\in \M$, if both live in some ambient Polish space $\Zz$, then one can can talk about the Hausdorff distance between $\XX$ and $\YY$ as subsets of $\Zz$ and about the Prohorov distance between their two measures. The (compact) GHP topology allows us to do this even when $\XX$ and $\YY$ are not in the same ambient space: the GHP metric $\dghp$ is defined as
\begin{equation*}
\dghp(\XX,\YY) = \inf\left\{d_{\textup{H}}^{\Zz}(f(\XX),g(\YY)) + d_{P}^{\Zz}(f_\#\mu_\XX, g_\#\mu_{\YY}) + d_\Zz(f(\varrho_\XX),g(\varrho_\YY))\right\}
\end{equation*}
where the infimum is taken over all compact metric spaces $(\Zz,\di_\Zz)$ and isometric embeddings $f:\XX\to \Zz$ and $g:\YY\to \Zz$. The $\dghp$ is a metric on $\M$.

If the PMMs $\XX$ and $\YY$ have null measures, then we recover simply the Gromov-Hausdorff distance.
The following lemma can be seen in \cite{Khezeli.20}.

\begin{lem}
$(\M,\dghp)$ and $(\M_{0},\dgh)$ are Polish.
\end{lem}

\subsubsection{Infinitely Marked PMM Spaces}\label{subsubsectionIMPMM}

We are concerned about glued metric spaces, and so we need to consider how to include infinitely many marked points on a metric space, just as in Construction \ref{cons:continuum2}. This is done by considering \textit{infinitely marked pointed metric measure spaces}. Here the ``pointed'' refers to the root while the marks correspond to additional points which are vital for the gluing procedure. 

Consider a 5-tuple $(\XX,\varrho,\di,\mu,(x_\ell)_{\ell\ge 1})$ where $(\XX,\varrho,\di,\mu)\in \M$ and $(x_\ell)_{\ell\ge 1}$ is an $\XX$-valued sequence. We say that two such $5$-tuples $\XX$ and $\XX'$ are equivalent if there exists an isometry $f:\XX\to \XX'$ such that $f(\varrho) = \varrho'$ and $f(x_\ell) = x_\ell'$ for all $\ell\ge 1$ and is measure preserving $f_\#\mu = \mu'$. 
For any $k\in \mathbb{N}$ we denote by $\mathfrak{M}^{(k)}_{\tiny \mbox{mp}}$ the subcollection of those $\XX$ such that $x_\ell = \varrho$ for all $\ell>k$. 
We identify $\mathfrak{M}^{(0)}_{\tiny \mbox{mp}} = \mathfrak{M}$ and $\fK:=\mathfrak{M}^{(\infty)}_{\tiny \mbox{mp}}$. 
Let $\mathfrak{M}_{\tiny \mbox{mp},0}$ (resp. $\mathfrak{M}^{(k)}_{\tiny \mbox{mp},0}$) denote the collection of elements in $\fK$ (resp. $\mathfrak{M}^{(k)}_{\tiny \mbox{mp}}$) with null measure. We call elements in $\mathfrak{M}_{\tiny \mbox{mp}}^{(k)}$ a PMM with $k$ marked points.

Let us write $\dghp^{(k)}$ and $\dgh^{(k)}$ as
\begin{equation}\label{eqnDefinitionGHPKMArkedSpace}
\begin{split}
\qquad \dgh^{(k)}(\XX,\YY) &= \inf\left\{d_{\textup{H}}^{\Zz}(f(\XX),g(\YY))  + d_\Zz(f(\varrho_\XX),g(\varrho_\YY))\vee \max_{\ell\le k} d_\Zz(f(x_\ell),g(y_\ell)) \right\}\\
\dghp^{(k)}(\XX,\YY) &= \inf\left\{d_{\textup{H}}^{\Zz}(f(\XX),g(\YY)) + d_{P}^{\Zz}(f_\#\mu_\XX, g_\#\mu_{\YY}) + d_\Zz(f(\varrho_\XX),g(\varrho_\YY))\vee \max_{\ell\le k} d_\Zz(f(x_\ell),g(y_\ell)) \right\},
\end{split}
\end{equation}where again the infimum is over all isometric embeddings $f:\XX\to \mathcal Z$ and $g:\YY\to \mathcal Z$ into a common metric space $(\Zz,\di_\Zz)$, and where the notation for the parameters of $\YY$ should be obvious. 
The GH metric on $\mathfrak{M}_{\tiny \mbox{mp},0}^{(k)}$, i.e. $\dgh^{(k)}$ above, was used in \cite{Miermont.09}. It is elementary to see that $\mathfrak{M}_{\tiny \mbox{mp},0}^{(k)}$ is Polish with respect to this metric $\dgh^{(k)}$. It is also not hard to see that using the GHP topology on PMM spaces with $k$ marked points 
makes $\mathfrak{M}^{(k)}_{\tiny \mbox{mp}}$ a Polish space with metric $\dghp^{(k)}$. 
Also observe that both functions above are well-defined for any two elements in $\fK$, but it no longer separates points and, hence, is not a metric.

For two elements $\XX,\YY\in \fK$, we define $\dghpI$ as
\begin{equation*}
\dghpI(\XX,\YY) = \sum_{k=1}^\infty 2^{-k}\left(\dghp^{(k)}(\XX,\YY) \wedge 1\right).
\end{equation*}Define similarly $\dghI$. 
The following lemma is not difficult to establish, and its proof is omitted.
\begin{lem}
$(\fK,\dghpI)$ and $(\mathfrak{M}_{\tiny \mbox{mp},0},\dghI)$ are Polish spaces.
\end{lem}

\subsubsection{Extension to vector-valued measures}

We wish to extend the GHP topology on $\fK$ to metric spaces equipped with $\R_+^d$-valued measures for $d\geq 2$, which are (essentially) used in \cite{ABBGM.17} in the case where $d = 2$. We do this as follows. For any $k = 0,1,\dotsm,\infty$ , we let $\mathfrak{M}_{\tiny \mbox{mp}}^{(k),d}$ denote the $d$-fold product space $\prod_{j=1}^d \mathfrak{M}^{(k)}_{\tiny \mbox{mp}}$. Elements of $\mathfrak{M}^{(k),d}_{\tiny \mbox{mp}}$ are of the form $(\XX_1,\dotms, \XX_d)$ where $\XX_i\in \mathfrak{M}_{\tiny \mbox{mp}}^{(k)}$ for every $i\in [d]$. We can define the sum metric:
\begin{equation*}\label{eqnDistanceGHPKMarksMeasured}
\dghp^{(k),d}((\XX_1,\dotsm,\XX_d), (\YY_1,\dotsm, \YY_d)) = \sum_{j=1}^d \dghp^{(k)}(\XX_j,\YY_j) \qquad k = 0, 1,2,\dotsm, \infty.
\end{equation*} The metric $\dghp^{(k),d}$ turns $\mathfrak{M}_{\tiny \mbox{mp}}^{(k),d}$ into a Polish space. We let $\mathfrak{M}_{\tiny \mbox{mp}}^{(k),\ast d}\subset\mathfrak{M}_{\tiny \mbox{mp}}^{(k),d}$ denote the subset of all $d$-tuples $(\XX_1,\dotsm, \XX_d)$ such that $\dgh^{(k)}(\XX_i,\XX_j) = 0$ for all $i,j$. That is, as marked PMM spaces $\XX_1,\dotsm, \XX_d$ are all equivalent so the only possible difference lies on their measures. It is not hard to see that $\mathfrak{M}_{\tiny \mbox{mp}}^{(k),\ast d}$ is a closed subset of $\mathfrak{M}_{\tiny \mbox{mp}}^{(k),d}$ for all $k$. 
\begin{lem}
For any $k \in \{0,1,\dotsm,\infty\}$, the spaces $(\mathfrak{M}_{\tiny \mbox{mp}}^{(k),d},\dghp^{(k),d})$ are Polish. 
\end{lem}

When $k = \infty$, we will write $\mathfrak{M}^{\ast d}_{\tiny \mbox{mp}}$ instead of $\mathfrak{M}_{\tiny \mbox{mp}}^{(\infty),\ast d}$ and when $k = 0$ we will write $\fMd$ instead of $\mathfrak{M}_{\tiny \mbox{mp}}^{(0),\ast d}$. We view the element $\XX = (\XX_1,\dotsm, \XX_d)\in \mathfrak{M}^{\ast d}_{\tiny \mbox{mp}}$ as the 5-tuple $\XX=(\XX, \varrho,\di, \vec{\mu},(x_\ell)_{\ell\ge 1})$ 
where $\vec{\mu} = \sum_{j=1}^d \mu_j \vec{\epsilon}_j$ is a vector-valued measure on $\XX$ {and each $\mu_j$ is the measure on $\XX_j$}. Note that for all $i\in [d]$ the first $k$ entries of $(x_\ell^i;\ell\ge 1)$ for the marks on $\XX_i$ all agree so this gives a well-defined element in $\mathfrak{M}_{\tiny \mbox{mp}}^{(k),\ast d}$ no matter how we select the entries $(x_\ell)_{\ell>k}$.

Also, by abuse of notation, we will consider finitely pointed spaces $(\mathfrak{M}_{\tiny \mbox{mp}}^{(k),d},\dghp^{(k),d})$ as elements of $(\mathfrak{M}_{\tiny \mbox{mp}}^{*d},\dghp^{(\infty),d})$ by arbitrarily extending the sequence $(x_\ell)_{\ell\leq k}$ to $x_\ell=\varrho$ for $\ell>k$.

\subsection{Gluing Metric Spaces Along the Ulam Tree} \label{sssec:metricDecorations}


 Following \cite{Senizergues.20}, any function $f:\bbU \to E$ is called an $E$-valued decoration of $\bbU$.
Consider a decoration $\cD:\bbU\to \mathfrak{M}^{\ast d}_{\tiny \mbox{mp}}$, where for each $\bp\in \bbU$ 
\begin{equation*}
\cD(\bp): = \left(D_\bp, \di_\bp , \rho_\bp,\vec{\mu}_\bp, (x_{\bp;\ell})_{\ell\ge 1}\right).
\end{equation*}Recall the gluing operation $\sG$ and the subglued decoration $\sG^*(\cD)$ from Definition \ref{definitionDecorationGluedMetricspaceGluedDecoration}. 
One can easily define a distance $\di$ on $\sG^*(\cD)$ as follows. Suppose that $z\in D_\bp$ and $y\in D_\bq$ {for some $\bp,\bq\in \bbU$ }. Then 
\setlength{\leftmargini}{.15in}
\begin{itemize}
\item if $\bp = \bq$, then $\di(z,y) = \di_\bp(z,y)$;
\item if $\bp\prec \bq$ with $\bp = (p_1,\dotsm, p_n)$ and $\bq = (p_1,\dotsm, p_n,q_{n+1},\dotsm, q_{n+k})$, {for $\ell=1,\dotsm,k-1$ }set $\bq^{(\ell)} = \bp(q_{n+1},\dotsm, q_{n+\ell})$
\begin{equation*}
    \di(z,y)=\di(y,z) := \di_{\bp}(z,x_{\bp;q_{n+1}}) + \sum_{\ell=1}^{k-1} \di_{\bq^{(\ell)}} (\rho_{\bq^{(\ell)}},x_{\bq^{(\ell)}; q_{n+\ell+1}}) + \di_{\bq}(\rho_\bq, y);
\end{equation*}
\item if $\mathbf{w} = \bp\wedge \bq$ is the least common ancestor of $\bp$ and $\bq$, with $\bp = \mathbf{w}(\ell_1,\ell_2,\dotms, \ell_k)$ and $\bq = \mathbf{w} (r_1,\dotsm, r_{m})$ for some $k, m\ge 1$ then
\begin{equation*}
    \di(z,y) = \di(y,z) := \di_{\mathbf{w}}(x_{\mathbf{w}; \ell _1},x_{\mathbf{w};r_1}) + \di(x_{\mathbf{w};\ell_1},z)+\di(x_{\mathbf{w};r_1},y).
\end{equation*}
\end{itemize}
We then define $\sG(\cD)$ as the Cauchy completion of $\sG^*(\cD)$, i.e. $\sG(\cD) = \overline{\sG^*(\cD)}.$ 
A priori, this space is just an arbitrary complete metric space and so it need not be compact. Similarly, it need not come equipped with a finite vector-valued measures $\vec{\mu}$. If it is compact and its measure is finite, then $\sG(\cD)\in \mathfrak{M}^{\ast d}$.

\subsection{Construction in an Embedded Space}
In order to prove our general topological results, we introduce now some formal topological preliminaries following \cite{Senizergues.20}. 
The objective is to isometrically embed all the (equivalence classses of the) subdecorations $\cD(\bp)$ into a common space, and construct there the glued decoration.

Let us consider the Urysohn space $(\U,\delta)$ with a fixed point $\ast\in \U$. This is a metric space (with metric $\delta$) which is the unique Polish space (up-to isomorphism) which has the following extension property: given any finite metric space $(X,\di)$ and any $x\in X$, any isometry $\phi:X\setminus\{x\}\to \U$ can be extended to an isometry from the entire space $X$ to $\U$. This therefore extends to all separable metric spaces and hence any compact metric space. Moreover, if $(K,\di,\rho)$ is a compact metric space, the isometry $\phi:K\to \U$ can be chosen such that $\phi(\rho) = \ast$ (or any element of $\U$). In particular, for every decoration $\cD$ and every $\bp\in \bbU$, the infinitely marked metric measure space $\cD(\bp) = (D_\bp,\di_\bp, \rho_\bp, \vec{\mu}_\bp, (x_{\bp;\ell})_{\ell\ge 1})$ can be embedded isometrically into the space $\U$, i.e. there is some isometric embedding $\phi_\bp:D_\bp\to \U$ with $\phi_\bp(\rho_\bp) = \ast$. If $\vec{\mu}_\bp$ is a $d$-dimensional measure, then clearly  there exists an infinitely marked metric measure space $\XX_\bp\subset \U$ which is equivalent to $\cD(\bp)$ (i.e.  $\dghp^{(\infty),d}(\cD(\bp),\XX_\bp) = 0$), defined by
\begin{align*}
\XX_\bp = (\phi_\bp(D_\bp), \delta, \ast, (\phi_\bp)_{\#}\vec{\mu}_\bp, (\phi_\bp(x_{\bp;\ell}))_{\ell\ge 1}).
\end{align*}

We refer the reader to Section $3.11\frac{2}{3}{}_+$ in \cite{Gromov.99} as well as to \cite{Huvsek.08,Melleray.07} and references therein for more information about the properties of the Urysohn space. In particular, we would like to recall the following claim (c.f. \cite[Corollary 1.2]{Melleray.07}, Exercise (b') in \cite[pg. 82]{Gromov.99}.
\begin{claim}\label{claim:isometry1}
If $K,L$ are compact subsets of $\U$ and $\phi:K\to L$ is an isometry, then there exists an extension $\widetilde{\phi}:\U\to \U$ such that $\widetilde{\phi}|_K = \phi$
\end{claim}

As pointed out in \cite[Section 2.3.2]{Senizergues.20}, there are some measure-theoretic subtleties about equivalence classes of metric spaces involved with the procedure above. To avoid these, we define
$\mathfrak{M}^{\ast d}_{\tiny \mbox{mp}}(\U)$ as the subset of elements $(K,\delta_{|K},\ast,\mu,(\rho_\ell)_{\ell\geq 1})$, where $K$ is a compact subset of $\U$, we have $*,\rho_\ell\in K\subset \U$ for all $\ell\geq 1$, also $\delta_{|K}$ is the distance on $\U$ restricted to $K$, and the measure $\mu$ is a $d$-dimensional valued measure. 
The latter is the analogue of $\mathfrak{M}^{\ast d}_{\tiny \mbox{mp}}$, but in the Urysohn space.
We thus define the projection map $\pi:\mathfrak{M}^{\ast d}_{\tiny \mbox{mp}}(\U)\to \mathfrak{M}^{\ast d}_{\tiny \mbox{mp}}$, which assigns to each element $(K,\delta_{|K},\ast,\mu,(\rho_\ell)_{\ell\geq 1})$ in $\mathfrak{M}^{\ast d}_{\tiny \mbox{mp}}(\U)$, the corresponding equivalence class, say $[(K,\delta_{|K},\ast,\mu,(\rho_\ell)_{\ell\geq 1})]$ in $\mathfrak{M}^{\ast d}_{\tiny \mbox{mp}}$. 
Since we are only dealing with compact spaces, then this map is surjective and hence, we can lift any deterministic element of $\mathfrak{M}^{\ast d}_{\tiny \mbox{mp}}$ to an element of $\mathfrak{M}^{\ast d}_{\tiny \mbox{mp}}(\U)$. 
Now, since we are dealing with random elements (equivalence classes of subdecorations) taking values on $\mathfrak{M}^{\ast d}_{\tiny \mbox{mp}}$, we must ensure that we can consider versions of those random subdecorations that take values on $\mathfrak{M}^{\ast d}_{\tiny \mbox{mp}}(\U)$.
This is done by using a result of \cite{MR330393}, which implies that for our Polish spaces, every probability distribution $\tau$ on $\mathfrak{M}^{\ast d}_{\tiny \mbox{mp}}$ can be lifted to a probability distribution $\sigma$ on $\mathfrak{M}^{\ast d}_{\tiny \mbox{mp}}(\U)$, such that $\pi_*\sigma$ is equal to $\tau$. 
The previous ensures that we can always work with  versions of our subdecorations, embedded in $\U$, whose root coincides with $\ast$, and whose distribution coincides with the law of the (original) subdecoration in $\mathfrak{M}^{\ast d}_{\tiny \mbox{mp}}$. 
In the following, we will be loose on the fact that we are considering equivalence classes, as well as elements on $\U$.

Consider the ``sequence space'' $\U^{\bbU}$ consisting of all sequences $(y_\bp)_{\bp\in \bbU}$ of $\bbU$-indexed and $\U$-valued sequences. While this space $\U^{\bbU}$ is too large, we can define the space
\begin{equation*}
\ell^1(\bbU\to \U) := \left\{(y_\bp)_{\bp\in \bbU}\in \U^{\bbU}: \sum_{\bp\in \bbU} \delta(\ast,y_\bp) <\infty\right\}
\end{equation*}

\begin{figure}[t]
\centering
\includegraphics[width=0.7\linewidth]{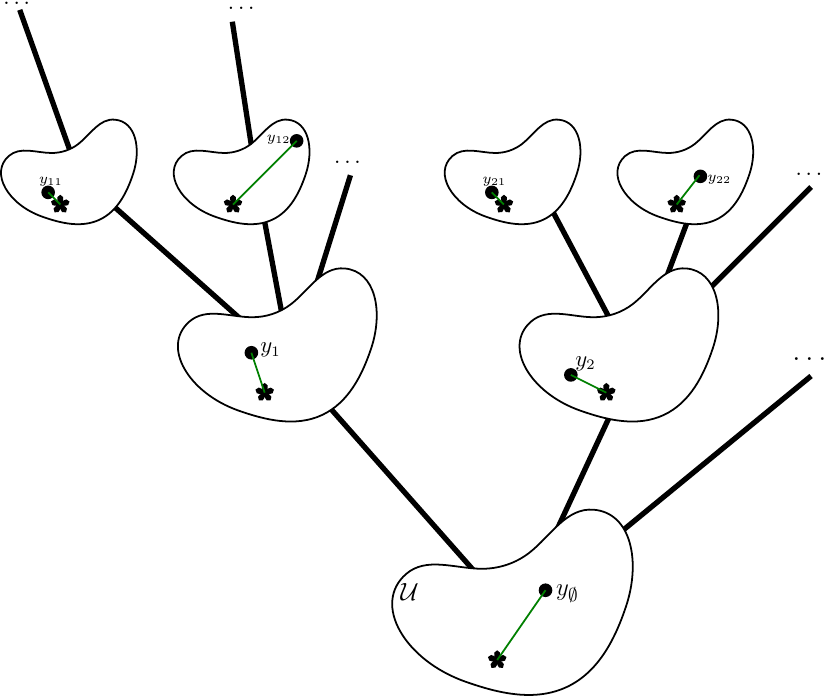}
\caption{Representation of $\ell^1(\bbU\to \U)$. The green lines represent the distance between the point $y_\bp$ and the root $\ast$. The thick black lines represent the edges in the Ulam-Harris tree $\bbU$. 
}
\label{fig:croppedl1}
\end{figure}

See Figure \ref{fig:croppedl1} for a representation of $\ell^1(\bbU\to\U)$. 
In words, $\ell^1(\bbU\to \U)$ is the collection of all $\bbU$-indexed collections of elements in $\U$ whose distance to some fixed point $\ast$ is summable. It is easy to see that $\ell^1(\bbU\to \U)$ is a Polish space with metric given by \begin{equation}\label{eqnDistanceOnTheSpaceEllbbUToCalU}
d_{\ell^1}((y_\bp), (z_\bp)): = \sum_{\bp\in \bbU} \delta(y_\bp,z_\bp)<\infty.
\end{equation} The following observation is helpful to know, its proof is trivial and omitted. 
\begin{claim}\label{claim:isometry2}
If for each $\bp\in \bbU$ we have some isometry $ \phi_\bp:\U\to \U$ which fixes $\ast$, then the product map,
\begin{equation*}
\phi = \prod_{\bp\in \bbU} \phi_\bp : \prod_{\bp\in \bbU} \U\to \prod_{\bp\in \bbU}\U\qquad \text{by}\qquad\phi\left((y_\bp)_{\bp\in \bbU}\right) = ( \phi_\bp(y_\bp))_{\bp\in \bbU}
\end{equation*}
is an isometry on $\ell^1(\bbU\to \U)$.
\end{claim}

As described above, for any decoration $\cD$ and any ${\bp}\in \U$ , we can (and do) take 
\begin{equation}\label{eqnSubdecorationDefinedOntheUryshonSpace}
\cD(\bp) = (D_\bp, \di_\bp, \rho_\bp, \vec{\mu}_\bp, (x_{\bp;\ell})_{\ell\ge 1}) = \left(D_\bp , \delta|_{D_\bp\times D_\bp}, \ast, \vec{\mu}_\bp|_{D_\bp} , (x_{\bp;\ell})_{\ell\ge 1} \right)
\end{equation}
where $D_\bp\subset \U$ and $\vec{\mu}_\bp$ is a measure on $\U$ supported on $D_\bp$. This allows the gluing operation to be defined on the product space $\prod_{u\in \bbU}\U$ instead of abstractly. Following \cite{Senizergues.20}, we define for each $\bp = (p_1,p_2,\dotsm, p_n)$ and with $\bp(h) := (p_1,\dotsm, p_h)$ for $h=1,\dotsm, n$, the space
\begin{equation*}
\widetilde{D}_\bp = \left\{(y_\bq)_{\bq\in \bbU}\big| y_{\bp(h-1)} = x_{\bp(h-1);p_h} \text{ for }h\in [n],  y_\bp\in D_\bp, \text{ and }y_\bq = \ast \text{ if }\bq\not\preceq \bp   \right\}.
\end{equation*} 
See Figure \ref{fig:embeddinginl1cropped} for a representation. It is easy to see that $\widetilde{D}_\bp\subset \ell^{1}(\bbU\to \U)$ and is isometric to $D_\bp$. Indeed, for any two points $(y_\bq)_{\bq\in \bbU}$ and $(y'_\bq)_{\bq\in \bbU}$ in $\widetilde{D}_\bp$ we have (by definition)
\begin{align*}
y_\bq = y_\bq'\qquad\textup{ for all }\bq\neq \bp\qquad\textup{ and }\qquad y_\bp, y'_\bp\in D_\bp\subset\U.
\end{align*} Hence $d_{\ell^1}((y_\bq),(y_\bq')) = \delta(y_\bp,y_\bp')$ is precisely the (subspace) metric on $D_{\bp}$. Thus, the subglued decoration and glued decoration
\begin{equation*}
\sG^*(\cD) := \bigcup_{\bp\in \bbU} \widetilde{D}_\bp\subset \ell^{1}(\bbU\to \U),\qquad \sG(\cD) := \overline{ \sG^*(\cD)}\subset\ell^{1}(\bbU\to \U),
\end{equation*}as in Definition \ref{definitionDecorationGluedMetricspaceGluedDecoration} (recall that they come with a distance, a root and a measure).

Observe that the gluing operations as a subset of $\ell^1(\bbU\to \U)$ depend on the isometric embedding of $D_\bp\hookrightarrow \U$. However, the metric spaces $\sG^*(\cD)$ and $\sG(\cD)$ are unique up-to an isometry appearing the maps in Claim \ref{claim:isometry2} where each $\phi_\bp$ is as in Claim \ref{claim:isometry1}.

\begin{figure}
\centering
\includegraphics[width=0.7\linewidth]{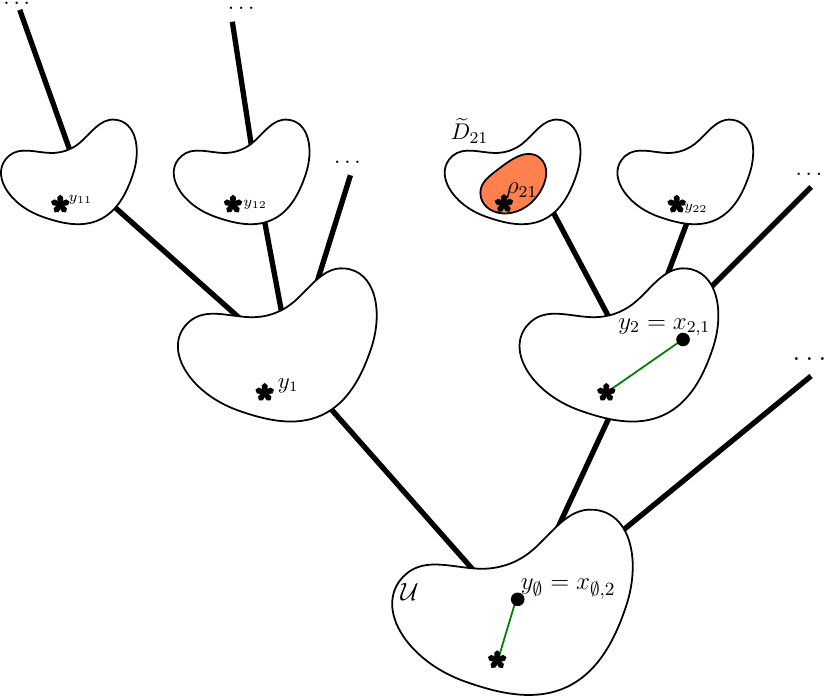}
\caption{The embedded subdecorated metric space $\widetilde{D}_{21}$. The orange blob in the copy of $\U$ is the embedded version of $D_{21}\subset U\hookrightarrow \ell^1(\bbU\to \U)$. }
\label{fig:embeddinginl1cropped}
\end{figure}

We view both spaces $\sG^*(\cD)$ and $\sG(\cD)$ as metric measure spaces.

\subsection{Bounds on Glued Decorations}

Some particular subsets of $\sG(\cD)$ will be helpful in the proof of Theorem \ref{thm:graphConv1}. For each $\tr\in \bbT$ and any $h\ge 0$, define
\begin{equation*}
	\sG(\tr, \cD) = \bigcup_{\bp\in \tr} \widetilde{D}_\bp\qquad 
	\sG_{\le h}^*(\cD) = \bigcup_{\substack{\bp\in \bbU\\ |\bp|\le h}} \widetilde{D}_\bp\qquad \sG_{\le h} = \overline{\sG_{\le h}^*( \cD)}, 
\end{equation*} for any decoration $\cD$.
These will be
viewed as glued decorations (metric measure spaces) as in Definition \ref{definitionDecorationGluedMetricspaceGluedDecoration}, and contained in $\ell^1(\bbU\to \U)$. Note that since all plane trees are assumed to be finite, then $\sG(\tr,\cD)$ is compact. 
The following lemma is a trivial extension of Lemma 2.4 in \cite{Senizergues.20} to the case of vector-valued measures, its proof is omitted.
\begin{lem}[S\'enizergues {\cite[Lemma 2.4]{Senizergues.20}}]\label{lem:finTreeLem}
	For any $\tr\in \bbT$, and any two $\mathfrak{M}^{\ast d}_{\tiny \mbox{mp}}$-valued decorations $\cD$ and $\cD'$
	\begin{equation*}
		\dghp(\sG(\tr,\cD), \sG(\tr,\cD')) \le 2\sum_{\bp\in \tr} 
\dghp^{(L_\tr(\bp)),d} (\cD(\bp),\cD'(\bp))
   \end{equation*}
	where $L_\tr(\bp) := \max\{\ell \in \N : \bp\ell \in \tr\}$ (setting it to 0 if $\bp$ is a leaf).
\end{lem}

This section is devoted to the proof of Theorem \ref{thm:graphConv1}, about the convergence of the glued decorations. The idea is the following: if $\cD$ satisfies Assumption \ref{ass:uniformBound} $\sG(\cD)$ then in the GHP topology we can approximate it by $\sG_{\le h}(\cD)$ and this, in turn, can be approximated by $\sG(\tr,\cD)$. This last is depicted in Figure \ref{fig:gtd-approx}.
\begin{figure}
	\centering
	\includegraphics[width=0.7\linewidth]{"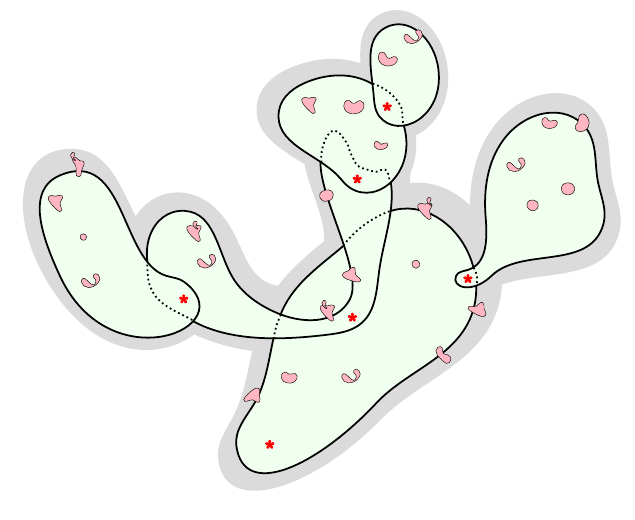"}
	\caption{Approximation of $\sG(\cD)$ by $\sG(\tr,\cD)$. The blobs in green are the only subdecorations large enough to significantly contribute to the mass or geometry in the limit, and the pink blobs are some of the subdecorations in $\bbU\setminus \tr$. The red dots are the roots of the green subdecorations. The gray region is an $\eps$-neighborhood. }
	\label{fig:gtd-approx}
\end{figure}
 Similarly, under \ref{ass:2thm1}--\ref{ass:3thm1} we can do the same approximations for $\cD^n$ uniformly for all $n$ large. Finally, the pointwise convergence in \ref{ass:1thm1} shows that $\sG(\tr,\cD^n)$ is arbitrarily close to $\sG(\tr,\cD)$ for all trees $\tr$ by the previous lemma.


We recall that for any metric space $(X,d)$, any point $x\in X$ and subset $A\subseteq X$, we have $d(x,A):=\inf_{y\in A} d(x,y)$. 

\begin{lem}\label{lem:limitfordecor}
Suppose $(\cD^n;n\in \mathbb{N}), \cD$ satisfy the assumptions of Theorem \ref{thm:graphConv1}. Then
\begin{enumerate}
\item\label{eqnThmGraphConv11} 
$
\displaystyle    \lim_{h\to\infty} \limsup_{n\to\infty} \dghp^d(\sG(\cD^n), \sG_{\le h}(\cD^n)) = 0.
$
\item\label{eqnThmGraphConv12} For all $h\ge 1$ fixed $\displaystyle     \inf_{\tr\in \bbT} \limsup_{n\to\infty} \dghp^d (\sG_{\le h} (\cD^n) , \sG(\tr, \cD^n)) = 0.$
Moreover, the infimum over $\tr$ can be taken over $\tr$ with height at most $h$.
\item\label{eqnThmGraphConv13} For all $\tr\in \bbT$, $  \displaystyle  \limsup_{n\to\infty} \dghp^d (\sG(\tr,\cD^n), \sG(\tr,\cD)) = 0.$
\end{enumerate}
\end{lem}

\begin{proof}
We begin by proving the first limit. Fix an $\eps>0$ and let $h \ge h(\eps)$ and $n_{h,\eps}$ from Assumption \ref{ass:3thm1} be fixed. We assume that all the glued metric spaces we analyze are embedded in $\ell^1(\bbU\to \U)$ as discussed above.

Consider any $n\ge n_{h,\eps}$.
Note that 
as $\sG_{\le h}(\cD^n)\subseteq \sG(\cD^n),$
\begin{equation*}
\di_{H} (\sG(\cD^n), \sG_{\le h}(\cD^n)) = \sup_{y\in \sG(\cD^n)}\di_{\ell^1}(y,\sG_{\le h}(\cD^n)),
\end{equation*}where we recall that the distance on $\ell^1(\bbU\to \U)$ is given in \eqref{eqnDistanceOnTheSpaceEllbbUToCalU}.
Therefore, we can find $y\in \sG(\cD^n)$ such that {for any $n$ big enough}
\begin{equation}\label{eqn:hausdorff1,lemma55}
 \di_{H} (\sG(\cD^n), \sG_{\le h}(\cD^n)) \le \di_{\ell^1}(y,\sG_{\le h}(\cD^n))+\eps.
\end{equation}
By the definition of Cauchy completions $\inf_{y'\in \sG^*(\cD^n)}\di_{\ell^1}(y,y') = 0$
and so we can find $\bp\in \bbU$ and $y'\in \cD^n(\bp)\hookrightarrow \sG^*(\cD^n)$ such that $\di_{\ell^1}(y,y') \le \eps.$ 
Hence, by \eqref{eqn:hausdorff1,lemma55}
\begin{equation}\label{eqn:hausdorff2,lemma55}
\di_{H} (\sG(\cD^n), \sG_{\le h}(\cD^n)) \le \di_{\ell^1}(y',\sG_{\le h}(\cD^n))+2\eps.
\end{equation}
Also, by definition, for all $z\in \sG_{\le h}(\cD^n)$ we have $\di_{\ell^1}(y',\sG_{\le h}(\cD^n)) \le \di_{\ell^1}(y',z)$. 
By Assumption \ref{ass:3thm1}, we can find $\bq\in \bbU$ with $|\bq|\le h$ and $z\in D^n(\bq)\subseteq \sG_{\le h}(\cD^n)$ such that $\di_{\ell^1}(y',z)\le \eps$, which holds trivially when $|\bp |\leq h$ as we can take $\bq=\bp$ and $z=y'$. Thus, by \eqref{eqn:hausdorff2,lemma55}
\begin{equation}\label{eqn:hausdorff3,lemma56}
\di_{H} (\sG(\cD^n), \sG_{\le h}(\cD^n)) \le3\eps.
\end{equation}
For the Prohorov distance, note that as $\sG_{\le h}(\cD^n) \subseteq \sG(\cD^n)$ we have
\begin{equation*}
\di_P(\sG(\cD^n),\sG_{\le h}(\cD^n)) \le \sum_{\bp: |\bp|> h} \|\vec{\mu}_{\bp}^n\|(D^n_\bp) \le \eps
\end{equation*}
where the last inequality is simply \eqref{eqn:massH3}. 
Hence by \eqref{eqn:hausdorff3,lemma56}, for all $n$ large we have
\begin{equation*}
\dghp (\sG(\cD^n), \sG_{\le h}(\cD^n)) \le 4\eps.
\end{equation*}
Taking $\limsup_{n\to\infty}$, $\limsup_{h\to\infty}$ and then $\lim_{\eps\downarrow 0}$ implies the convergence in \eqref{eqnThmGraphConv11}. 

The proof of \eqref{eqnThmGraphConv12}  is similar except we use Assumption \ref{ass:2thm1} instead. 
Indeed, let $h\geq 1$ and $\eps>0$ be fixed but arbitrary. For each $j\le h$, let $I(j,\eps)$ and $n_{j,\eps}$ be as in \ref{ass:2thm1}. Let $n\ge \max_{j\le h}n_{j,\eps}$ be fixed and let 
\begin{equation}\label{eqn:Ihdef}
I^h := \bigcup_{j\le h} I(j,\eps).
\end{equation}As $I^h$ is finite, we can find a tree $\tr\in \bbT$ such that $I^h\subseteq \tr$. Without loss of generality we assume that $\tr$ is the smallest such tree and therefore if $\bp\in \tr$ then $|\bp|\le h$. 
Noting that $\sG(\tr,\cD^n)\subseteq \sG_{\le h}(\cD^n)$, and any path from a vertex $\bq\in \bbU$ with height at most $h$ is connected to $\tr$ by elements $\bp\in \bbU\setminus \tr $ with $|\bp|\leq h$, we have for any $y\in \cD^n_\bq \subset  \sG_{\le h}(\cD^n)$ that the distance $\di_{\ell^1}(y,\sG(\tr,\cD^n))$ is at most
\begin{align*}
\dghp &(\sG_{\le h} (\cD^n), \sG(\tr,\cD^n)) \le \sum_{\bp\notin \tr: |\bp|\le h} \|\vec{\mu}^n_{\bp}\|(D^n_\bp)+\sum_{j\leq h}\sup_{\bp\notin \tr: |\bp|=j}  \operatorname{diam}(D^n_\bp)\\
&\le \sum_{\bp\notin I^h: |\bp|\le h} \|\vec{\mu}^n_{\bp}\|(D^n_\bp)+\sum_{j\leq h}\sup_{\bp\notin  I^h: |\bp|=j}  \operatorname{diam}(D^n_\bp)\\
&\le \sum_{j\le h}\sum_{\bp\notin I(j,\eps):|\bp|=j} \|\vec{\mu}^n_{\bp}\|(D^n_\bp)+\sum_{j\le h} \sup_{\bp\notin I(j,\eps):|\bp|=j} \operatorname{diam}(D^n_\bp)\\
& \le 2h \eps.
\end{align*}
Here we used in the last inequality Assumption \ref{ass:2thm1}. Therefore, for all $\eps>0$, there exists a finite tree $\tr\in \bbT$ with height at most $h$ such that 
\begin{equation*}
\limsup_{n\to\infty} \dghp (\sG_{\le h} (\cD^n), \sG(\tr,\cD^n)) \le 2h \eps
\end{equation*} where we recall $h$ is fixed. Taking the infimum over $\tr$, and then $\eps\downarrow 0$ gives the desired conclusion. We note that there is no loss of generality in taking the infimum over $\tr$ that are of height at most $h$ instead of all $\bbT$. 

Part 3 is a direct consequence of Lemma \ref{lem:finTreeLem} and Assumption \ref{ass:1thm1}.
\end{proof}

{ It is not hard to check that \ref{ass:1thm1}--\ref{ass:3thm1} hold for $\cD^n\equiv \cD$ whenever $\cD$ satisfies Assumption \ref{ass:uniformBound}. Consequently, we get the following corollary.}

\begin{cor}\label{cor:limittreefornicedecoration}
Suppose $\cD$ satisfies Assumption \ref{ass:uniformBound}, then $\inf_{\tr\in \bbT} \dghp (\sG(\tr,\cD), \sG(\cD)) = 0.$
\end{cor}

\subsection{Convergence of glued decorations: Proof of Theorem \ref{thm:graphConv1}}

\begin{proof}[Proof of Theorem \ref{thm:graphConv1}] 

Suppose, for the moment, that $\cD$ satisfies Assumption \ref{ass:uniformBound}. Then, using the triangle inequality for any tree $\tr\in \bbT$ and any $n,h\geq 1$, we have
\begin{align}\label{eqn:triangle1}
\dghp^d (\sG(\cD^n), \sG(\cD))\leq \dghp^d &(\sG(\cD^n), \sG_{\leq h}(\cD^n))+\dghp^d (\sG_{\leq h}(\cD^n), \sG(\tr,\cD^n))\\
\nonumber&\hspace{.5cm}+\dghp^d (\sG(\tr,\cD^n), \sG(\tr,\cD))+\dghp^d (\sG(\tr,\cD), \sG(\cD)). 
\end{align}

Fix $\eps>0$. By Lemma \ref{lem:limitfordecor}\eqref{eqnThmGraphConv11}, there exists $h_1 = h_1(\eps)$ such that for all $n$ large $\dghp^d (\sG(\cD^n),\sG_{\le h_1}(\cD^n))\le \eps$. Since for each fixed $n$, the mapping  $h\mapsto \dghp^d(\sG(\cD^n), \sG_{\le h}(\cD^n))$ is decreasing we have
\begin{align*}
\limsup_{n\to \infty} \dghp^d (\sG(\cD^n),\sG_{\le h}(\cD^n))\le \eps\qquad\textup{ for all }h\ge h_1.
\end{align*}
From Corollary \ref{cor:limittreefornicedecoration}, we can find a tree $\tr_1 = \tr_1(\eps)$ such that $\limsup_{n\to\infty} \dghp^d(\sG(\tr_1,\cD),\sG(\cD)) \le \eps$. Again, by monotonicity properties we conclude
\begin{align*}
\limsup_{n\to\infty} \dghp^d (\sG(\cD^n),\sG_{\le h}(\cD^n)) + \dghp^d(\sG(\tr,\cD),\sG(\cD)) \le 2\eps \qquad\textup{ for all }h\ge h_1, \tr\supset \tr_1.
\end{align*}
Let us set $h_2 = h_2(\eps) =h_1(\eps)\vee  \max\{|\bp| : \bp\in \tr_1(\eps)\}$. 

Recall from Lemma \ref{lem:limitfordecor}\eqref{eqnThmGraphConv12}, for $h_2$ above we can find a tree $\tr_2=\tr_2(\eps)\supset \tr_1(\eps)$ with height at most $h_2$ such that for all $n$ large $\dghp^d(\sG_{\le h_2}(\cD^n), \sG(\tr_2,\cD^n))\le \eps$. In particular, 
\begin{align*}
\limsup_{n\to\infty} \dghp^d (\sG(\cD^n),\sG_{\le h_2}(\cD^n)) + \dghp^d(\sG(\tr_2,\cD),\sG(\cD))+ \dghp^d\left(\sG_{\le h_2}(\cD^n), \sG(\tr_2,\cD^n)\right) \le 3\eps.
\end{align*}

By using Lemma \ref{lem:limitfordecor}\eqref{eqnThmGraphConv13}, 
\begin{align*}
\lim_{n\to\infty}\dghp^d (\sG(\tr_2,\cD^n), \sG(\tr_2,\cD)) =0.
\end{align*}
Consequently, by \eqref{eqn:triangle1} with $h = h_2, \tr = \tr_2$,
\begin{align*}
\limsup_{n\to\infty}\dghp^d (\sG(\cD^n), \sG(\cD)) \le 3\eps\qquad\textup{ for all }\eps>0.
\end{align*}

Therefore, the proof of Theorem \ref{thm:graphConv1} follows once we show that $\cD$ satisfies Assumption \ref{ass:uniformBound}.
To establish Assumption \ref{ass:uniformBound}(1) we fix an $\eps>0$ and let $h=h(\eps)$ be as in Assumption \ref{ass:3thm1}. Let $I^h$ be as in \eqref{eqn:Ihdef} for this fixed $h$. 
Since the size of all subdecorations above $h$ is bounded by \eqref{eqn:massH3} in \ref{ass:3thm1}, and the size of the subdecorations in each of the generations below generation $h$ but outside $I^h$ is bounded by \ref{ass:2thm1}, then for all $n$ large
\begin{equation*}
\sum_{\bp\in \bbU}\|\vec{\mu}^n_{\bp}\|(D^n_\bp) \le \sum_{\bp\in I^h} \|\vec{\mu}^n_{\bp}\|(D^n_\bp) + (h+1)\eps.
\end{equation*}
As $I^h$ is a finite set, by Assumption \ref{ass:1thm1} we can conclude that
\begin{equation*}
\lim_{n\to\infty} \sum_{\bp\in I^h} \|\vec{\mu}^n_{\bp}\|(D^n_\bp) = \sum_{\bp\in I^h} \|\vec{\mu}_{\bp}\|(D_\bp)<\infty.
\end{equation*}
Therefore $\displaystyle \limsup_{n\to\infty} \sum_{\bp\in \bbU}\|\vec{\mu}^n_{\bp}\|(D^n_\bp) <\infty.$ Hence,  Assumption \ref{ass:uniformBound}(1) for $\cD$ follows from by Assumption \ref{ass:1thm1} and Fatou's lemma.
Now we prove Assumption \ref{ass:uniformBound}(2).
Fix any $h\geq 1$ and $\eps>0$.
For $I^h$ (recall that the latter depends on $\eps$) defined in \eqref{eqn:Ihdef}, we have that for all $n$ large
\begin{equation*}
\{\bp\in \bbU : |\bp|\in\{1,\dotsm, h\}, \textup{ and }\operatorname{diam}(D^n_\bp)> 2\eps\}\subseteq I^h
\end{equation*}
by Assumption \ref{ass:2thm1}. But if $\bp\in \bbU$ with $|\bp|\in\{1,\dotsm, h\}$ is such that $\operatorname{diam}(D_\bp)>2\eps$ then by Assumption \ref{ass:1thm1} necessarily 
\begin{equation*}
\limsup_{n\to\infty} \operatorname{diam}(D^n_\bp) > 2\eps. 
\end{equation*}Hence such point $\bp$ is in $I^h$. 
It follows that
\begin{equation*}
\{\bp\in \bbU: |\bp|\in\{1,\dotsm, h\}, \textup{ and } \operatorname{diam}(D_\bp)>2\eps\} \subset I^h,
\end{equation*} which is a finite set (depending on $\eps$). This verifies Assumption \ref{ass:uniformBound}(2).

Finally, to reach a contradiction, suppose that Assumption \ref{ass:uniformBound}(3) fails to hold. Then for some $\delta>0$, and every $h=1,2,\dotms$ there exists some $\bp =\bp(\delta,h) \in \bbU$ with $|\bp|> h$ and some $y\in D_\bp\hookrightarrow\sG(\cD)\subset \ell^1(\bbU\to \U)$ such that
\begin{equation}\label{eqn:infqzandy}
\forall\ |\bq|\le h, \forall\ z\in D_\bq\textup{ it holds that }\di \left(y,z\right) \ge 3\delta.
\end{equation}
By parts \eqref{eqnThmGraphConv11} and \eqref{eqnThmGraphConv12} of Lemma \ref{lem:limitfordecor}, for all $\eps>0$ there exists an $h(\eps)>0$ such that for all $h\ge h(\eps)$ there is a finite tree $\tr= \tr(\eps,h)\in \bbT$ such that for all $n$ sufficiently large
\begin{equation}\label{eqn:infqzandy2}
\sup_{z\in \sG(\cD^n)} \di(z,\sG_{\le h}(\cD^n))+  \sup_{z\in \sG_{\leq h}(\cD^n)} \di\left(z, \sG(\tr;\cD^n) \right) \le 2\eps,
\end{equation}where $\tr(\eps,h)$ has height at most $h$.
Here we used in each summand the equivalent definition of the Hausdorff distance between two metric spaces $(X,d)$ and $(Y,d)$ as $\max\{\sup_{x\in X}d(x,Y), \sup_{y\in Y}d(X,y)\}$. 

Define $\tr_0 = \tr(\delta, h(\delta))$ as above and let $\tr_1 = \tr_1(\delta,h(\delta),\bp(\delta,h(\delta)))$ be the smallest finite tree such that
\begin{equation*}
\tr_0\subset \tr_1\qquad\textup{and}\qquad \bp(\delta,h(\delta)) \in \tr_1.
\end{equation*}Note that $\tr_0\neq \tr_1$ as $|\bp(\delta,h(\delta))| \ge h(\delta)+1$. 
Also we have
\begin{equation}\label{eqn:infqzandy3}
\sup_{y\in \sG(\tr_1;\cD)} \di(y,\sG(\tr_0;\cD)) \ge     \sup_{y\in \sG(\tr_1;\cD)} \inf_{z\in D_\bq, |\bq|\le h}\di(y,z)\geq \sup_{y\in D_\bp} \inf_{z\in D_\bq, |\bq|\le h}\di(y,z)\ge 3\delta
\end{equation} by \eqref{eqn:infqzandy}.
But for all $n$ large, by \eqref{eqn:infqzandy2}, 
\begin{align*}
&  \sup_{y\in \sG(\tr_1;\cD^n)} \di(y,\sG(\tr_0;\cD^n)) \leq \sup_{y\in \sG(\cD^n)} \di(y,\sG_{\le h}(\cD^n))+  \dghp\left(\sG_{\leq h}(\cD^n), \sG(\tr_0;\cD^n) \right)\le 2\delta.
\end{align*}Taking limsup as $n$ goes to infinity, together with \eqref{eqn:infqzandy3} gives us a contradiction, since Lemma \ref{lem:finTreeLem} implies that in the GHP topology $\sG(\tr_1;\cD^n) \to \sG(\tr_1;\cD)$ and $\sG(\tr_0;\cD^n) \to \sG(\tr_0;\cD)$ as $n\to\infty$.
\end{proof}

\section{Path Properties of L\'{e}vy Processes}\label{sec:levy}

The key to establishing the hypotheses of Theorem \ref{thm:graphConv1} for MBGW trees is some fluctuation theory for L\'{e}vy processes and L\'{e}vy fields which we now recall.

\subsection{Multidimensional First Hitting Times}\label{sec:hittingTimes}

We have already noted the discrete first hitting times of \cite{MR3449255} in Section \ref{subsectionGluedDecorationConditionedWithALargeFirstSubtree} (see equation \eqref{eqn:Rrecurse2}), so in this section we recall some of the recent work of Chaumont and Marolleau \cite{MR4193902,Chaumont:2021} in the continuous time setting.

\subsubsection{Deterministic Additive Fields and first hitting times}

Consider a collection of functions $(x^{i,j};i,j\in[d])$ such that 
\begin{enumerate}
\item For all $i,j$, the function $x^{i,j}:\R_+\to \R$ is c\`adl\`ag and  $x^{i,j}(0) = 0$;
\item For all $i$, we have $x^{i,i}(t)\ge x^{i,i}(t-)$ for all $t\ge 0$;
\item For all $i\neq j$, we have that $x^{i,j}$ is non-decreasing.
\end{enumerate}

Recalling the minimal solution to \eqref{eqn:minSolDiscrete}, the analogous minimal solution when the underlying paths $(x^{i,j})$ are indexed by continuous time $t\ge 0$, is given for $\vec{r}\in\R_+^d$, by 
\begin{equation}
\label{eqn:minSol}  \vec{T}(\vec{r}) := \vec{T}(\vec{r},(x^{i,j})_{i,j}) =   \inf\left\{\vec{t}\in[0,\infty]^d:  \sum_{i=1}^d x^{i,j}(t_i-) = -r_j,\qquad \forall j\textup{ s.t. }t_j<\infty\right\}
\end{equation}
for the \textit{unique} element  $\vec{t}\in [0,\infty]^d$ such that, if $\vec{s}$ is any solution to $\sum_{i=1}^d x^{i,j}(s_i-) = -r_j$ for all $j$ such that $s_j<\infty$, then $s_i\ge t_i$ for all $i\in[d]$. The existence and uniqueness of $\vec T(\vec{r})$ is shown in \cite{MR4193902}.
\label{subsecConstructionFirstHittingTime}

We now describe the construction of the first hitting time as described in the appendix of \cite{MR4193902}. Fix $\vec r\in \mathbb{R}^d_+$. 
Then, for any $i\in [d]$ we define $u^{(0)}_i = \inf\{t: x^{i,i}(t-) = -r_i\}$
and for $h\geq 0$
\begin{equation*}
u^{(h+1)}_i = \inf\{t: x^{i,i}(t-) = -(r_i+\sum_{\ell\neq i}x^{\ell,i}(u_\ell^{(h)})) \}.
\end{equation*}
Note $u^{(h+1)}_i\ge u^{(h)}_i$ for all $h = 0,1,\dotsm$ and $i\in[d]$ so there exists a limit
 $\lim_{n\to\infty} u^{(h)}_i= u^{(\infty)}_i\in[0,\infty].$ As shown in \cite{MR4193902}, $\vec{T}(\vec{r}) = (u^{(\infty)}_1,\dotms, u^{(\infty)}_d)$ is the minimal solution to \eqref{eqn:minSol}.
When $d = 2$, it need not be the case that $u_i^{(h+1)}>u_i^{(h)}$ for all $i$ even if $x^{i,j}$ are strictly increasing for all $i\neq j$. However, when $x^{i,j}$ are strictly increasing for all $i\neq j$ it is easy to see that $u_i^{(h+2)}>u_{i}^{(h)}$ for all $i$ and $h$. 

\subsubsection{Discrete time} The same algorithm above works when the functions $x^{i,j}:\Z_+\to \Z$ satisfy
\begin{enumerate}
\item $x^{i,i}(m) \ge x^{i,i}(m-1)-1$ for all $i\in[d]$ and $m\ge 1$; and
\item $x^{i,j}(m)\ge x^{i,j}(m-1)$ for all $i\neq j$ and $m\ge1.$
\end{enumerate}
The difference is that now one wants to solve \eqref{eqn:minSolDiscrete}, a discrete version of \eqref{eqn:minSol} for the value $\vec{r}\in \Z_+^d\setminus\{0\}$.
We refer the interested reader to \cite{MR3449255} for more details.


\subsection{Properties for Additive  L\'{e}vy Fields}

We now consider the situation where the additive fields are the $\R^d$-valued L\'{e}vy processes $(\vec \bX^{i}(t);t\in\R_+)$ for $i\in [d]$, described in the overview. That is, they have Laplace exponents $\Psi_i$ described in \eqref{eqn:ass1Laplace}. In this random setting, we write $\vec \bT(\vec{r})$ for the first hitting time of level $-\vec{r}$ for the field defined by $\vec \bX^{i}$. First observe that the first hitting time may have some infinite coordinate $\vec \bT(\vec{r})\in[0,\infty]^d\setminus \R_+^d$. Proposition 3.1 and Theorem 3.4 in \cite{MR4193902} give sufficient conditions for $\vec \bT(\vec{r})\in \R_+$ a.s. for each $\vec{r}\in \R^d$. 
\begin{thm}[Proposition 3.1 and Theorem 3.4 \cite{MR4193902}]\label{thm:CM}
Suppose that $\vec \bX^i$ have Laplace transforms defined in \eqref{eqn:ass1Laplace}. Suppose that
\begin{equation}\label{eqn:(H)}
    \textup{the set }\{\lambda\in \R_+^d: \Psi_j(\lambda) > 0\textup{ for all }j\in[d]\} \textup{ is non-empty}
\end{equation}
and that $ J^\Psi:= \left(\E[X^{i,j}(1)]; i,j\in[d] \right)\in \R^{d\times d}$ is irreducible.
Then $\vec \bT(\vec{r})=( \bT^1(\vec{r}),\dotsc,\bT^d(\vec{r}))\in \R_+^d$ a.s. for every $\vec{r}\in \R_+^d$ if and only if the Perron-Frobenius eigenvalue of $J^\Psi$ is non-positive. Whenever this happens, $ X^{i,j}(\bT^i(\vec{r})- ) = X^{i,j}(\bT^i(\vec{r}))$ a.s.
\end{thm}
\begin{remark}\label{rmk:(H)Satisfies}
    Note that \eqref{eqn:(H)} holds if each $\Psi_j$ satisfies \eqref{eqn:ass2Laplace}, and $J^\Psi$ is irreducible if $X^{i,j}$ is strictly increasing for each $i\neq j$. That is, whenever \ref{ass:a2} and \ref{ass:a4} hold.
\end{remark}

\subsection{Convergence of Excursions} 

We now turn to general results about convergence in the Skorohod space. Suppose that $\vec f:\R_+\to\R^d$ is a c\`adl\`ag function. In the limit, we will consider functions $\vec f = (f^1,\dotsm, f^d)$ whose coordinates satisfy
\begin{enumerate}[label = \textbf{F.\arabic*}]
\item \label{F1}$\vec f(0) = \vec{0}$.
\item \label{F2}There exists a unique $i\in[d]$ such that for all $j\neq i$, $f^j(t)$ is non-decreasing.
\item \label{F3} $f^{j}(t)\ge f^j(t-)$ for all $t\ge 0$ for every $j$.
\end{enumerate}
In the discrete we will replace \eqref{F3} with 
\begin{enumerate}
\item[\textbf{F.3'}]\label{F3'} $f^i(t)-f^i(t-)\in \{-\delta, 0,\delta, 2\delta,\dotsm\}$ for some $\delta>0$.
\end{enumerate}

We say that $(l,r)$ is an \textit{(continuous) excursion interval} of $\vec f$ if
\begin{equation*}
f^i(l-)\wedge f^i(l) = \inf_{s\le l} f^i(s) = \inf_{s\le r} f^i(s) = f^i(r-)\wedge f^{i}(r),
\end{equation*}and
\[
f^i(t-)>\inf_{s\le r} f^i(s)\textup{ for all }t\in(l,r).
\]Note that we only consider the excursion intervals for the $i^{\text{th}}$ entry, in the $i$ coordinate appearing in \eqref{F2}.  
We denote the set of excursion intervals for a function $\vec f$ by $\mathcal I(\vec f)$. For an excursion interval $(l,r)$ we denote by $\vec f_{(r)} = (f_{(r)}^1,\dotms, f_{(r)}^d)$ the \textit{excursion} defined by
\begin{equation*}
\vec f_{(r)}(t) = \vec f((t+l)\wedge r)-\vec f(l-) ,\qquad t\ge 0.
\end{equation*}
We also let $\mathcal R(\vec f) = \{r: (l,r)\in \mathcal I(\vec f)\}$ and $\mathfrak L(\vec f) = \{l: (l,r)\in \mathcal I(\vec f)\}$. Note that for any compact interval $[0,T]$, we can always order the excursions so that their lengths are non-increasing. We define for each $k\ge 1$, the \emph{length} $\mathcal{L}_{(k)}(\vec f)$ of the $k^\text{th}$-longest excursion on $[0,T]$.

In the discrete, the excursion intervals are defined through the first hitting times since $f^i(t)\in \delta \Z$ lies in a lattice. We say that $(l,r)$ is a \textit{(discrete) excursion interval} if there is some $k = 0,1,2,\dotms$ such that $l = \inf\{t: f^i(t)=-k\delta\}$ and $r = \inf\{t: f^i(t)=-(k+1)\delta\}$. The remaining notation is defined analogously.

{In the following, we fix $T\in \R_+$ such that $\vec f$ is continuous at $T$. }
Slightly altering the definition of good functions from that in \cite{DvdHvLS.17}, we make the following definition.

\begin{definition}
We say that a function $\vec f\in \D([0,T],\R^d)$ is \textit{good} if it satisfies \eqref{F1}--\eqref{F3} and
\begin{enumerate}[label=\textbf{G.\arabic*}]
    \item \label{G1} $\mathcal R(\vec f)$ does not have any isolated points;
    \item \label{G2} $[0,\sup {\mathcal R(\vec f)}]\setminus \bigcup_{(l,r)\in \mathcal I(\vec f)} (l,r)$ has Lebesgue measure zero;
    \item \label{G3} $f^i$ does not attain a local minimum at any $r\in \mathcal R(\vec f)$; and
    \item \label{G4} $\vec f$ is continuous at every $r\in \mathcal R(\vec f)$ and $l\in \mathfrak L(\vec f)$. 
\end{enumerate}
\end{definition}

For any $k\ge 1$, let $(l_{(i)},r_{(i)})_{i\ge 1}$ be the ordered (with respect to the length) sequence of excursions of $\vec f$ on the interval $[0,T]$ so that $\mathcal{L}_{(i)}(\vec f) = r_{(i)}-l_{(i)}$. If $\mathcal{L}_{(i)}(\vec f) = \mathcal{L}_{(j)}(\vec f)$ for $i< j$, then we can choose this ordering so that $r_{(i)}<r_{(j)}$. Define the map $\mathcal E_k(\vec f) = (\vec f_{(r_1)},\vec f_{(r_2)},\dotsm, \vec f_{(r_k)})\in \D(\R_+,\R^d)^k$ where $r_1<r_2<\dotms <r_k$ and 
\begin{align*}
\{r_1,\dotsm, r_k\} = \{r_{(1)},\dotms, r_{(k)}\}.
\end{align*} That is $\mathcal E_k(\vec f)$ are the $k$ longest excursions of $\vec f$ (with ties broken by order of appearance) and these excursions are listed in order of appearance. We also let $\vec{\mathcal{L}}_{k}(\vec f) = (\mathcal{L}_{(1)}(\vec f),\dotsm, \mathcal{L}_{(k)}(\vec f))$. 
The following proposition is a straight-forward generalization of Lemma 14 in \cite{DvdHvLS.20}, Proposition 22 in \cite{NP.10}, and Proposition 5.12 in \cite{CKG.23} and the proof is omitted.
\begin{prop}\label{prop:good} Suppose that $\vec f, \vec f_n\in \D([0,T],\R^d)$, with $\vec f$ good on $[0,T]$. 
Assume also that
$\vec f_n\to \vec f${ in  the $J_1$ topology}. Then $\vec{\mathcal{L}}_{k}(\vec f_n)\to \vec{\mathcal{L}}_{k}(\vec f)$ in $\R^k$ for each $k\ge 1$. 

If, in addition, for some $k$ it holds that $\mathcal{L}_{(j)}(\vec f)> \mathcal{L}_{(j+1)}(\vec f)$ for every $1\le j \le k-1$, then $\mathcal E_{k}(\vec f_n)\to \mathcal E_k(\vec f)$ in $\D(\R_+,\R^d)^k$.  
\end{prop}
The following corollary is immediate.
\begin{cor}\label{cor:longExcursion2}
Suppose that $\vec f_n\to \vec f$ in the Skorohod space $\D([0,T],\R^d)$ and that $\vec f$ is good. Then for any $\delta>0$ and $j\in [d]$
\begin{equation*}
    \limsup_{n\to\infty} \#\left\{(l,r)\in \mathcal{I}(\vec f_n): f_n^j(r)-f_n^j(l)> \delta\right\}<\infty.
\end{equation*}
\end{cor}

The following lemma will also be useful and is essentially Proposition VI.2.11 in \cite{JS.13}.
\begin{lem}\label{lem:firstPassing}
Suppose that $f_n\to f$ in the Skorohod space of real valued functions, where $f$ is a function with non-negative jumps and satisfies $\liminf_{t\to\infty} f(t) = -\infty$, and $f_n(0),f(0)\ge 0$. For any $x\ge 0$, and any function $g$, define $S_g(x) = \inf\{t: g(t-) \le -x\textup{ or }g(t)\le -x\}$.

If $x_n\to x> 0$, for some $x$ such that $S_f(x+) = S_f(x)$, then $S_{f_n}(x_n)\to S_{f}(x)$.
\end{lem}
\begin{proof}
By replacing $f_n(t)$ with $\widetilde{f}_n(t) = f_n(t) + (x_n-x)$ we can suppose that $x_n = x$ for all $n$. Next, there is no loss in generality in supposing that $f$ is continuous and that all functions $f_n$ and $f$ are non-increasing. Indeed, this can be done by replacing $f(t)$ with $\widetilde{f}(t) = \inf_{s\le t} f(s)$ and similarly replacing $f_n$. This operation leaves the first hitting times $S_f(x) = S_{\widetilde{f}}(x)$ unchanged. 

First if $t< S_f(x)$, then $f(t)> - x$ and so for all $n$ large $f_n(t)>-x$ as well. Hence $\liminf S_{f_n}(x)\ge S_f(x).$ Next if $t>S_f(x)$, then $t> S_f(y)$ for some $y>x$, by the hypothesis $S_f(x+)=S_f(x)$. Then $f(t)< x$ and so for all $n$ large $f_n(t)<x$ as well, implying $\limsup_n S_{f_n}(x)\le S_f(x)$. 
\end{proof}

\subsection{Properties of Spectrally Positive L\'{e}vy Processes under Assumption \ref{ass:a2}}

\begin{lem}\label{lem:good}
Under Assumption \ref{ass:a2}, for each $i\in[d]$ and $T >0$, $\vec \bX^i$ is good on $[0,T]$ a.s.
\end{lem}
\begin{proof}
The conditions \ref{F1}--\ref{F3} are obvious. By Theorem 1 and Corollary 5 in \cite[Chapter VII]{Bertoin.96}, the point $0$ is a \emph{regular point} of $(0,\infty)$ and $(-\infty,0)$ for the L\'{e}vy process $X^{i,i}$, that is $\inf\{t:X^{i,i}(t)>0\} = \inf\{t: X^{i,i}(t)<0\} = 0$, a.s. Together with the Markov property, this implies that $X^{i,i}$ a.s. satisfies \ref{G1} and \ref{G3}.
To establish \ref{G2}, note that by the occupation density formula (Proposition 1.3.3 in \cite{MR1954248}) there is a process $L^i = L_{H^i} = (L(a,t);a\in \R_+,t\in \R_+)$ such that 
\begin{align*}
\int_0^T g(H^i(t))\,dt = \int_0^\infty g(a) L^i(a,T)\,da
\end{align*}
for all non-negative measurable functions $g$. In particular,
\begin{align*}
  \int_0^T 1_{[X^{i,i}(s-) = I^i(s)]}\,ds = \int_0^T 1_{[H^i(s) = 0]}\,ds = \int_0^\infty 1_{0}(a) \, L_{H^i}(a,T)\,da = 0.
\end{align*} 

We now establish \ref{G4}. For simplicity, we suppose that $d = 2$ (the general case is completely analogous) and that $i = 1$. 
We prove that $\vec X=(X^{1,1},X^{1,2})$ is continuous on $\mathcal R(\vec X^1)$. 
Note that by the L\'{e}vy-It\^{o} decomposition and Poisson thinning, we can decompose $  X^{1,2}(t) = X^{1,2}_0(t) + X^{1,2}_1(t)$
where the jump times of $X^{1,2}_0$ are contained in the jump times of $X^{1,1}$, and $X^{1,2}_1$ is independent of the L\'{e}vy process $(X^{1,1},X^{1,2}_0)$. Now given any random set $\mathcal{Z}\subseteq [0,T]$ of Lebesgue measure 0 which is independent of $X_1^{1,2}$, then by standard Poisson process calculations
\begin{equation}\label{eqnX12DoesNotJumpAtIndependentTimes}
\begin{split}
  &\PR(X^{1,2}_1(t)>X^{1,2}_1(t-) \textup{ for some }t\in \mathcal{Z}) \le  \E[\widetilde{\mathcal{N}}_1(\mathcal{Z}\times (0,\infty))] = \E[\Leb\otimes \widetilde{\Pi}_1(\mathcal{Z}\times (0,\infty))] = 0,
\end{split}
\end{equation}where $\widetilde{\Pi}_1$ is the L\'{e}vy measure of $X^{1,2}_1$ and $\widetilde{\mathcal{N}}$ is a Poisson random measure with intensity $\Leb\otimes\widetilde{\Pi}_1$. Hence, $X^{1,2}_1$ is continuous on $\mathcal R(\vec X^1)$ and $\mathfrak L(\vec X^1)$ a.s., by \ref{G2}. Thus, the conclusion \ref{G4} then follows if we can establish that $X^{1,1}$ is continuous at each $r\in \mathcal R(\vec X^1)$ and $l\in \mathfrak L(\vec X^1)$. 

On the one hand, by Remark 12 in \cite{DvdHvLS.20}, if $\vec f^i$ satisfies \ref{F1}--\ref{F3} and \ref{G1}--\ref{G3}, then $\vec f^i$ is continuous at each $r\in \mathcal R(\vec f^i)$. 
On the other hand, note that each $l\in \mathfrak L(\vec X^1)$ is a local infimum by its definition. Then the regularity of the process implied by Assumption \ref{ass:a2}, and \cite[Theorem 3.1]{MR0433606} or \cite[Theorem 2]{MR4130409}, imply that $X^{1,1}$ is continuous at $l$. 
This gives the desired continuity of $X^{1,1}$ and establishes our claim.
\end{proof}

Recalling the construction of the minimal solution in Section \ref{sec:hittingTimes}, we define for any $\vec{r}\in \R^d_+$ and $h\geq 0$:
\begin{align}
R^{(0),j} &= r_j, & U^{(0),j} &= \min\{t: X^{j,j}(t)=-r_j\}\nonumber\\
R^{(h+1),j} &= r_j + \sum_{\ell\neq j} X^{\ell,j}(U^{(h),\ell}), &U^{(h+1),j} &= \min\left\{t: X^{j,j}(t) = - \left( r_j + \sum_{\ell\neq j} X^{\ell,j}(U^{(h),\ell})\right)\right\}\label{eqn:Rrecurse}.
\end{align}Note that a similar equation in the discrete has been introduced in \eqref{eqn:Rrecurse2}. 
Write $\tau^{j}(x) = \inf\{t:X^{j,j}(t) = -x\}$ for any $x\geq 0$ and $j\in [d]$. 
Then by induction $U^{(h),j} = \tau^{j}(R^{(h),j})$.

Fix $r\in \mathbb{R}_+$ and $i_0\in [d]$. 
Let us now sample an excursion $\vec \be^{i_0} = (e^{i_0,1},\dotsm, e^{i_0,d})\sim \vec \bN^{i_0}(-|\zeta\ge r)$, which will serve as the `subtree' of type $i_0$ connected to the root, together with the places where to paste the subtrees of type $i\neq i_0$. 
Consider an independent process $\vec \bX^i$ for $i\neq i_0$. Let us define $(R^{(h),j};h\ge 0, j\in[d])$ as in \eqref{eqn:Rrecurse} for the initial conditions $r_j = e^{i_0,j}(\zeta)$ for all $j\in [d]$.
The following lemma implies the independence of all the subdecorations at reduced height $h\geq 1$, growing from $\cD(\emptyset)$, the latter constructed from $\vec \be^{i_0}$.

\begin{lem}\label{lem:firstHitting1}
Suppose that $\vec \bX^j$ satisfies \ref{ass:a2} for all $j\in[d]$. Then, almost surely, for every $j\in[d]$ and $h\ge 0$,
\begin{equation}\label{eqn:firstHittingAtRks}
\tau^{j}(R^{(h),j}+) = S_{X^{j,j}} (R^{(h),j}+) = S_{X^{j,j}}(R^{(h),j}) = \tau^{j}(R^{(h),j}),
\end{equation}
and for every $\ell\neq j$
\begin{equation}\label{eqn:firstHittingAtRks1}
 X^{\ell,j}(U^{(h),\ell}-) = X^{\ell,j}(U^{(h),\ell}).
\end{equation}
Moreover,
\begin{equation}\label{eqn:firstHittingAtRks2}
\left(\vec \bX^{j}(U^{(h),j}+t) - \vec \bX^{j}(U^{(h),j});t\ge 0\right)
\end{equation}
is independent of $\sF(\vec{U}^{(h)})$ where $\sF(\vec{t})$ is the standard augmentation of $\sigma(\cup_{j=1}^d\{\vec \bX^j(s): s\le t_j\})$.
\end{lem}
\begin{proof}
From the strong Markov property of $X^{j,j}$, together with the fact that the first hitting time of a closed set is a stopping time, we obtain that $S_{X^{j,j}}(R^{(0),j})$ is a stopping time. 
An induction argument then implies that $S_{X^{j,j}}(R^{(h),j})$  is a stopping time for any $h\geq 1$. 
As noted above, the point $0$ is regular for both $(0,\infty)$ and $(-\infty,0)$ for the L\'{e}vy process $X^{j,j}$ \cite[Theorem VII.1, Corollary VII.5]{Bertoin.96}. 
The previous two assertions imply the middle equality in \eqref{eqn:firstHittingAtRks}.
The outside two are simply a notation change between a deterministic setting and the stochastic setting.

To prove \eqref{eqn:firstHittingAtRks1}, we use Proposition I.7 in \cite{Bertoin.96}. More precisely, by induction suppose that $U^{(h),j}$ is a.s. a continuity point of $\vec\bX^j$. 
If $U^{(h+1),j} = U^{(h),j}$, then $U^{(h+1),j}$ is a.s. a continuity point of $\vec\bX^j$. Otherwise
\begin{align*}
U^{(h),j}< U^{(h+1),j}& = \inf\{t: X^{j,j}(t-) = -{R}^{(h+1),j}\} \\
&= \lim_{k\to\infty} \inf\left\{t: X^{j,j}(t-) = - R^{(h),j} - \frac{k-1}{k}\left(R^{(h+1),j}- R^{(h),j}\right)\right\}
\end{align*} where the right-hand side is the limit of stopping times increasing towards $U^{(h),j}$. The rest of the argument is essentially the same as that of Proposition 3.1 in \cite{MR4193902}, and we omit the details.
\end{proof}

Now we prove that Assumption \ref{ass:a2} implies that the L\'evy process does not have excursion lengths of equal size. 
This is the condition needed in the second part of Proposition \ref{prop:good}.

\begin{lem}\label{lem:distinctExcursionLengths}
	Suppose that $\vec \bX^i$ satisfies Assumption \ref{ass:a2}. Let $(\zeta_k;k\ge 1)$ be some enumeration of all the excursion lengths of $X^{i,i}$ on $[0,T]$. Then $\PR(\zeta_i = \zeta_j\textup{ for some }i\neq j) = 0$. 
\end{lem}

\begin{proof}
Since this is just a statement of $X^{i,i}$ we write $X$ for $X^{i,i}$. The proof follows an approach analogous to \cite[Fact 1]{DvdHvLS.20}. Fix a positive rational number $q$. If $q$ is contained in an excursion interval, we will write the excursion interval as $I(q) = (g(q),d(q))$ and write $\zeta(q)$ for its length. If $q$ is an endpoint, then just set these as $\dagger$, some cemetery state.  It clearly suffices to show \begin{equation}\label{eqn:distinctexcurisonintervals}\forall q_1<q_2 \in \Q_+ \qquad \PR(\zeta(q_1) = \zeta(q_2) > 0 \mbox{ with } q_1, q_2\textup{ in distinct excursion intervals}) = 0.
\end{equation}

Let $\F^\circ(t) = \sigma\{X(s):s\le t\}$ and let $\F(t)$ be the standard augmentation of $\F^\circ$. Note that the event $\{q_1<q_2 \mbox{ are contained in distinct excursion intervals}\}$ is $\F(q_2)$ measurable. Also, $\zeta(q_1), g(q_2),$ and $X(q_2) - X(g(q_2)-)$ are all $\F(q_2)$-measurable. It follows that if we set $Y: = X(q_2)-X(g(q_2))$ then 
\begin{align*}
d(q_2) - q_2 = \inf\{t: X(t+q_2) \le  X(g(q_2))\} = \inf\{t: X(t+q_2)-X(q_2) \le - Y\}.
\end{align*} Note that $\tilde{X}(t) =  X(t+q_2)-X(q_2)$ is independent of $\F(q_2)$ by the Markov property. 
Since
\begin{align*}
\{&\zeta(q_2) = \zeta(q_1) \mbox{ with }q_1,q_2, \textup{ in distinct excursion intervals}\}\\
 &= \{g(q_2)> g(q_1)\}\cap \{d(q_2)- q_2 = \underbrace{g(q_2) + \zeta(q_1)-q_2}_{R:=}\},
\end{align*} we see that 
\eqref{eqn:distinctexcurisonintervals} follows from showing that $\PR(\inf\{t: \tilde{X}(t)\le -Y\} = R)$ is $0$. Since $\tilde{X}$ is independent of $-Y$ and $R$ (which are $\F(q_2)$ measurable), it is enough to show that
\begin{align*}
\PR(\inf\{t: \tilde{X}(t)< -y\} =  r) \le \PR(X(r) = -y) = 0 \qquad\textup{ for all }y>0, r>0.
\end{align*}
This follows from \cite[Theorem 27.4]{MR1739520} and \eqref{eqn:ass2Laplace}, since either the Gaussian component of $X$ is non-trivial or the L\'{e}vy measure $\pi$ is infinite as $\int_0^1 r\widetilde{\pi}_i(dr) = +\infty$.
\end{proof}

\section{Convergence of Conditioned MBGW Trees to the Conditioned Multitype L\'evy Tree}\label{sec:convergence}
In this section we prove Theorem \ref{thm:MAIN} by 
coding the MBGW tree with glued decorations (Lemma \ref{lem:decoration}), and proving that the latter satisfy \ref{ass:1thm1}--\ref{ass:3thm1}. 
More precisely, the setup is the following. 
Let $i_0\in [d]$ and $r_n = \fl{b_n^{i_0} r}$ for some $r>0$. 
We consider a sequence of conditioned multitype trees $T_{n,r_n}$ as constructed in Theorem \ref{thm:MAIN}. 
By Lemma \ref{lem:decoration} we can change from the conditioned MBGW tree to the conditioned glued discrete decoration $\sG( \cD^n_{r_n})\sim M^{i_0,\vec \varphi_n}(-|\#\cD^n(\emptyset) \ge r_n)$ (recall \eqref{eqnDecorationsWithBigSubtreeBig} for the notation of $\sG( \cD^n_{r_n})$). 
The decoration $\cD^n_{r_n}$ is composed of the subdecorations $(\cD^n(\bp);\bp\in \bbU)$. 
Also, by Corollary \ref{cor:DecorationSpalf}, we can construct such conditioned glued decoration and subdecorations with the random walks $(\vec X ^j;j\in [d])$ conditioned on $\tau^{i_0}(1) \geq r_{n}$, as described therein. 
We will write $\vec \be_{\bp,n} = (e^1_{\bp,n},\dotsm, e^{d}_{\bp,n})$ for the discrete excursions of $\vec X^j_n$ whenever $\bp = \bq \ell$ and $\ell\equiv j\mod d$, for any $\bp\in \bbU$.  
This encodes the decoration $\cD^n(\bp)$. Similarly write $h_{\bp,n}$ for the corresponding height process. We also let $\zeta_{\bp,n}$ denote the lifetime of the excursion $h_{\bp,n}$.

If $\cD^n(\bp) = (D^n_\bp,\rho_\bp, d_\bp, \vec{\mu}_\bp, (x^n_{\bp;\ell})_{\ell\ge 1})$ we will write $\overline{\cD}^n(\bp)$ for the rescaled subdecoration defined by
\begin{equation}\label{eqnRescaledSubdecoration}
\overline{\cD}^n(\bp) = (D^n_\bp, \rho_\bp,a_n^{-1}d_\bp,  \diag(\vec{b}_n)^{-1} \vec{\mu}_\bp, (x^n_{\bp;\ell})_{\ell\ge 1}),
\end{equation}where $\diag(\vec{b}_n)$ is the diagonal matrix whose $(i,i)^\textup{th}$ entry is $b_n^i$, for all $i\in [d]$.
Note that in the notation of the subdecorations of Theorem \ref{thm:graphConv1}, that is, $\cD^n(\bp): = \left(D^n_\bp, \di^n_\bp , \rho^n_\bp,\vec{\mu}^n_\bp, (x^n_{\bp;\ell})_{\ell\ge 1}\right)$, we have $d^n_\bp:=a_n^{-1}d_\bp$ and $\vec{\mu}^n_\bp:=\diag(\vec{b}_n)^{-1} \vec{\mu}_\bp$. 

Now, for the continuous setting, recall the notation $\bM^{i_0}_r=\bM^{i_0,\vec \Psi}_r$ from Definition \ref{constructionMultitypeLevyTreeGluedDecoration}. 
Let $\sG(\cD_r)\sim \bM^{i_0}_r$ and recall from Construction \ref{cons:continuum2}, that for any $\bp\in \bbU$, the decoration $\cD(\bp)$ is coded by an excursion $(\be,h)$ under $\bN^j$ where $\bp = \bq \ell$ for $\ell\equiv j\mod d$, and $\bq\in \bbU$. 
Note that only $\cD(\emptyset)$ is conditioned, whereas the other subdecorations are not conditioned. 
Given the excursion $(\vec \be_\bp,h_\bp)$, we will say that the decoration $\cD(\bp)$ is coded by $(\vec\be_\bp,h_\bp)$.
Moreover, for each $\bp$ the sequence of marks $(x_{\bp;\ell} ;\ell\equiv j\mod\ d)$  are conditionally i.i.d. with law proportional to the Stieljes measure $d\be^j$.

Let us mention that the verication of \ref{ass:1thm1} is an inductive argument.  Indeed, we use the convergence of $\cD^n(\bp)$ for $\bp\in \bbU$ with $|\bp|\le h$ to get the convergence of $\cD^n(\bq)$ for $\bq\in \bbU$ with $|\bq|= h+1$. The verification of \ref{ass:2thm1} and \ref{ass:3thm1} is done by relating the MBGW tree with spaLf's recalled in Section \ref{sec:levy}. Before getting to that, we prove that the rescaled process $a_n^{-1} H_n^i(\fl{b_n^i\, \cdot })$ has only finitely many excursions whose height or length exceeds any fixed threshold $\eps$, see Lemma \ref{lem:heightExcursions}. 
After that, we show that the rescaled height process and random walk, conditioned on the first subexcursion being large, converges, see Lemma \ref{lem:longExc}. 
Those results depend strongly on Proposition \ref{prop:dlHeight} and Proposition \ref{prop:good}.

\subsection{Single-Type Convergence}

In this subsection, we work under the setting of Proposition \ref{prop:dlHeight}.
Recall the construction of the height process of a tree $\tr$ using the functional \eqref{eqn:HeightDiscFunctional}. 
For any fixed $i\in [d]$, we let $H^{i}_n$ denote the associated height process of the \L ukasiewicz path $X_n^{i,i}$.
Such BGW forest starts with $\fl{b_n^i/a_n}$ roots.  
Recall that under the assumptions \ref{ass:a2} and \ref{ass:a3} there exists a continuous height process $H^i$ associated to $X^{i,i}$.  

For the next lemma, fix any $x>0$, $\eps>0$ and $i\in [d]$.  Define $\gamma_n:=\gamma_n(x,\epsilon)$ as
\begin{align*}
\gamma_n &:= \#\left\{k\leq \frac{b_n^i}{a_n}x: \sup_{\tau^{i}_n(k-1)\le m <\tau^{i}_n(k)} \frac{1}{a_n} H_n^i(m)>\eps\textup{  or  }\frac{\tau^{i}_n(k)-\tau^{i}_n(k-1)}{b^i_n} >\eps  \right\}.
\end{align*}
Note that $\gamma_n$ is the number of excursion intervals $(l,r)$ on $[0,\tau^i_n(\fl{b_n^ix/a_n})]$ such that either the rescaled height on such interval or the rescaled excursion length is bigger than  $\eps$. 

\begin{lem}\label{lem:heightExcursions}
 Suppose the hypotheses of Proposition \ref{prop:dlHeight} hold. Then $(\gamma_n;n\in \mathbb{N})$ 
is tight.
\end{lem} 
\begin{proof}

Define $I_n^i(m): = \min_{k\le m} X_n^{i,i}(k)$ and $I^i(t):= \inf_{s\le t} X^{i,i}(s)$ for any $m\in \mathbb{N}, t\in \mathbb{R}_+$.
Then, jointly with the convergence in Assumption \ref{ass:a1} and jointly with the convergence of the rescaled height processes of Proposition \ref{prop:dlHeight}, it holds that
\begin{equation*}
    \left(\frac{a_n}{b_n^i} I_n^i(\fl{b_n^it});t\ge 0\right)\weakarrow \left(I^i(t);t\ge 0\right).
\end{equation*}
Note that for any $k\in \mathbb{N}$ we have $ \tau^{i}_n(k) = \inf\{m: -I_{n}^i(m)  = k\}$ and so for any $y\in \mathbb{R}_+$
\begin{equation*}
  \left\{\tau_n^{i}\left(\fl{\frac{b_n^{i}}{a_n}x}\right)> b_n^i y \right\} = \left\{ \frac{a_n}{b_n^i}I_n^i(\fl{b_n^iy}) > -x\right\}.
\end{equation*} 
Note also that the jumps of $\tau_n^{i}$ are the excursion lengths of $X^{i,i}_n$. Let $\tilde{\mathcal{I}}_y(X_n^{i,i})$ denote the collection of discrete excursion intervals of $X_n^{i,i}$ which end before time $b_n^iy$.

Now for each $y\in \mathbb{R}_+$, we can define new random variables $\gamma_n^{(y)}$ as
\begin{equation*}
    \gamma_n^{(y)}  := 1_{[\frac{a_n}{b_n^i}I_{n}^i(\fl{b_n^iy}) \leq -x]}\#\left\{(l,r)\in \tilde{\mathcal{I}}_y(X_n^{i,i}): \sup_{m\in \{l,l+1,\dotsm, r-1 \}} \frac{1}{a_n} H_n^i(m)>\eps,\, \text{or}\, \frac{r-l}{b_n^i}> \eps\right\}.
\end{equation*}
Then clearly for all $z>0$ we have
\begin{equation*}
    \PR(\gamma_n>z) \le \PR\left(\frac{a_n}{b_n^i} I_n^i(\fl{b_n^iy}) > -x\right) + \PR\left(\gamma_n^{(y)} >z\right).
\end{equation*}
Recall from Lemma \ref{lem:good} that the limit $\vec \bX^{i}$ is good.
Thus, by the joint convergence in Assumption \ref{ass:a1}, the rescaled height processes of Proposition \ref{prop:dlHeight},  Proposition \ref{prop:good} and the reverse Fatou's lemma we have
\begin{equation}\label{eqnlem:heightExcursions}
    \limsup_{n} \PR(\gamma^{(y)}_n >z ) \le \PR\left(\#\left\{(l,r)\in \mathcal{I}(X^{i,i}):  \sup_{t\in(l,r)} 	H(t)\ge \eps,\, \text{or}\,  r-l\ge\eps, r\le y\right\}>z\right).
\end{equation}
In fact, the random variable appearing in the right-hand side above is finite almost surely on the compact $[0,y]$.
Hence, the probability in \eqref{eqnlem:heightExcursions} converges to $0$ as $z\to\infty$ for each fixed $\eps,x,y>0$. 

Similarly, by the convergence of $X_n^{i,i}$ we get
\begin{equation*}
    \limsup_{n} \PR\left(\frac{a_n}{b_n^i} I_n^i(\fl{b_n^iy}) > -x\right) \le \PR(I^i(y)\ge -x).
\end{equation*} As $\liminf_{t\to\infty} X^{i,i}(t)= -\infty$, one can find $y$ sufficiently large so that the right-hand side above is arbitrarily small. The above discussion about excursion intervals implies the desired tightness.
\end{proof}

We will also use Proposition 2.1(b.5) in \cite[Chapter VI]{JS.13} frequently in the sequel so we recall this here without proof.
\begin{lem}\label{lem:JSlem}
Suppose that $f_n\to f$ in the 
$J_1$ topology, $t_n\to t$, and $f(t) = f(t-)$. Then $f_n(t_n)\to f(t)$.\end{lem}

Now we state an important corollary which involves conditioning the first excursion of the decoration to be macroscopic in a certain sense. The full proof is only sketched as the argument is similar to how Proposition 2.5.2 follows from Lemma 2.5.3 in \cite{MR1954248}. See also the remark of Page 73.

First we describe the notation for the conditioned process. 
Fix $i\in [d]$ and $r>0$.
Let $r_n = \fl{b_n^{i} r}$ and let 
\[
\left(\left(X^{i,j}_{n,r_n}(m); j\in[d]\right),H_{n,r_n}^{i}(m) ; m\le \zeta_{n,r_n}\right)
\]be distributed as $\big((X^{i,j}_{n}(m); j\in[d]),H_{n}^{i}(m) ; m\le \zeta_n\big)$ conditioned on $\zeta_n:=\tau^i_n(1)\geq r_n$.
In words, the process is the \L ukasiewicz path and height process of a MBGW type $i$ subtree constructed from $\vec N^{i,\vec \varphi_n}(-\big| \zeta_n\geq r_n)$.
Note that for such a forest, we use the depth-first search order, that is, traverse the first tree in depth-first search order, and then the second, and so on.

\begin{lem}\label{lem:longExc}
{Fix $i\in [d]$ and $r>0$. }Suppose the assumptions of Proposition \ref{prop:dlHeight} hold. Then for each $r> 0$, $\vec \bN^{i}(\zeta =r) = 0$ and provided $r_n/b^i_n\to r$
\begin{equation*}
    \left(  \left(\frac{a_n}{b_n^j}X^{i,j}_{n,r_n}(b_n^{i}t); j\in[d]\right),\frac{1}{a_n} H_{n,r_n}^{i}(b^{i}_nt), \zeta_{n,r_n}/b_n^{i}; t\le \zeta_{n,r_n}/b_n^{i}\right) \weakarrow \left(\vec \be_r,h_r,\zeta\right) \sim \vec \bN^{i}(-|\zeta \ge r).
\end{equation*}

\end{lem}
\begin{proof}[Proof Sketch] The fact that $\vec{\bN}^{i}(\zeta =r) = 0$ follows from Lemma \ref{lem:distinctExcursionLengths} and some standard computations involving Poisson processes. 
	
	Let $H_n^i$ be the unconditioned discrete height process which encodes the forest of $\fl{b_n^i/a_n}$ BGW trees, as in the setting of Proposition \ref{prop:dlHeight}. 
Let $$Q_r^n = \inf\{m: H_n^i(m-k) > 0 \textup{ for all }k = 0,1,\dotsm, b_n^ir\}.$$
Note that the tree in which the vertex $Q_r^n$ belongs, is the first in the forest with at least $r_n$ individuals.
Let $G_r^n = \sup\{m\le Q_r^n: H_n^i(m) = 0\}$ and $D_r^n = \inf\{m\ge Q_r^n: H_n^i(m) = 0\}$ be the roots of the aforementioned tree, and the following one, respectively. 
Let $Q_r = \inf\{t: H^i(t-s) > 0 \textup{ for all }s\in[0,r]\}$ and similarly define $G_r$ and $D_r$. Note that these times could have also been formulated in terms of the behavior of the $i^\text{th}$ coordinates $X^{i,i}_n$ and $X^{i,i}$ as well.

Then $\widetilde{H}_n^i(m\wedge \zeta_n) \overset{d}= H_n^i ((G_r^n + m)\wedge D_r^n)$,  $h^i(t) \overset{d}{=} H^i((G_r + t)\wedge D_r)$ and $G_r<\infty$ a.s. by general excursion theory and properties of Poisson random measures. See e.g. Lemma 1.13 in \cite[Chapter XII]{MR1725357}. 

From Proposition \ref{prop:dlHeight}, applying Skorohod's representation theorem we can assume the convergence of the vector-valued process $\vec X_n^{i}$ jointly with $H_n^i$ 
(appropriately rescaled) 
occurs almost surely in the Skorohod space. 
The assumption that $\vec \bN^i(\zeta = r) = 0$ implies that $\PR(h^i (r) = 0) = 0$, and so the convergence of $H^i_n$ to $H$ allows us to see that $Q^n_r\to Q_r$. 
Using standard arguments (see p. 72 in \cite{MR1954248}, and for related convergences Section 2.2 in \cite{mijatovic2022limit}), we can also deduce $(b_n^i)^{-1} G_r^n\to G_r$ and $(b_n^i)^{-1} D_r^n\to D_r$ a.s. under this coupling. 
This implies the convergence of the rescaled excursion length $\zeta_n$. 
Also, by the continuity of $H^i$ and Lemma \ref{lem:JSlem}, it follows that $\frac{1}{a_n} H_n^i(  (G_r^n + b_n^it) \wedge D_r^n) \longrightarrow H^i((G_r+t)\wedge D_r)$ locally uniformly in $t$. 
We leave the details to the reader. 
\end{proof}

In the following subsections we prove that \ref{ass:1thm1}-\ref{ass:3thm1} hold for our conditioned multitype BGW trees. 

\subsection{Establishing \ref{ass:1thm1}}

In the following lemma, we establish the joint convergence of a subdecoration $\overline{\cD}^{n}(\bp)$, together with all the subdecorations pasted to it.

\begin{lem}\label{lem:inductionProof} Assume \ref{ass:a1}--\ref{ass:a4}.

Fix some $\bp\in \bbU$ where $\cD(\bp)$ has a root of type $i$. 
Suppose that
\begin{align}\label{eqn:enConverge1}
    &\left(\left(\frac{a_n}{b_n^1} e_{\bp,n}^{1}(b^i_nt), \dotsm, \frac{a_n}{b_n^d}e^{d}_{\bp,n}(b^i_nt)\right);t\ge 0\right) \weakarrow \vec \be_\bp, \\
& \left(\frac{1}{a_n}h_{\bp,n}(b_n^it);t\ge 0\right)\weakarrow h_\bp,
    \qquad \text{and}\qquad  \frac{\zeta_{\bp,n}}{b_n^i}\weakarrow \zeta_\bp.\nonumber
\end{align}
Then
\begin{enumerate}
    \item $\overline{\cD}^{n}(\bp)\weakarrow\cD(\bp)$ in $\mathfrak{M}^{\ast d}_{\tiny \mbox{mp}}$ where $\cD(\bp)$ is coded by the excursion {pair} $(\vec \be_\bp,h_\bp)$.
    \item Jointly with the convergence in \eqref{eqn:enConverge1}, for each $m\ge 0$ and $j\neq i$, if $\ell = j+md$ then
     \begin{align*}
    &\left(\left(\frac{a_n}{b_n^1} e_{{\bp \ell,n}}^{1}(b^j_nt), \dotsm, \frac{a_n}{b_n^d}e_{\bp \ell,n}^{d}(b^j_nt)\right);t\ge 0\right) \weakarrow \vec \be_{\bp \ell}, \\
& \left(\frac{1}{a_n}h_{\bp \ell,n}(b_n^jt);t\ge 0\right)\weakarrow h_{\bp \ell},\qquad
    \text{and} \qquad\frac{\zeta_{\bp \ell,n}}{b_n^j}\weakarrow \zeta_{\bp \ell},
\end{align*}where $\{(\vec \be_{\bp \ell}, h_{\bp \ell}) : \ell = j+md, m\ge 0\}$ are the ordered atoms (by decreasing lifetime) of a conditionally independent Poisson random measure on $\D(\R_+,\R^d)\times C(\R_+,\R)$ with intensity measure $e_{\bp}^{j}(\zeta_\bp)\times \vec \bN^j(-)$ listed by decreasing lifetimes, conditionally given $e_{\bp}^j(\zeta_\bp)$ for all $j\in [d]$.
\end{enumerate}
\end{lem}

Before moving to the proof of the above lemma, we state a claim whose proof is an easy combinatorial computation and is omitted.
\begin{claim}\label{claim:permutation}
Fix any $x\in \mathbb{R}_+$. Suppose that $(Q_n;n\ge 1)$ are positive random integers, $\beta_n\to\infty$ and $Q_n/\beta_n\weakarrow x$, as $n\to \infty$. Conditionally given $Q_n$, let $\pi_n$ be a uniform permutation in $\mathfrak{S}_{Q_n}$, the symmetric group on $Q_n$ letters. Then, in the product topology on $\R_+^\infty$, $ \left(\beta_n^{-1} \pi_n(\ell); \ell\ge 1 \right)   \weakarrow (U_\ell;\ell\ge 1)$ where $(U_\ell;\ell\ge 1)$ are i.i.d. uniform random variables on $[0,x]$. 
\end{claim}
\begin{proof}[Proof of Lemma \ref{lem:inductionProof}] We will fix $\bp\in \bbU$ and omit it from the notation, unless it is needed for clarity. For example we will write $\vec \be_{n}$ for $\vec \be_{\bp,n}$ and $\vec \be_{\ell,n}$ for $\vec \be_{\bp \ell,n}$. Let us first note that Item 2 follows from Proposition \ref{prop:good} and Lemma \ref{lem:good}. 
Indeed, by the results in Section \ref{sec:discreteDecorations}, we know that for each $j\neq i$,
\begin{align*}
 	\left(\vec \be_{\ell,n};\ell\ge 1\textup{ and }\ell\equiv j\mod d\right)
\end{align*}
is equal in law to the ordered excursions of $\vec X_n^j$ run until $\inf\{t: X_n^{j,j}(t) = -e^j_{n}(\zeta_{n})\}$ for the independent process $\vec \be_{n}$. By Lemma \ref{lem:good}, we know that the limiting process $\vec \bX^j$ is good a.s. and so by Lemma \ref{lem:distinctExcursionLengths} and Proposition \ref{prop:good}, we can transfer the rescaled convergence of $\vec X_n^j\to \vec \bX^j$ to the rescaled convergence of the excursions $\vec \be_{\ell,n}\to \vec \be_\ell$ (recall that the convergence of the excursions in the product space is equivalent to the convergence of the longest $k$ excursions, for any $k$ arbitrary). 
Extending this to the height processes and lengths is not difficult.

Recall from Construction \ref{construction:Discretesingletype} of the subdecoration $\cD^n(\bp)$, the points $(x_{\ell};\ell\ge 1, \ell\equiv j\mod d)$ {for $j\neq i$} are constructed as follows: 
\begin{equation*}
\{ x_\ell: \ell = j+md, 0\le m < e^j_n(\zeta_n)\}  = \{v_t: v_t\textup{ appears }e^j_n(t)-e^j_n(t-1) \textup{ many times}\},
\end{equation*}
where $v_t$ are all the vertices of a plane tree with \L ukasiewicz path $e^i_n$ listed in depth-first order. Moreover, if $\pi_n^j$ is uniformly distributed on $\mathfrak{S}_{e^j_n(\zeta_n)}$ then one can show that the sequence $(x_\ell;\ell\geq 1, \ell = j+md, \ell\equiv j\mod d)$ is equal in law to
\begin{equation}\label{eqnV_Sjnm}
\left(v_{S_n^j(1)},\dotms, v_{S^j_n({e^j_n(\zeta_n))}}\right),\qquad \mbox{where }\ S_n^j(m): = \inf\{t: e_n^j(t)\ge \pi_n^j(m)-1\}.
\end{equation}The previous assertion follows easily from the following argument. 
Consider any $t\in [0,\zeta_n]$, and for simplicity we work with $\widetilde S_n^j(m) = \inf\{t: e_n^j(t)\ge m-1\}$. There are $e^j_n(t)-e^j_n(t-1)$ values of $m$ such that $t$ is the first time that $e^j_n(t)\geq m-1$, namely $m=e^j_n(t-1)+2,\dotsm, e^j_n(t)+1$. 
Note that $\widetilde S_n^j(e^j_n(t-1)+2)=\dotsm=\widetilde S_n^j(e^j_n(t)+1)$ in this case. 
If $s:=\widetilde S_n^j(e^j_n(t-1)+2)$, the latter implies $v_s$ appears $e^j_n(t)-e^j_n(t-1)$ times in \eqref{eqnV_Sjnm}. 

By Skorohod's representation theorem, we can and do suppose that the convergence in \eqref{eqn:enConverge1} holds almost surely. 
Lemma \ref{lem:good} {and \ref{G4} imply} that $e^j$ is continuous at both zero and $\zeta$. 
Note that here we use the second part of \ref{lem:good} when $\bp=\emptyset$. 
Also, recall that $\zeta_n/b_n^i\to \zeta$ by assumption. By Lemma \ref{lem:JSlem} we have $\frac{a_n}{b_n^j} e_n^j(\zeta_n) \to e^j(\zeta).$ Applying Claim \ref{claim:permutation}, we see that jointly for every $m\ge 1$ 
\begin{equation*}
\frac{a_n\pi_n^j(m)}{b_n^j} \weakarrow U_m^j,\qquad (U^j_m;m\ge 1)|\vec \be \overset{i.i.d.}\sim \textup{Unif}[0,e^j(\zeta)].
\end{equation*}
Another application of Skorohod's representation theorem implies that we can take the convergence above to be almost sure convergence (which is also jointly with the previous convergences of heights and excursion lengths). By Assumption \ref{ass:a4}, the process $e^j$ is strictly increasing on $[0,\zeta)$ and so $(e^j)^{-1}(u) = \inf\{t: e^j(t)>u\}$ is continuous on $[0,e^j(\zeta))$. 

By Theorem 7.2 in \cite{Whitt.80}, for each $y\in[0,\be^j(\zeta))$ the following convergence holds in the Skorohod space $\D([0,y],\R_+)$ (and hence uniformly since the limit is continuous)
\begin{equation*}
\left(\inf\left\{t: \frac{a_n}{b_n^j} e_n^j(\fl{b_n^i t})> u\right\};u\in[0,y] \right)\longrightarrow\left( \inf\{t: e^j(t)>u\};u\in[0,y]\right).
\end{equation*}
This implies that for each $m$ such that $U^j_m<y$
\begin{equation*}
\frac{S_n^j(m)}{b_n^i} \longrightarrow \inf\{t: e^j(t)> U_m^j\} = \inf\{t: e^j(t)\ge U_m^j\}.
\end{equation*}
Since $y\in [0,e^j(\zeta))$ was arbitrary, the above convergence holds for every $m\ge 1$ a.s. In particular, we have shown that 
\begin{equation*}
\frac{S_n^j(m)}{b_n^i} \longrightarrow S^j(m),\qquad\textup{ where }\,S^j(m)|\vec \be\overset{i.i.d.}\sim \frac{e^j(ds)}{e^j(\zeta)},
\end{equation*}and $(S^j(m))_{m\geq 1}|\vec \be$ are (conditionally) i.i.d.

Recall that we are working with the rescaled distance $a_n^{-1}d_\bp$. 
Let $t_n(\ell),t(\ell)$ be such that
$x_{n;\ell}:=p_{h_n}(t_n(\ell))\in \T_n$ (resp. $x_\ell:=p_{h}(t(\ell))\in \T$) under the canonical quotient $p_{h_n}:[0,\zeta_n]\to \T_n$ (resp. under the canonical quotient $p_{h}:[0,\zeta]\to \T$). 
Consider the rescaled versions $\overline x_{n;\ell}:=\overline p_{h_n}(\cdot):=p_{h_n}(\cdot/b^i_n)$. 
Then, equipping $\T_n$ with the rescaled metric $a_n^{-1}d_{h_n}$, the points $p_{h_n}(S_n^j(m))$ converge to $p_h(S^j(m))$.

The desired convergence of the subdecorations is now a modification of Lemma 2.4 in \cite{MR2203728} and Lemma 21 in \cite{ABBG.12} which we include as a separate lemma.
\end{proof}

\begin{lem}\label{lemmaConvergenceOfMArkedTreesGivenConvergenceOfHeingthFns}
Suppose that $h_n, h$ are continuous non-negative functions such that the support of $h_n$ satisfies $\textup{supp}(h_n) = [0,\zeta_n]$. Suppose also that $h(t)>0$ if and only if $t\in(0,\zeta)$, and that $h_n\to h$ and $\zeta_n\to \zeta$. For each $\ell\ge 1$, let $(t_n(\ell))$ be a sequence with $t_n(\ell)\in[0,\zeta_n]$ such that $t_n(\ell)\to t(\ell)\in[0,\zeta]$. 

Let $\T_n = (\T_n,\mu_n, \rho_n)$ (resp. $\T = (\T,\mu,\rho)$) be the weighted real tree coded by $h_n$ (resp. $h$) and define $x_{n;\ell}:=p_{h_n}(t_n(\ell))\in \T_n$ (resp. $x_\ell:=p_{h}(t(\ell))\in \T$) denote points under the canonical quotient $p_{h_n}:[0,\zeta_n]\to \T_n$ (resp. under the canonical quotient $p_{h}:[0,\zeta]\to \T$). Then
\begin{equation*}
    (\T_n, \mu_n, \rho_n,(x_{n;\ell})_{\ell\ge 1}) \to (\T,\mu,\rho,(x_\ell)_{\ell\ge 1})\qquad \textup{in   } \dghp^{(\infty)}. 
\end{equation*}
\end{lem}
\begin{proof} We will only prove convergence in the Gromov-Hausdorff topology, and not the Gromov-Hausdorff-Prohorov topology. The incorporation of the measure is standard, and we refer to \cite[Proposition 2.10]{ADH.14}.

Let us start more abstractly, and refer the reader to \cite[Lemma 2.4]{MR2203728}, and also \cite{MR2221786}. 
Given two metric spaces $\mathcal{X}$ and $\mathcal{Y}$, consider the disjoint union $\mathcal{Z} = \mathcal{X}\sqcup \mathcal{Y}$. Given any subset $\mathfrak{R}\subset \mathcal{X}\times \mathcal{Y}$ such that the canonical projections $\mathfrak{R}\to \mathcal{X}, \mathcal{Y}$ are each surjective, we can define the \textit{distortion} of $\mathfrak{R}$ as
\begin{equation*}
\operatorname{dis}(\mathfrak{R}) = \sup\left\{ \left| d_{\mathcal{X}}(x_1,x_2) - d_{\mathcal{Y}}(y_1,y_2) \right|:(x_i,y_i)\in \mathfrak{R}\right\}.
\end{equation*}
Note that when working with pointed metric spaces, the corresponding roots should be added to $\mathfrak{R}$, and this gives us a bound on the pointed GH distance. The distortion below includes the corresponding roots, and thus, we omit it in the above definition. 
With this, one can define $d_{\mathcal Z}$ a metric (which depends on $\mathfrak{R}$) on $\mathcal{Z}$ such that the restriction of $d_{\mathcal Z}$ to $\mathcal X$ (resp. $\mathcal Y$) is $d_{\mathcal X}$ (resp. $d_{\mathcal Y}$), and such that
\begin{equation*}
d_{\mathcal{Z}}(x,y)= \inf\left\{d_{\mathcal{X}}(x,x')+d_{\mathcal{Y}}(y,y') + \frac{1}{2} \text{dis}(\mathfrak{R}): (x',y')\in \mathfrak{R} \right\}, \qquad x\in \mathcal{X}, y \in \mathcal{Y}.
\end{equation*}
Note that whenever $(x,y)\in \mathfrak R$, then $d_{\mathcal Z}(x,y)=\frac{1}{2} \text{dis}(\mathfrak{R})$. 
With this, $\mathcal{Z}$ is a metric space, the canonical injections $\mathcal{X}, \mathcal{Y}\hookrightarrow \mathcal{Z}$ are isometric embeddings and
\begin{equation*}
d_H(\mathcal{X}, \mathcal{Y}) \le \frac{1}{2} \text{dis}(\mathfrak{R}).
\end{equation*}
Using the idea above Lemma 17 in \cite{ABBG.12} to construct the distortion between trees, define $\mathcal{Z}_n = \T_n\sqcup \T$ and let 
\begin{equation*}
\mathfrak{R}_n = \left\{(u,v)\in \mathcal \T_n\times \mathcal \T: \exists\ t\geq 0\text{ s.t. }p_{h_n}(\zeta_n \wedge t) = u, p_h(\zeta\wedge  t) = v \right\},
\end{equation*} where recall $p_{h_n}$ and $p_h$ are the canonical quotient maps.

By Equation (18) and Lemma 2.4 in \cite{MR2203728}, we have $\text{dis}(\mathfrak{R}_n)\le 4\sup_{t} |h(t)-h_n(t)|$. 
Also, by the construction of $d_{\mathcal Z_n}$ and since the roots of $\T_n$ and $\T$ are $\rho_n:=p_{h_n}(0)$ and $\rho:=p_h(0)$, we get $d_{\mathcal{Z}_n}(\rho_n,\rho) = \frac{1}{2} \text{dis}(\mathfrak{R}_n)$. 
Consider any $t_n,t\in [0,\zeta]$, and define $v_n:=p_{h_n}(t_n)\in \T_n$ and $v:=p_{h}(t)\in \T$.
Then, for any $s\in [0,\zeta]$, consider $(p_{h_n}(s),p_h(s))\in \mathfrak{R}_n$, then
\begin{equation}\label{eqnConvergenceOfMArkedTreesGivenConvergenceOfHeingthFns}
\begin{split}
 d_{\mathcal{Z}_n}(p_{h_n}(t_n),p_h(t)) \le & \frac{1}{2}\textup{dis}(\mathfrak{R}_n)
+ |h_n(t_n)+h_n(s)-2\inf_{t_n\wedge s\le u\le t_n\vee s} h_n(u)|\\
&+ |h(t)+h(s)-2\inf_{t\wedge s \le u\le t\vee s} h(u)|.
\end{split}
\end{equation}
Denote by $\psi$ the modulus of continuity of $h$, that is $\psi(\delta)=\sup_{|t-s|\leq \delta }|h(t)-h(s)|$ with $\delta\in [0,\zeta]$. 
Note that $\psi(\delta)\to 0$ as $\delta\downarrow 0$, by uniform continuity. 
Thus, taking $(p_{h_n}(t_n),p_h(t_n))\in \mathfrak{R}_n$ in \eqref{eqnConvergenceOfMArkedTreesGivenConvergenceOfHeingthFns}, we simplify 
\begin{equation*}
\begin{split}
 d_{\mathcal{Z}_n}(p_{h_n}(t_n),p_h(t)) \le  &2\sup_{t} |h(t)-h_n(t)|
+ |h(t)+h(t_n)-2\inf_{t\wedge t_n \le u\le t\vee t_n} h(u)|\\
 \leq &2\sup_{t} |h(t)-h_n(t)|+2\psi (|t-t_n|).
\end{split}
\end{equation*}

Now, using the definition \eqref{eqnDefinitionGHPKMArkedSpace}, for any $\epsilon>0$ we fix a large $k$ enough such that $\sum_{n\geq k}2^{-n}<\epsilon$. 
Hence, considering $x_{n;\ell}=p_{h_n}(t_n(\ell))$ and $x_\ell=p_{h}(t(\ell))$, as $n\to\infty$
\begin{equation*}
\begin{split}
& d_H(\T_n,\T) + \max \big\{d_{\mathcal{Z}_n}(\rho_n,\rho) ,\max_{\ell\leq k} d_{\mathcal{Z}_n} (x_{n;\ell}, x_\ell)\big\}\\
&\qquad  \le 4\|h-h_n\|_{\infty} +  2\max_{\ell\leq k} \psi(|t(\ell)-t_n(\ell)|),
\end{split}
\end{equation*}which can be made arbitrarily small for such fixed $k$. 
This proves the result.
\end{proof}

Using Lemma \ref{lem:inductionProof} recursively, it is easy now to show that all the subdecorations $(\overline{\cD}^{n}(\bp);\bp \in \bbU)$ converge jointly.

\begin{cor}
Suppose that $r>0$ and $r_n = \lfloor b_n^{i_0} r\rfloor$. Then $\overline{\cD}^n(\bp)\weakarrow \cD(\bp)$ jointly for all $\bp\in \bbU$ where $\sG(\cD_r)\sim \bM^{i_0}_r$.
\end{cor}
\begin{proof}[Sketch of the proof:]
The hypothesis of the convergence \eqref{eqn:enConverge1} in Lemma \ref{lem:inductionProof} for $\bp=\emptyset$, has been proved in Lemma \ref{lem:longExc}.
This accounts for the reduced height $h=0$. 
For each reduced height $h\geq 1$ and every $n$, consider sequences of independent random walks $(X^{i,j}_{h,n})_{i,j}$. The sequence $(X^{i,j}_{h,n})_{i,j}$ codes the subdecorations $\cD^n(\bp)$ with $|\bp|=h$. Since they are independent, by \ref{ass:a1} and Skorohod's representation theorem,  we can rescale them and assume all of them converge jointly a.s. 
Then by Proposition \ref{prop:good} we can assume that the excursions and excursion lengths also converge almost surely. Thus, from Item (1) of Lemma \ref{lem:inductionProof}, we have $\overline{\cD}^{n}(\emptyset)\to \cD(\emptyset)$, where such subdecorations are conditioned on their length. Item (2) shows the convergence of the rescaled excursions, height functions and excursion lengths that grow from 
$\overline{\cD}^{n}(\emptyset)$.
Hence, by Item (1) of the same Lemma \ref{lem:inductionProof} all subdecorations $(\overline{\cD}^{n}(\bp);\bp \in \bbU,|\bp|=1)$ converge.
Continuing in this way, we obtain the result.
\end{proof}

\subsection{Establishing \ref{ass:2thm1}}

In this section we show that Assumptions \ref{ass:a1}--\ref{ass:A6} are sufficient for Assumption \ref{ass:2thm1} to hold. Recall that informally, Assumption \ref{ass:2thm1} tells us that there are not too many long excursions nor too many tall excursions which encode the scaled subdecorations $\overline{\cD}^n(\bp)$. A key tool we use in establishing \ref{ass:2thm1} (and \ref{ass:3thm1}) is Corollary \ref{cor:DecorationSpalf} and so we will first need to introduce some notation.

As before, we fix $i_0\in [d]$ and some $r>0$ such that $\vec \bN^{i_0}(\zeta=r)  = 0$ and let $r_n = \fl{b_n^{i_0}r}$. 
The objective now is to fix the subtree of type $i_0$ connected to the root, and describe how to code the \emph{forest} that grows from it.
Recalling the notation of Lemma \ref{lem:inductionProof}, generate $(\vec e_{\emptyset,n},h_{\emptyset,n})=\big((e_{\emptyset,n}^{j};j\in [d]),h_{\emptyset,n}\big)\sim \vec N^{i_0,\vec \varphi_n}(-|{\zeta_{\emptyset,n} }> r_n)$ and independently, for any $j\in [d]$ let $\vec X_n^j$ denote $\Z^d$-walks starting at zero. 
By Lemma \ref{lem:longExc} and Skorohod's representation theorem, we can and do suppose that almost surely
\begin{align}\label{eqn:emptysetConverge}
a_n\diag(\vec{b}_n)^{-1} \vec e_{\emptyset,n}(\fl{b_n^{i_0}\,\cdot })&\longrightarrow \vec \be_\emptyset(\cdot),&&\frac{1}{a_n} h_{\emptyset,n}(\fl{b_n^{i_0} \,\cdot}) \longrightarrow h_{\emptyset}(\cdot)&&\textup{and} & \frac{\zeta_{\emptyset,n}}{b_n^{i_0}} \to \zeta.
\end{align}
For $j\in [d]$, define 
\begin{equation}\label{eqn:tildeX}    \widetilde{X}^j_n(\cdot) = (e^{j}_{\emptyset,n}(\zeta_{\emptyset,n}) \vee 0) \vec{\epsilon}_j + \vec X^j_n(\cdot),
\end{equation}
where recall that $\vec{\epsilon}_j$ is the $j^{\text{th}}$ unit vector.
The corresponding height functions will be denoted by $\widetilde H^j_n$. 
Similarly to \eqref{eqn:Rrecurse2} and \eqref{eqn:Rrecurse}, define the random times
\begin{equation}\label{eqnTildeUhn}
\begin{split}
\widetilde U_{n}^{(0),j} := \min\{m: \widetilde{X}^{j,j}_n(m) = 0\},  \qquad  \widetilde U_{n}^{(h+1),j} := \min\left\{m: \widetilde{X}^{j,j}_n(m) = - \sum_{\ell\neq j} \widetilde{X}^{\ell,j}_n (\widetilde U^{(h),\ell}_{n})\right\},
\end{split}
\end{equation}for $h\geq 0$. 
Note that they are equivalent to
\begin{align*}
\widetilde U_n^{(0),j} &:= \min\{m: X^{j,j}_n(m) = -(e^{j}_{\emptyset,n}(\zeta_{\emptyset,n}) \vee 0)\},\\
  \widetilde U_n^{(h+1),j} &:= \min\left\{m: X^{j,j}_n(m) =-(e^{j}_{\emptyset,n}(\zeta_{\emptyset,n}) \vee 0) - \sum_{\ell\neq j} X^{\ell,j}_n (\widetilde U^{(h),\ell}_n)\right\}.
\end{align*}That is, the random times $(\widetilde U^{(h),j}_{n};j\in [d],h\in \mathbb{N})$ code the size of a multitype forest starting with $e^{j}_{\emptyset,n}(\zeta_{\emptyset,n})\vee 0$ vertices type $j$, for every $j\in [d]$ (as explained in Remark \ref{rmrkCorollaryCouplingGluedAndRWs}).
Note that such a forest, codes all subtrees growing from the subtree of type $i_0$ connected to the root, of $T_{n,r_n}$. 
Therefore, by Lemma \ref{lem:spalfAndDiscTree} and its proof, jointly over $j\in[d]$ and $h\ge 1$, 
\begin{align*}
    \#&\{v\in T_{n,r_n}: \TYPE(v) = j, v\textup{ is at reduced height }= h\} = \widetilde U^{(h-1),j}_n-\widetilde U^{(h-2),j}_n,
\end{align*}where we set $\widetilde U^{(-1),j}_{n}: = 0$ for all $n$ and $j\in[d]$. 

By Lemma \ref{lem:spalfAndDiscTree} and Corollary \ref{cor:DecorationSpalf}, we can (and do) couple for any $j\in [d]$ the walks $\widetilde{X}^j_n$ with the decorations $\cD^n_{r_n}$ in such a way that for any value $y>0$, and any reduced height 
$h\geq 1$, we have
\begin{equation}\label{eqnCouplingDiscreteRescaledExcursionsWithLength}
\begin{split}
\#\{&\bp: |\bp| = h, \langle \vec{\mu}_\bp,\vec{\epsilon}_j\rangle (\cD^n(\bp)) > b_n^jy ,\,\bp = \bq \ell \textup{ with }\ell\equiv j \mod d\}\\
&{=}\#\bigg\{\textup{excursions of }\frac{a_n}{b_n^j}\widetilde X^{j,j}_n(\cdot)\textup{ which have} \\
&\qquad\qquad\textup{length longer than }b_n^jy \textup{ occurring between }\widetilde U^{(h-2),j}_n \textup{ and }\widetilde U^{(h-1),j}_n\bigg\},
\end{split}
\end{equation}
where recall from Lemma \ref{lem:decoration}, that we are using the counting measure.
The aforementioned coupling also satisfies for any $j\in [d]$, $h\geq 1$ and $y>0$
\begin{align}
\nonumber \#\{&\bp: |\bp| = h, \textup{ht}(\overline{\cD}^n(\bp)) > y,\, \bp = \bq \ell \textup{ with }\ell \equiv j \mod d\}\\
\label{eqnCouplingDiscreteRescaledExcursionsWithHeight}&{=}\#\left\{\textup{excursions of }\frac{1}{a_n}\widetilde H^j_n(\cdot)\textup{ which exceed height }y \textup{ occurring between }\widetilde U^{(h-2),j}_{n} \textup{ and }\widetilde U^{(h-1),j}_{n}\right\},
\end{align}where $\textup{ht}(\mathcal{X}) := \sup_{x\in \mathcal{X}} d(\rho,x)$, for any rooted compact metric space $(\mathcal{X},d,\rho)$, and where we used the rescaling of the subdecorations $\overline \cD^n(\bp)$ given in \eqref{eqnRescaledSubdecoration}. 

In the next lemma, the coupling in \eqref{eqnCouplingDiscreteRescaledExcursionsWithLength} will be used to prove the first part of \ref{ass:2thm1}, whereas the one in \eqref{eqnCouplingDiscreteRescaledExcursionsWithHeight} will be used to prove the second part. 
We emphasize that the reduced height $h = 0$ is not included in such coupling formulas, as this is reserved for the type $i_0$ subtree coded by the excursion $\vec e_{\emptyset,n}$.

\begin{lem}\label{lem:verifyH2}
Under \ref{ass:a1}--\ref{ass:a4}, the condition \ref{ass:2thm1} holds for the rescaled decoration $\overline{\cD}^n$. 
\end{lem}
\begin{proof}
By Skorohod's representation theorem, we can and do suppose that the convergence in \ref{ass:a1} and \eqref{eqn:emptysetConverge} both hold almost surely. 
Since Lemma \ref{lem:good} {and \ref{G4} imply} that $e^j_\emptyset$ is continuous at $\zeta$, then for $j\neq i_0$
\[
\frac{a_n}{b_n^j}\widetilde X^{j,j}_n(0)\to \be^j_\emptyset(\zeta)\in  [\be^j_\emptyset(r),\infty),
\]by \eqref{eqn:emptysetConverge}, Lemma \ref{lem:JSlem}  and \ref{ass:a4}.
Furthermore, jointly for all $i,j\in[d]$ almost surely
\begin{equation}\label{eqn:limitXtildeDef}
 \left(\frac{a_n}{b_n^j}\widetilde{X}_n^{i,j}\big(\fl{b_n^it}\big);t\ge 0\right) \to \left(\be^{i}_\emptyset(\zeta)\mathbf{1}_{[i=j]} + X^{i,j}(t);t\ge 0\right)=:\widetilde{\bX}^{i,j}.
\end{equation} 
First we prove the lemma for $h=1$. The idea is to show that for any $\eps>0$, there exists a finite set $I(\eps)\subset \{\bp\in \bbU:|\bp|=1 \}$ independent of $n$, such that all the mass of the subtrees not in $I(\eps)$ is at most $\eps$.
More formally, if $B(\eps):=\{\bp\in \bbU:\bp=\bq\bq',\bq\in I(\eps),\bq'\in\bbU\}$, we show that for any $n$ big enough
\begin{equation}\label{eqnVerifyH2}
    \sum_{\substack{ \bp\notin B(\eps)\\ |\bp| \geq 1}} \|\vec{\mu}_\bp^n\|(D^n_{\bp}) \le\eps \qquad\text{and}\qquad \sup_{\substack{\bp\notin B(\eps)\\ |\bp|\geq 1 }} \textup{diam}(D^n_{\bp}) \le \eps.
\end{equation}To prove \eqref{eqnVerifyH2}, the idea is the following. 
Define $\widetilde r_n$ as the vector with $j^{\mbox{th}}$ entry $e^j_{\emptyset,n}(\zeta_{\emptyset,n})\vee 0$ for all $j\in [d]$. 
Then, from  \eqref{eqn:minSolDiscrete}, the discrete multidimensional hitting time $\widetilde T_n(\widetilde r_n):=(T^1_n(\widetilde r_n),\dotsm ,T^d_n(\widetilde r_n))$ coming from the pair $(\widetilde r_n,X^{i,j}_n)$, codes the total mass of the forest type $j$ growing from the type $i_0$ subtree attached to the root.
By \eqref{eqn:tildeX}, such hitting time can be rewritten as
\begin{align*}
&\min\left\{\vec{m}\in \Z^d_+: \sum_{i=1}^d X^{i,j}_n(m_i) = -e^j_{\emptyset,n}(\zeta_{\emptyset,n})\vee 0,\quad \forall j\textup{ s.t. }m_j<\infty\right\}\\
& =\min\left\{\vec{m}\in \Z^d_+: \sum_{i=1}^d \widetilde X^{i,j}_n(m_i) = 0,\quad \forall j\textup{ s.t. }m_j<\infty\right\}=: \widetilde T_n(\vec 0).
\end{align*}
By Lemma \ref{lem:spalfAndDiscTree} and Corollary \ref{cor:DecorationSpalf} (similarly also to \eqref{eqnCouplingDiscreteRescaledExcursionsWithLength}), for all $j\in [d]$ the total mass $\widetilde T^j_n(\vec 0)$ of vertices type $j$ at reduced height at least one satisfies
\begin{equation}\label{eqnTotalSizeAboveHeightZeroWithSizeOfSubdecorations}
\widetilde T^j_n(\vec 0)=\sum_{\substack{ |\bp| \geq 1\\ \bp\equiv j\mod d}} \|\vec{\mu}_\bp\|(D^n_{\bp}).
\end{equation}Hence, it is enough to show that for any $\eps>0$, we can cover almost all mass of vertices type $j$, with the $K^j_\eps$ longest excursions of $\widetilde X^{j,j}_n$, before hitting $\widetilde T^j_n(\vec 0)$, for some $K^j_\eps$. 

As the first hitting times will be important for the sequel, for every $\vec r\in \mathbb{R}^d_+\setminus \{\vec 0\}$, we will write $\widetilde \bT(\vec{r}):=(\widetilde \bT^1(\vec{r}),\dotsm, \widetilde \bT^d(\vec{r}))$ for the minimal solution to \eqref{eqn:minSol} for the limiting field $\widetilde{\bX}$ (that is, the solution of such equation for the pair $(\vec{r},\widetilde{\bX})$).
The results of Chaumont and Marolleau recalled in Theorem \ref{thm:CM} above, imply that $\PR(\widetilde \bT(\vec{r})\in \R^d_+) = 1$ for every $\vec{r}\in \R^d_+$ under Assumptions \ref{ass:a2} and \ref{ass:a4}. See also Remark \ref{rmk:(H)Satisfies}.

We now work with $\widetilde \bT(\eta\vec{\boldsymbol{1}})$ for fixed $\eta>0$, where $\vec{\boldsymbol{1}}$ is the vector of size $d$ with all of its entries one. 
The value of $\eta$ is not important in this proof, since we just need
$\eta\vec{\boldsymbol{1}}>\vec{\boldsymbol{0}}$. 
The latter, together with \cite[Lemma 2.2(ii)]{MR3449255} or \cite[Lemma 2.3(2)]{MR4193902}, imply the monotonicity of the first hitting times $\widetilde \bT(\vec{\boldsymbol{0}})\leq \widetilde \bT(\eta\vec{\boldsymbol{1}})$.
(In the next subsection, we will choose $\eta$ sufficiently small.)

Since $\PR(\widetilde \bT(\eta\vec{\boldsymbol{1}}) \in \R_+^d) = 1$, then any $\delta\in (0,1)$ there exists a $t_\delta>0$ such that 
\begin{equation*}
    \PR\left(\widetilde \bT(\eta\vec{\boldsymbol{1}})\le t_\delta\vec{\boldsymbol{1}} \right)  \ge 1-\delta.
\end{equation*}

By Theorem \ref{thm:CM} it holds with probability 1 that the field $\widetilde{\bX}$ is continuous to the left at each coordinate $\widetilde \bT^i(\eta\vec{\boldsymbol{1}})$. 
Therefore, on the event $A_\delta = \{\widetilde \bT(\eta\vec{\boldsymbol{1}})\le t_\delta\vec{\boldsymbol{1}}\}$, 
\begin{equation*}
    \frac{a_n}{b_n^j}\widetilde{X}_n^{i,j}\Big(\fl{b_n^i \widetilde \bT^i(\eta\vec{\boldsymbol{1}})}\Big) \longrightarrow \widetilde{\bX}^{i,j}(\widetilde \bT^i(\eta\vec{\boldsymbol{1}}))\qquad \textup{ a.s. on }A_\delta
\end{equation*} by \eqref{eqn:limitXtildeDef} and Lemma \ref{lem:JSlem}.
By the definition of the first hitting time and the above display, we have
\begin{equation*}
    \frac{a_n}{b_n^j}\sum_{i=1}^d \widetilde X_n^{i,j}\Big (\fl{b_n^i \widetilde \bT^i(\eta\vec{\boldsymbol{1}})}\Big) \longrightarrow -\eta \qquad\textup{ a.s. on }A_\delta.
\end{equation*}
So in particular, by the monotonicity of the first hitting times of \cite[Lemma 2.2(ii)]{MR3449255}, a.s. on $A_\delta$ there exists an $n = n(\delta, \omega)$ sufficiently large such that
\begin{equation}\label{eqnBoundingDiscreteFirstHittingTimeAtZeroWithContinuousFirstHittingTimeAtOne}
\min\left\{\vec{t}: \sum_{i=1}^d \widetilde{X}^{i,j}_n(\fl{b_n^i t_i}) = 0\textup{  for all }j\right\} \le \widetilde \bT(\eta\vec{\boldsymbol{1}}). 
\end{equation}
This implies for all $n$ sufficiently large that
\begin{equation*}
  \liminf_{n\to\infty} \PR\left(\min\left\{\vec{t}: \sum_{i=1}^d \widetilde{X}^{i,j}_n(\fl{b_n^i t_i}) = 0\textup{  for all }j\right\} \le t_\delta\vec{\boldsymbol{1}}\right) \ge 1-2\delta.
\end{equation*}
In particular, the size of the subdecorations $(\overline{\cD}^n(\bp))$ with reduced height at least one, is coded by the length of excursions of $(\frac{a_n}{b_n^j}\widetilde{X}_n^{i,j}(\fl{b_n^i t_i}); t_i\in[0,t_\delta])$ via Corollary \ref{cor:DecorationSpalf}, with probability exceeding $1-2\delta$.

Fix $\eps>0$ and $i\in [d]$. 
By Lemma \ref{lem:good}, we can find some random $K  = K(\delta,\eps,\omega)$, such that the length of the $K$ longest excursions of $\widetilde\bX^i$ on $[0,t_\delta]$ including the a.s. unfinished excursion straddling $t_\delta$, cover more than $t_\delta-\eps$.
Consider any $k\in \mathbb{N}$ sufficiently large, such that $K\leq k$ with probability at least $1-\delta$. 
Thus, denoting by $(l_1,r_1),(l_2,r_2),\dotsm,(l_K,r_K)$ such lengths, we have
\begin{equation*}
    \sum_{q=1}^k (r_q-l_q) >t_\delta-\eps,
\end{equation*}with probability at least $1-\delta$.
A similar result holds for the excursions of $(\frac{a_n}{b_n^j}\widetilde{X}_n^{i,j}(\fl{b_n^i t_i}); t_i\in[0,t_\delta])$ by Proposition \ref{prop:good}. Thus, the measure of the length of the union of the $k$ longest excursions of such discrete processes, again possibly included an unfinished excursion at the end, have union covering at least a set of measure $t_\delta-2\eps$ for all $n$ sufficiently large. 
Repeating the same argument for every $i\in [d]$, and since the mass of each type $i$ subtree is precisely the length of an excursion, we have established \eqref{eqnVerifyH2} holds 
on a set of probability at least $1-3\delta$.

The first part of Assumption \ref{ass:2thm1} for general $h\geq 2$, follows the same reasoning: find $K=K(h,\delta,\eps,\omega)$ excursions at reduced height $h$ that contain almost all the mass, in the continuous and then in the discrete setting. 
For that one can use \eqref{eqn:Rrecurse2},  \eqref{eqn:Rrecurse}, the definition of the random times $\widetilde U_{n}^{(h),j}$, the coupling in  \eqref{eqnCouplingDiscreteRescaledExcursionsWithLength} and Proposition \ref{prop:good}.

The second part of Assumption \ref{ass:2thm1} follows from similar reasoning as above. 
For that we use \eqref{eqnCouplingDiscreteRescaledExcursionsWithHeight}, invoke the convergence of the height processes from Proposition \ref{prop:dlHeight} and the inequality $\textup{diam}(\cX)\le 2 \textup{ht}(\cX)$.
These details are omitted. 
\end{proof}

\subsection{Establishing \ref{ass:3thm1}}

Recall the MBGW branching process $\vec Z_n$ defined in Section \ref{sec:associatedBranchingProcesses}, and the single-type process $Y^i_n$. In particular, recall it starts with $\lfloor b^i_n/a_n\rfloor $ many individuals type $i$, for all $i\in[d]$.

Let us note two consequences of Assumption \ref{ass:A6}. The first is that Assumption \ref{ass:a5} follows from Assumption \ref{ass:A6} by a coupling of $Y_n$ and $Z_n$.

The second is that if for every $\eps >0$ we define $\vec Z_{n,\eps}=(\vec Z_{n,\eps}(m);m\in \mathbb{N})$ with $j^{\mbox{th}}$ entry given by 
\begin{equation*}
Z_{n,\eps}^j(m+1) = \fl{\frac{b_n^j \eps}{a_n}} + \sum_{i=1}^d X_{n}^{i,j} \left(\sum_{\ell=0}^m Z_{n,\eps}^i(\ell)\right),
\end{equation*} then by the branching property and since $\fl{b_n^j k^{-1}/a_n}k\leq \fl{b_n^j/a_n}$ for large $n$ and for all $k\geq 1$, we have $\PR\left(\vec Z_n(a_n\delta) = \vec {\boldsymbol{0}} \right) \leq  \PR\left(\vec Z_{n,k^{-1}} (a_n\delta) = \vec {\boldsymbol{0}}\right)^k.$ 
Hence, for all $\delta>0$
\begin{equation}\label{eqnZ_n_eps}
\liminf_{\eps\downarrow 0} \liminf_{n\to \infty} \PR\left(\vec Z_{n,\eps}(a_n\delta) = \vec {\boldsymbol{0}} \right) = 1.
\end{equation}

\begin{lem}\label{lem:verifyH3}
Suppose Assumptions \ref{ass:a1}--\ref{ass:A6} hold. Then with high probability \ref{ass:3thm1} holds for the rescaled decoration $\overline{\cD}^n$. 
\end{lem}

\begin{proof} We first establish \eqref{eqn:massH3}. By Skorohod's representation theorem, we work on a probability space such that the convergence in Assumption \ref{ass:a1} and Equation \eqref{eqn:emptysetConverge} hold almost surely. Recall the definitions of $\widetilde{X}^{i,j}_n$ and $\widetilde{\bX}^{i,j}$, given in \eqref{eqn:tildeX} and \eqref{eqn:limitXtildeDef}, respectively. 
Let $(\widetilde U_n^{(h),j};h\geq 0)$ be the random times defined in \eqref{eqnTildeUhn} for  $(\widetilde{X}^{i,j}_n(t))$. Similarly, define $(\widetilde U^{(h),j};h\geq 0)$ as the random times defined in \eqref{eqn:Rrecurse}, using $(\widetilde \bX^{i,j})$ with initial conditions $r_j = 0$ for all $j\in [d]$.  
Note that for every $h\ge 0$ and $j\in[d]$ we have 
\begin{equation}\label{eqnConvergenceDiscreteUhToUh}
\frac{1}{b_n^j}\widetilde U_n^{(h),j} \to \widetilde U^{(h),j},
\end{equation}as $n\to \infty$, by Lemmas \ref{lem:firstPassing} and \ref{lem:firstHitting1} and a simple induction argument.

Now, we use the same setting for the first hitting times $(\widetilde \bT(\eta \vec{\boldsymbol{1}});\eta>0)$, introduced in the proof of Lemma \ref{lem:verifyH2}. 
That is, we denote by $\widetilde \bT(\vec 0):=(\widetilde \bT^1(\vec 0),\dotsm, \widetilde \bT^d(\vec 0))$ to be the minimal solution to \eqref{eqn:minSol} for the limiting field $\widetilde{\bX}$ (that is, the solution of such equation for the pair $(\vec 0,\widetilde{\bX})$).
Recall that $\eta\mapsto \widetilde \bT(\eta \vec {\boldsymbol{1}})$ is left-continuous by Lemma 2.3 in \cite{MR4193902}, and so on a set of probability one, for any $\eps>0$, there exists $\eta = \eta(\eps,\omega)>0$ such that $\|\widetilde \bT(\eta \vec {\boldsymbol{1}})-\widetilde \bT(\vec {\boldsymbol{0}})\|_1\le \eps$ where $\|\vec{x}\|_1$ is the $\ell^1$ norm. Also, by the construction in Subsection \ref{subsecConstructionFirstHittingTime} we have $\widetilde \bT(\vec {\boldsymbol{0}}):= \lim_{h\to\infty} \widetilde U^{(h)}$, where $\widetilde{U}^{(h)}:=(\widetilde U^{(h),1},\dotsm, \widetilde U^{(h),d})$. 
Note that the previous convergence is monotone, since each component of $\widetilde U^{(h)}$ is non-decreasing.
Thus, there exists $h = h(\eps,\omega)$ sufficiently large such that $\|\widetilde \bT(\vec{\boldsymbol{0}}) - \widetilde {U}^{(h)}\|_1\le\eps$. In particular, the triangle inequality gives \begin{equation}\label{eqnComparingWidetildeTWithUh}
\|\widetilde \bT(\eta\vec {\boldsymbol{1}}) - \widetilde U^{(h)}\|_1\le 2\eps.
\end{equation}  

Similarly as before, for any $\gamma$ small, we can assume \eqref{eqnComparingWidetildeTWithUh} holds for any (deterministic) $h$ large enough, with probability at least $1-\gamma$. 

Recall from \eqref{eqnBoundingDiscreteFirstHittingTimeAtZeroWithContinuousFirstHittingTimeAtOne}, that one can deduce that a.s. for all $n$ large enough (but depending on $\omega$) we have $b_n^j\widetilde \bT^j(\eta\vec{\boldsymbol{1}})\ge \widetilde  T^j_n(\vec{\boldsymbol{0}})$.
Here we used  $\widetilde T_n(\vec{\boldsymbol{0}}):=(\widetilde T^1_n(\vec{\boldsymbol{0}}),\dotsm,\widetilde{T}^d_n(\vec{\boldsymbol{0}}))$ for the multidimensional first hitting time of the pair $(\vec{\boldsymbol{0}},\widetilde X_n)$.
Similarly to \eqref{eqnCouplingDiscreteRescaledExcursionsWithLength}, we can deduce that  $\widetilde U^{(h-1),j}_n$ is the total number of type $j$ vertices in $\bigsqcup_{\bp\equiv j\mod d:0< |\bp|\le h} \cD^n(\bp)$. 
Taking $h\to\infty$ we see that $\widetilde U^{(\infty),j}_n$ is the total number of type $j$ vertices in the subdecoration $\cD^n$, excluding the one with $\bp = \emptyset$. 
The latter is also the $j^{\mbox{th}}$ component of the multidimensional first hitting time, by \eqref{eqnTotalSizeAboveHeightZeroWithSizeOfSubdecorations}. 
A similar observation was given in \cite[Section 2]{MR3449255}. 
Thus, Equation \eqref{eqnBoundingDiscreteFirstHittingTimeAtZeroWithContinuousFirstHittingTimeAtOne}
implies that
\begin{equation}\label{eqnRelatingWidetildeUnInftyWithWidetildeTeta}
\sum_{\substack{ |\bp| \geq 1\\ \bp\equiv j\mod d}} \|\vec{\mu}^n_\bp\|(D^n_{\bp})=\frac{\widetilde{ U}^{(\infty),j}_n}{b_n^j}\leq \widetilde \bT^j(\eta\vec{\boldsymbol{1}}).
\end{equation}

We can conclude from  \eqref{eqnConvergenceDiscreteUhToUh}, \eqref{eqnComparingWidetildeTWithUh} and \eqref{eqnRelatingWidetildeUnInftyWithWidetildeTeta} that for $h$ sufficiently large
\begin{equation}\label{eqn:Ukclose}
\limsup_{n\to\infty}\sum_{ |\bp| > h} \|\vec{\mu}^n_\bp\|(D^n_{\bp})=\limsup_{n\to\infty}\sum_{j=1}^d\left( \frac{1}{b_n^j} \widetilde U_n^{(\infty),j} - \frac{1}{b_n^j}\widetilde U_n^{(h-1),j} \right) \le 4\eps,
\end{equation}with high probability. This proves \eqref{eqn:massH3}.

It remains to prove the second part of \ref{ass:3thm1}. Let us fix $\delta,\eta>0$ arbitrary (and deterministic). By \ref{ass:A6} and \eqref{eqnZ_n_eps}, there exists an $\eps>0$ sufficiently small such that 
\begin{equation*}
\limsup_{n\to\infty} \PR( \vec Z_{n,3\eps}(a_n\delta) \neq \vec{\boldsymbol{0}}) \le \eta.
\end{equation*}

Recall the earlier discussion in which we showed the non-decreasing convergence  $\widetilde \bT(\vec {\boldsymbol{0}}):= \lim_{h\to\infty} \widetilde U^{(h)}$.
Thus, by Theorem \ref{thm:CM}, jointly for all $i, j\in [d]$ we have the a.s. convergence
\begin{equation}\label{eqn:UhtoTi}
\widetilde\bX^{i,j}\left(\widetilde U^{(h),i} \right) \longrightarrow \widetilde\bX^{i,j}(\widetilde \bT^{i}(\vec{\boldsymbol{0}}))\qquad\textup{ as }h\to\infty.
\end{equation}
Consequently, for every $\eta>0$ and $\eps>0$ there exists $h<\infty$ sufficiently large such that
\begin{equation*}
\PR\left(\sum_{i\neq j} \big(\widetilde \bX^{i,j}(\widetilde \bT^i(\vec{\boldsymbol{0}})) - \widetilde \bX^{i,j}(\widetilde U^{(h),i}) \big)\le \eps,\quad \forall \ j\in[d]\right) \ge 1-\eta.
\end{equation*}

Let us prove that for any $h\in \Z_+$, for all $i\neq j$ we have
\begin{equation}\label{eqnConvergenceXtildeUtildeDiscreteToXtildeUtildeContinuous}
\frac{a_n}{b_n^j} \widetilde{X}_{n}^{i,j}  \left(\fl{ b_n^i\frac{1}{b^i_n}\widetilde U^{(h),i}_n}\right)  \to \widetilde \bX^{i,j}(\widetilde U^{(h),i})
\end{equation}almost surely. 
For that, we will use an induction argument over $h$. 
Indeed, first let us show that 
the whole vector $\widetilde \bX^i$ is continuous at $\widetilde U^{(0),i}$. 
Note that $\widetilde \bX^{i,i}(0)>0$,  and $\widetilde U^{(0),i}$ is the first hitting time to zero. 
Using \ref{ass:a2}, by Lemma \ref{lem:firstHitting1}, we have that $\widetilde \bX^{i,j}$ is continuous at $\widetilde U^{(0),i}$.
Hence, the convergence in \eqref{eqnConvergenceXtildeUtildeDiscreteToXtildeUtildeContinuous} holds for $h=0$ by Assumption \ref{ass:a1}, \eqref{eqnConvergenceDiscreteUhToUh}, and Lemma \eqref{lem:JSlem}. 
For the general case, using the same idea we conclude the induction hypothesis.

Then, by Fatou's Lemma for all $h$ large enough
\begin{equation*}
\liminf_{n\to\infty} \PR\left(\frac{a_n}{b_n^j} \sum_{i\neq j}   \left(\widetilde{X}_{n}^{i,j} \left( \widetilde U^{(h-1),i}_n\right)  - \widetilde{X}_n^{i,j}\left( \widetilde U^{(h-2),i}_n\right)\right) \le 2\eps,\quad \forall\ j\in[d]\right) \ge 1-\eta.
\end{equation*}Define $(\widetilde R^{(h),j}_n)$ as in \eqref{eqn:Rrecurse2} for $\widetilde X_n$, and initial condition $\vec r=\vec{\boldsymbol{0}}$.
Then, the above display is simply stating that $\widetilde R^{(h),j}_n - \widetilde R^{(h-1),j}_n \le \frac{2\eps b_n^j}{a_n}$ for all $j$ with probability at least $1-\eta$. By Lemma \ref{lem:spalfAndDiscTree} and similarly as in the derivation of \eqref{eqnCouplingDiscreteRescaledExcursionsWithLength}, this tells us that the number of type $j$ trees added to the glued decoration at height $h+1$ (i.e. $\cD^n(\bp): |\bp| = h+1$, and $\bp=\bq \ell$ with $\ell\equiv j\mod d$) is at most $\frac{b_n^j}{a_n} 2\eps$ with probability $1-\eta$. 

On the decoration $\cD^n$, and for any $h\ge 0$, we can define the process \sloppy$ \widetilde{Z}_{n,(h)}(m) = ( \widetilde{Z}_{n,(h)}^1(m),\dotsm, \widetilde{Z}_{n,(h)}^d(m))\in \Z^d$ by
\begin{equation*}
\widetilde{Z}_{n,(h)}^j(m) = \langle\vec\mu, \vec{\epsilon}_j\rangle\left(\{y\in \sG(\cD^n): \di(y,\sG_{\le h}(\cD^n)) = m+1\}\right).
\end{equation*}
In words, $\widetilde{Z}_{n,(h)}^j(m)$ is the number of vertices at distance $m+1$ of $\sG_{\le h}(\cD^n)$, and on subdecorations of type $j$. 
The latter is a MBGW process starting with $\widetilde Z^j_{n,(h)}(0)=\widetilde R^{(h+1),j}_n-\widetilde R^{(h),j}_n$. 
Choosing $h$ large enough defined in the preceding paragraph, we can couple the MBGW branching process $\vec Z_{n,3\eps}$ and the decoration $\cD^n$ such that there exists a set $A  = A(\eps,\delta,\eta)$ of probability at most $\eta$
\begin{equation*}
\left\{\sup_{z\in D^n_{{\bp}}; |\bp|>h+1} \di \left(z, \sG_{\le h+1}(\cD^n)\right) >a_n \delta\right\} \subset \left\{ Z_{n,3\eps}\big(\lfloor a_n\delta \rfloor \big) \neq \vec{0} \right\}\cup A.
\end{equation*}

Hence for any $\eta>0$ and $\delta>0$ there is some $h\ge 1$ such that 
\begin{equation*}
\PR\left(\sup_{(z,\bp):|\bp|> h+1,z\in D^n(\bp)} \di\left(z, \sG_{\le h+1}\left( \cD^n\right)\right) \le a_n\delta \right) \ge 1-\eta.
\end{equation*} Since $\eta\in(0,1)$ was arbitrary, this implies that the second condition of hypothesis \ref{ass:3thm1} holds with arbitrarily large probability. 

\end{proof}

\begin{proof}[Proof of Theorem \ref{thm:MAIN}]
By Theorem \ref{thm:graphConv1} and Lemmas \ref{lem:inductionProof}, \ref{lem:verifyH2}, and \ref{lem:verifyH3} we know that $\sG(\overline{\cD}^n) \weakarrow \sG(\cD)$. Since the MBGW tree can be constructed by that decoration (Lemma \ref{lem:decoration} and Corollary \ref{cor:DecorationSpalf}), we get the desired result.
\end{proof}


\begin{center}
{\bf NOTATION}
\end{center}

\begin{table}[H]
\begin{tabularx}{\textwidth}{p{0.30\textwidth}X}
	\multicolumn{2}{l}{{\underline{Spaces:}}}\\
	$\bbU$ &Ulam-Harris tree.\\
	$\U = (\U,\ast,\delta)$ &Urysohn space with fixed point $\ast\in \U$ and metric $\delta$.\\
	$\mathfrak{M}$,  $\mathfrak{M}^{(k)}_{\tiny \mbox{mp}}$, $\mathfrak{M}_{\tiny \mbox{mp}}$ &Collection of rooted/pointed compact metric measure spaces with respectively $0$, $k$, $\infty$ many marked points.\\
	$\mathfrak{M}_0$,  $\mathfrak{M}^{(k)}_{\tiny \mbox{mp},0}$,  $\mathfrak{M}_{\tiny \mbox{mp},0}$ &Same as above but equipped with the null measure.\\
		$\mathfrak{M}^{\ast d}$,  $\mathfrak{M}^{(k),\ast d}_{\tiny \mbox{mp}}$ , $\mathfrak{M}^{\ast d}_{\tiny \mbox{mp}}$ &Same as above but equipped with $\R^d_+$-valued measures.\\
$\mathfrak{S}_n$ & Symmetric group on $n$ letters. \\
\, &	\\
\end{tabularx}
\end{table}

\begin{table}[H]
\begin{tabularx}{\textwidth}{p{0.30\textwidth}X}
\multicolumn{2}{l}{{\underline{Stochastic Processes, Paths:}}}\\
$\vec{X}^i = (X^{i,j};j\in[d])$ & $\Z^d$-valued random walk. When necessary, it has subscript $n$.\\
$\vec \bX^i = (X^{i,j};j\in[d])$ & $\R^d$-valued L\'{e}vy process.\\
$\overline{X}^{i,j}$& $X^{i,j }\circ\tau^{i}$.\\
$\Psi_i,\psi_i$ & Laplace exponents of $\vec \bX^i$ and $X^{i,i}$.\\
$\tau^i(k)$ & First hitting time at $-k$ of $X^{i,i}$.\\
$H^i$& Height process of $X^{i,i}$.\\
$\vec T(\vec r)$& Multivariate first hitting/passage times.\\
$U^{(h),j}$& Times approximating $\vec T(\vec r)$.\\
$\vec Z_n$ & MBGW branching process starting from $\lfloor b^i_n/a_n\rfloor $ many individuals 
type $i$, for all $i\in[d]$.\\
$\vec \be^i_n = (e_n^{i,1},\dotms, e_n^{i,d})$ & Excursion of $X^i_n$. When necessary, it has  an additional subscript $\bp\in \bbU$\\
$\vec \be^i = (e^{i,1},\dotms, e^{i,d})$ & Excursion of $\bX^i$. When necessary, it has an additional subscript $\bp\in \bbU$.\\
$g_\ell^i,d_\ell^i, \zeta_\ell^i$ & Endpoints of the excursion interval in $X^{i,i}$ containing $\ell$, and length of the excursion.\\
$(\vec \be,h), \zeta$ & Generic excursion and its length. \\
$\mathcal I(\vec f)$ & Set of excursion intervals of $\vec f$.\\
$\mathcal R(\vec f),\mathfrak L(\vec f)$ & Set of right and left end-points of $\mathcal I(\vec f)$.\\
$\mathcal{L}_{(k)}(\vec f)$  & Length of  the $k^\text{th}$-longest excursion.\\
$\mathcal E_k(\vec f)$ & $k$ longest excursions of $\vec f$ (with ties broken by order of appearance) and these excursions are listed in order of appearance.\\
$\vec{\mathcal{L}}_{k}(\vec f)$ & Sequence of lengths of $k$ longest excursions of $\vec f$.\\
$p_h$ & Canonical quotient map.\\
	\, &	\\
\end{tabularx}
\end{table}

\begin{table}[H]
\begin{tabularx}{\textwidth}{p{0.30\textwidth}X}
\multicolumn{2}{l}{{\underline{Decorations:}}}\\
$\cD$ & Generic decoration.\\
$\cD^{i_0,\vec \varphi}$ & Random decoration using the excursion measures $(\vec N^{i,\vec \varphi})_{i\in [d]}$.\\
$\cD(\bp)$ & Subdecorated metric measure space/subdecoration at $\bp\in \bbU$. \\
$D_\bp$ & Set of values of the subdecoration $\cD(\bp)$. Sometimes also referred as subdecoration.\\
$\widetilde{D}_\bp$ & Planted subdecoration. \\
$\sG^*(\cD)$ & Glued subdecorated space of the decoration $\cD$.\\
$\sG(\cD)$ & Glued decoration.\\ 
$\vec{\#}^{(i)}_jD_\bp$ & Number of marks in the subdecoration $D_\bp$ type $i$, of the form $x_{j+md}$ for some $m$.\\
$\vec{\#}^{(i)}_iD_\bp$ & Size of the subdecoration $D_\bp$.\\
$\cD^{i_0,\vec \Psi}_r$& Continuum decoration from the excursion measures $\vec \Psi$, with base subtree of type $i_0$ and $\zeta\geq r$.\\
$\sG(\cD^{i_0,\vec \Psi}_r)$ & Continuum Glued decoration.\\
$\sG(\cD^{n}_{r_n})$& Glued decoration with law $\sG(\cD^{n,i_0,\vec \varphi_n}_{r_n})\sim 	M^{i_0,\vec \varphi_n}(-| \#\cD^n(\emptyset) \ge r_n)
$.\\
$\sG(\tr, \cD), \sG_{\le h}^*(\cD)$ & Glued decorations over $\bp\in \tr$ and $|\bp|\leq h$, respectively.\\
\, &	\\
\multicolumn{2}{l}{{\underline{Measures:}}}\\
$g_\#\nu$ & Pushforward of the measure $\nu$ under $g$.\\
$\langle \vec{\mu},\vec{\epsilon}_j\rangle$ & $j$th coordinate measure for a vector-valued measure $\vec{\mu}$.\\
$\|\vec{\mu}\|$ & Sum of the coordinate measures for a vector-valued measure $\mu$.\\	
$\vec{\mu}^{(i)}$ & Counting measure on the type $i$ vertices on $\tr$.\\
$\vec \bN^i$ & Excursion measure of $(\vec \bX^i,H^i)$.\\
$\vec {\mathcal{N}^i}$ & Poisson random measure with intensity $\Leb\otimes \vec \bN^i$.\\
$M^{i_0,\vec \varphi}$ & Excursion measure of the glued decoration $\sG(\cD^{i_0,\vec{\varphi}})$.\\
$\bM^{i_0,\vec \Psi}_r$ & Law of the Continuum Glued decoration $\sG(\cD^{i_0,\vec{\Psi}}_r)$.\\
$\vec N^{i,\vec \varphi},\vec N^{\vec \varphi^i}$ & Excursion measure of $\vec X^{i}$.\\
\, &	\\
\end{tabularx}
\end{table}

\begin{table}[H]
\begin{tabularx}{\textwidth}{p{0.30\textwidth}X}
\multicolumn{2}{l}{{\underline{Trees:}}}\\
$\tr$, $\tred$ & Finite plane tree and its reduced tree.\\
$\tr'_{\bp}, \tr'_{j,(k)}$&Subtrees associated with some $w\in \tred$ or  of a given type, in a multitype tree $\tr$.\\
$\widetilde{\tr}$ & Planted planar tree (tree $\tr$ with added additional vertex joined to the (old) root).\\ 
$\left( T_n, \rho_n, d_n,  \vec{\mu}_n\right)$ & MBGW tree, with root $\rho_n$, graph distance $d_n$ and $\mathbb{R}^d_+$-valued measure $\vec{\mu}_n$.\\
$T_{n,r_n}$ & MBGW tree started from a type $i_0$ individual, conditioned on the type $i_0$ subtree containing the root has size at least $r_n$.\\
$\T_{h} = (\T_{h}, \rho, d_h,\vec{\mu})$ & Single-type L\'evy tree coded by $h$.\\
$\sT_t$ & Multitype L\'evy tree with law $\bM^{i_0}_r$.\\  
$\chi_\tr(v)$ & Number of children of $v$ in $\tr$. \\
$X_\tr$ & \L ukasiewicz path of $\tr$.\\
${\rm hgt}(v)$ & Height of vertex $v$.\\  
$H_\tr$ & Height function of $\tr$. \\
$\TYPE(v)$ & Type of vertex $v$ in a multitype tree\\
$\PAR(v)$ & Parent of vertex $v$.\\
$\vec \chi(v)$ & Vector with number of children of $v$ of each type. \\ 
$\varphi$ and $\vec \varphi$ & Offspring distribution of a single-type and multitype BGW tree.\\
$\vec \xi ^i$ & Random vector of number of children from a vertex type $i$ in a MBGW.\\
\end{tabularx}
\end{table}

\begin{funding}
 The first author was supported in part by the Deutsche Forschungsgemeinschaft (through grant DFG-SPP-2265).
 The second author was supported in part by NFS-DMS 2023239.
\end{funding}

\bibliography{references}

\providecommand{\bysame}{\leavevmode\hbox to3em{\hrulefill}\thinspace}
\providecommand{\MR}{\relax\ifhmode\unskip\space\fi MR }
\providecommand{\MRhref}[2]{%
  \href{http://www.ams.org/mathscinet-getitem?mr=#1}{#2}
}
\providecommand{\href}[2]{#2}
\begin{thebibliography}{DvdHVLS17}

\bibitem[Abb17]{Abbe:2017}
Emmanuel Abbe, \emph{Community detection and stochastic block models: recent
  developments}, J. Mach. Learn. Res. \textbf{18} (2017), Paper No. 177, 86.
  \MR{3827065}

\bibitem[ABBG12]{ABBG.12}
L.~Addario-Berry, N.~Broutin, and C.~Goldschmidt, \emph{The continuum limit of
  critical random graphs}, Probab. Theory Related Fields \textbf{152} (2012),
  no.~3-4, 367--406. \MR{2892951}

\bibitem[ABBGM17]{ABBGM.17}
Louigi Addario-Berry, Nicolas Broutin, Christina Goldschmidt, and Gr\'{e}gory
  Miermont, \emph{The scaling limit of the minimum spanning tree of the
  complete graph}, Ann. Probab. \textbf{45} (2017), no.~5, 3075--3144.
  \MR{3706739}

\bibitem[ABBST25]{addarioberry2025scalinglimitsmultitypebienayme}
Louigi Addario-Berry, Philipp Beltran, Benedikt Stufler, and Paul Thévenin,
  \emph{Scaling limits of multitype bienaym\'e trees}, 2025.

\bibitem[ADH14]{ADH.14}
Romain Abraham, Jean-Fran{\c{c}}ois Delmas, and Patrick Hoscheit, \emph{Exit
  times for an increasing l{\'e}vy tree-valued process}, Probability Theory and
  Related Fields \textbf{159} (2014), no.~1, 357--403.

\bibitem[AHUB20]{MR4130409}
Osvaldo Angtuncio~Hern\'{a}ndez and Ger\'{o}nimo Uribe~Bravo, \emph{Dini
  derivatives and regularity for exchangeable increment processes}, Trans.
  Amer. Math. Soc. Ser. B \textbf{7} (2020), 24--45. \MR{4130409}

\bibitem[Ald91a]{Aldous.91}
David Aldous, \emph{The continuum random tree. {I}}, Ann. Probab. \textbf{19}
  (1991), no.~1, 1--28. \MR{1085326}

\bibitem[Ald91b]{MR1085326}
\bysame, \emph{The continuum random tree. {I}}, Ann. Probab. \textbf{19}
  (1991), no.~1, 1--28. \MR{1085326}

\bibitem[Ald91c]{Aldous.91a}
\bysame, \emph{The continuum random tree. {II}. {A}n overview}, Stochastic
  analysis ({D}urham, 1990), London Math. Soc. Lecture Note Ser., vol. 167,
  Cambridge Univ. Press, Cambridge, 1991, pp.~23--70. \MR{1166406}

\bibitem[Ald93]{Aldous.93}
\bysame, \emph{The continuum random tree. {III}}, Ann. Probab. \textbf{21}
  (1993), no.~1, 248--289. \MR{1207226}

\bibitem[Ald97]{Aldous.97}
\bysame, \emph{Brownian excursions, critical random graphs and the
  multiplicative coalescent}, Ann. Probab. \textbf{25} (1997), no.~2, 812--854.
  \MR{1434128}

\bibitem[ALW16]{MR3522292}
Siva Athreya, Wolfgang L\"{o}hr, and Anita Winter, \emph{The gap between
  {G}romov-vague and {G}romov-{H}ausdorff-vague topology}, Stochastic Process.
  Appl. \textbf{126} (2016), no.~9, 2527--2553. \MR{3522292}

\bibitem[AM08]{MR2438817}
Marie Albenque and Jean-Fran\c~cois Marckert, \emph{Some families of increasing
  planar maps}, Electron. J. Probab. \textbf{13} (2008), no. 56, 1624--1671.
  \MR{2438817}

\bibitem[AN72]{AN.72}
Krishna~B. Athreya and Peter~E. Ney, \emph{Branching processes}, Die
  Grundlehren der mathematischen Wissenschaften, vol. Band 196,
  Springer-Verlag, New York-Heidelberg, 1972. \MR{373040}

\bibitem[{Ang}20]{Hernandez:2020}
Osvaldo {Angtuncio Hern{\'a}ndez}, \emph{{On Multitype Random Forests with a
  Given Degree Sequence, the Total Population of Branching Forests and
  Enumerations of Multitype Forests}}, arXiv e-prints (2020), arXiv:2003.03036.

\bibitem[AP00]{AP.00}
David Aldous and Jim Pitman, \emph{Inhomogeneous continuum random trees and the
  entrance boundary of the additive coalescent}, Probability theory and related
  fields \textbf{118} (2000), 455--482.

\bibitem[Arc24]{MR4780507}
Eleanor Archer, \emph{Random walks on decorated {G}alton-{W}atson trees}, Ann.
  Inst. Henri Poincar\'e{} Probab. Stat. \textbf{60} (2024), no.~3, 1849--1904.
  \MR{4780507}

\bibitem[BBSW14]{BBSW.14}
Shankar {Bhamidi}, Nicolas {Broutin}, Sanchayan {Sen}, and Xuan {Wang},
  \emph{{Scaling limits of random graph models at criticality: Universality and
  the basin of attraction of the Erd{\H{o}}s-R{\'e}nyi random graph}}, arXiv
  e-prints (2014), arXiv:1411.3417.

\bibitem[BCR25]{bertoin2025selfsimilarmarkovtreesscaling}
Jean Bertoin, Nicolas Curien, and Armand Riera, \emph{Self-similar markov trees
  and scaling limits}, 2025.

\bibitem[BDW21]{BDW.21}
Nicolas Broutin, Thomas Duquesne, and Minmin Wang, \emph{Limits of
  multiplicative inhomogeneous random graphs and {L}\'{e}vy trees: limit
  theorems}, Probab. Theory Related Fields \textbf{181} (2021), no.~4,
  865--973. \MR{4344135}

\bibitem[Ber96]{Bertoin.96}
Jean Bertoin, \emph{L\'{e}vy processes}, Cambridge Tracts in Mathematics, vol.
  121, Cambridge University Press, Cambridge, 1996. \MR{1406564}

\bibitem[Ber17]{Bertoin.17}
\bysame, \emph{Markovian growth-fragmentation processes}, Bernoulli \textbf{23}
  (2017), no.~2, 1082--1101. \MR{3606760}

\bibitem[Bet15]{MR3335010}
J\'er\'emie Bettinelli, \emph{Scaling limit of random planar quadrangulations
  with a boundary}, Ann. Inst. Henri Poincar\'e{} Probab. Stat. \textbf{51}
  (2015), no.~2, 432--477. \MR{3335010}

\bibitem[BJR07]{BJR.07}
B\'{e}la Bollob\'{a}s, Svante Janson, and Oliver Riordan, \emph{The phase
  transition in inhomogeneous random graphs}, Random Structures Algorithms
  \textbf{31} (2007), no.~1, 3--122. \MR{2337396}

\bibitem[BO18]{MR3748121}
Gabriel~Hern\'{a}n Berzunza~Ojeda, \emph{On scaling limits of multitype
  {G}alton-{W}atson trees with possibly infinite variance}, ALEA Lat. Am. J.
  Probab. Math. Stat. \textbf{15} (2018), no.~1, 21--48. \MR{3748121}

\bibitem[Car16]{MR3573291}
Alessandra Caraceni, \emph{The scaling limit of random outerplanar maps}, Ann.
  Inst. Henri Poincar\'e{} Probab. Stat. \textbf{52} (2016), no.~4, 1667--1686.
  \MR{3573291}

\bibitem[CH17]{CH.17}
Nicolas Curien and B\'{e}n\'{e}dicte Haas, \emph{Random trees constructed by
  aggregation}, Ann. Inst. Fourier (Grenoble) \textbf{67} (2017), no.~5,
  1963--2001. \MR{3732681}

\bibitem[Cha15]{MR3444314}
Lo\"{\i}c Chaumont, \emph{Breadth first search coding of multitype forests with
  application to {L}amperti representation}, In memoriam {M}arc
  {Y}or---{S}\'{e}minaire de {P}robabilit\'{e}s {XLVII}, Lecture Notes in
  Math., vol. 2137, Springer, Cham, 2015, pp.~561--584. \MR{3444314}

\bibitem[CHK15]{MR3382675}
Nicolas Curien, B\'en\'edicte Haas, and Igor Kortchemski, \emph{The {CRT} is
  the scaling limit of random dissections}, Random Structures Algorithms
  \textbf{47} (2015), no.~2, 304--327. \MR{3382675}

\bibitem[CKG23]{CKG.23}
Guillaume Conchon-Kerjan and Christina Goldschmidt, \emph{The stable graph: the
  metric space scaling limit of a critical random graph with iid power-law
  degrees}, The Annals of Probability \textbf{51} (2023), no.~1, 1--69.

\bibitem[CKL24]{CKL.22}
David {Clancy, Jr.}, Vitalii Konarovskyi, and Vlada Limic, \emph{Excursion
  representation of the degree-corrected stochastic blockmodel}, In Preparation
  (2024).

\bibitem[CL16]{MR3449255}
Lo\"ic Chaumont and Rongli Liu, \emph{Coding multitype forests: application to
  the law of the total population of branching forests}, Trans. Amer. Math.
  Soc. \textbf{368} (2016), no.~4, 2723--2747. \MR{3449255}

\bibitem[CM20]{MR4193902}
Lo\"{\i}c Chaumont and Marine Marolleau, \emph{Fluctuation theory for
  spectrally positive additive {L}\'{e}vy fields}, Electron. J. Probab.
  \textbf{25} (2020), Paper No. 161, 26. \MR{4193902}

\bibitem[CM21]{Chaumont:2021}
Lo{\"\i}c {Chaumont} and Marine {Marolleau}, \emph{{Extinction times of
  multitype, continuous-state branching processes}}, arXiv e-prints (2021),
  arXiv:2109.02912.

\bibitem[CM23]{MR4575008}
Lo\"ic Chaumont and Marine Marolleau, \emph{Extinction times of multitype
  continuous-state branching processes}, Ann. Inst. Henri Poincar\'e{} Probab.
  Stat. \textbf{59} (2023), no.~2, 563--577. \MR{4575008}

\bibitem[CPGUB13]{MR3098685}
M.~Emilia Caballero, Jos{\'e}~Luis P{\'e}rez~Garmendia, and Ger{\'o}nimo
  Uribe~Bravo, \emph{A {L}amperti-type representation of continuous-state
  branching processes with immigration}, Ann. Probab. \textbf{41} (2013),
  no.~3A, 1585--1627. \MR{3098685}

\bibitem[CPGUB17]{MR3689968}
M.~Emilia Caballero, Jos\'e~Luis P\'erez~Garmendia, and Ger\'onimo Uribe~Bravo,
  \emph{Affine processes on {$\mathbb{R}_+^m\times\mathbb{R}^n$} and
  multiparameter time changes}, Ann. Inst. Henri Poincar\'e Probab. Stat.
  \textbf{53} (2017), no.~3, 1280--1304. \MR{3689968}

\bibitem[Des24]{MR4819750}
Colin Desmarais, \emph{Depths in random recursive metric spaces}, J. Appl.
  Probab. \textbf{61} (2024), no.~4, 1448--1462. \MR{4819750}

\bibitem[DHS19]{MR3916322}
Luc Devroye, Cecilia Holmgren, and Henning Sulzbach, \emph{Heavy subtrees of
  {G}alton-{W}atson trees with an application to {A}pollonian networks},
  Electron. J. Probab. \textbf{24} (2019), Paper No. 2, 44. \MR{3916322}

\bibitem[DLG02]{MR1954248}
Thomas Duquesne and Jean-Fran{\c{c}}ois Le~Gall, \emph{Random trees, {L}\'evy
  processes and spatial branching processes}, Ast\'erisque (2002), no.~281,
  vi+147. \MR{1954248}

\bibitem[dR17a]{deRaphelis.17}
Lo\"{\i}c de~Raph\'{e}lis, \emph{Scaling limit of multitype {G}alton-{W}atson
  trees with infinitely many types}, Ann. Inst. Henri Poincar\'{e} Probab.
  Stat. \textbf{53} (2017), no.~1, 200--225. \MR{3606739}

\bibitem[dR17b]{MR3606739}
\bysame, \emph{Scaling limit of multitype {G}alton-{W}atson trees with
  infinitely many types}, Ann. Inst. Henri Poincar\'{e} Probab. Stat.
  \textbf{53} (2017), no.~1, 200--225. \MR{3606739}

\bibitem[DvdHVLS17]{DvdHvLS.17}
Souvik Dhara, Remco van~der Hofstad, Johan~SH Van~Leeuwaarden, and Sanchayan
  Sen, \emph{Critical window for the configuration model: finite third moment
  degrees}, Electronic Journal of Probability \textbf{22} (2017), 1--33.

\bibitem[DvdHvLS20]{DvdHvLS.20}
Souvik Dhara, Remco van~der Hofstad, Johan S.~H. van Leeuwaarden, and Sanchayan
  Sen, \emph{Heavy-tailed configuration models at criticality}, Ann. Inst.
  Henri Poincar\'{e} Probab. Stat. \textbf{56} (2020), no.~3, 1515--1558.
  \MR{4116701}

\bibitem[EPW06]{MR2221786}
Steven~N. Evans, Jim Pitman, and Anita Winter, \emph{Rayleigh processes, real
  trees, and root growth with re-grafting}, Probab. Theory Related Fields
  \textbf{134} (2006), no.~1, 81--126. \MR{2221786}

\bibitem[FP26]{MR5009763}
Mar\'ia~Clara Fittipaldi and Sandra Palau, \emph{On multitype branching
  processes with interaction}, Stochastics \textbf{98} (2026), no.~1, 73--94.
  \MR{5009763}

\bibitem[GH15a]{MR3317158}
Christina Goldschmidt and B\'en\'edicte Haas, \emph{A line-breaking
  construction of the stable trees}, Electron. J. Probab. \textbf{20} (2015),
  no. 16, 24. \MR{3317158}

\bibitem[GH15b]{GH.15}
Christina Goldschmidt and B\'{e}n\'{e}dicte Haas, \emph{A line-breaking
  construction of the stable trees}, Electron. J. Probab. \textbf{20} (2015),
  no. 16, 24. \MR{3317158}

\bibitem[GKPS99]{Gromov.99}
Mikhael Gromov, Misha Katz, Pierre Pansu, and Stephen Semmes, \emph{Metric
  structures for riemannian and non-riemannian spaces}, vol. 152, Springer,
  1999.

\bibitem[Gri74]{MR0362529}
Anders Grimvall, \emph{On the convergence of sequences of branching processes},
  Ann. Probability \textbf{2} (1974), 1027--1045. \MR{0362529 (50 \#14969)}

\bibitem[HHP25]{hernandez2025coalescentstructuremultitypecontinuoustime}
Osvaldo~Angtuncio Hern\'{a}ndez, Simon Harris, and Juan~Carlos Pardo, \emph{The
  coalescent structure of multitype continuous-time galton-watson trees}, 2025.

\bibitem[HLL83]{HLL.83}
Paul~W Holland, Kathryn~Blackmond Laskey, and Samuel Leinhardt,
  \emph{Stochastic blockmodels: First steps}, Social networks \textbf{5}
  (1983), no.~2, 109--137.

\bibitem[HM12]{MR3050512}
B\'en\'edicte Haas and Gr\'egory Miermont, \emph{Scaling limits of {M}arkov
  branching trees with applications to {G}alton-{W}atson and random unordered
  trees}, Ann. Probab. \textbf{40} (2012), no.~6, 2589--2666. \MR{3050512}

\bibitem[HNT18]{MR3876899}
Jan Hladk\'y, Asaf Nachmias, and Tuan Tran, \emph{The local limit of the
  uniform spanning tree on dense graphs}, J. Stat. Phys. \textbf{173} (2018),
  no.~3-4, 502--545. \MR{3876899}

\bibitem[HOW25]{hernandez2025scalinglimitaldousbroderchain}
Osvaldo~Angtuncio Hern\'andez, Gabriel~Berzunza Ojeda, and Anita Winter,
  \emph{Scaling limit of the aldous-broder chain on regular graphs: the
  transient regime}, 2025.

\bibitem[HS21]{HS.21}
B\'{e}n\'{e}dicte Haas and Robin Stephenson, \emph{Scaling limits of multi-type
  {M}arkov branching trees}, Probab. Theory Related Fields \textbf{180} (2021),
  no.~3-4, 727--797. \MR{4288331}

\bibitem[Hu{\v{s}}08]{Huvsek.08}
Miroslav Hu{\v{s}}ek, \emph{Urysohn universal space, its development and
  hausdorff's approach}, Topology and its Applications \textbf{155} (2008),
  no.~14, 1493--1501.

\bibitem[HW25]{haig2025randombipartitegraphsiid}
Alastair Haig and Minmin Wang, \emph{Random bipartite graphs with i.i.d.
  weights and applications to inhomogeneous random intersection graphs}, 2025.

\bibitem[Jag89]{MR1014449}
Peter Jagers, \emph{General branching processes as {M}arkov fields}, Stochastic
  Process. Appl. \textbf{32} (1989), no.~2, 183--212. \MR{1014449}

\bibitem[JM86]{MR827331}
A.~Joffe and M.~M\'{e}tivier, \emph{Weak convergence of sequences of
  semimartingales with applications to multitype branching processes}, Adv. in
  Appl. Probab. \textbf{18} (1986), no.~1, 20--65. \MR{827331}

\bibitem[Jr21]{clancy2021encodingmultitypegaltonwatsonforests}
David~Clancy Jr, \emph{Encoding multitype galton-watson forests and a multitype
  ray-knight theorem}, 2021.

\bibitem[Jr24a]{clancy2024componentsizesrank2multiplicative}
\bysame, \emph{Component sizes of rank-2 multiplicative random graphs}, 2024.

\bibitem[Jr24b]{clancy2024nearcriticalbipartiteconfigurationmodels}
\bysame, \emph{Near-critical bipartite configuration models and their
  associated intersection graphs}, 2024.

\bibitem[JS03]{MR1943877}
Jean Jacod and Albert~N. Shiryaev, \emph{Limit theorems for stochastic
  processes}, second ed., Grundlehren der Mathematischen Wissenschaften
  [Fundamental Principles of Mathematical Sciences], vol. 288, Springer-Verlag,
  Berlin, 2003. \MR{1943877}

\bibitem[JS13]{JS.13}
Jean Jacod and Albert Shiryaev, \emph{Limit theorems for stochastic processes},
  vol. 288, Springer Science \& Business Media, 2013.

\bibitem[JS15]{MR3342658}
Svante Janson and Sigurdur~\"Orn Stef\'ansson, \emph{Scaling limits of random
  planar maps with a unique large face}, Ann. Probab. \textbf{43} (2015),
  no.~3, 1045--1081. \MR{3342658}

\bibitem[Khe20]{Khezeli.20}
Ali Khezeli, \emph{Metrization of the {G}romov-{H}ausdorff (-{P}rokhorov)
  topology for boundedly-compact metric spaces}, Stochastic Process. Appl.
  \textbf{130} (2020), no.~6, 3842--3864. \MR{4092421}

\bibitem[KL21]{KL.21}
Vitalii Konarovskyi and Vlada Limic, \emph{Stochastic block model in a new
  critical regime and the interacting multiplicative coalescent}, Electron. J.
  Probab. \textbf{26} (2021), Paper No. 30, 23. \MR{4235481}

\bibitem[KN11]{Karrer:2011}
Brian Karrer and Mark~EJ Newman, \emph{Stochastic blockmodels and community
  structure in networks}, Physical review E \textbf{83} (2011), no.~1, 016107.

\bibitem[LG05]{MR2203728}
Jean-Fran{\c{c}}ois Le~Gall, \emph{Random trees and applications}, Probab.
  Surv. \textbf{2} (2005), 245--311. \MR{2203728}

\bibitem[LGLJ98a]{LL.98b}
Jean-Fran{\c{c}}ois Le~Gall and Yves Le~Jan, \emph{Branching processes in
  l{\'e}vy processes: Laplace functionals of snakes and superprocesses}, The
  Annals of Probability \textbf{26} (1998), no.~4, 1407--1432.

\bibitem[LGLJ98b]{LL.98a}
Jean-Francois Le~Gall and Yves Le~Jan, \emph{Branching processes in l{\'e}vy
  processes: the exploration process}, The Annals of Probability \textbf{26}
  (1998), no.~1, 213--252.

\bibitem[Li06]{MR2225068}
Zenghu Li, \emph{A limit theorem for discrete {G}alton-{W}atson branching
  processes with immigration}, J. Appl. Probab. \textbf{43} (2006), no.~1,
  289--295. \MR{2225068}

\bibitem[Lub74]{MR330393}
Arthur Lubin, \emph{Extensions of measures and the von {N}eumann selection
  theorem}, Proc. Amer. Math. Soc. \textbf{43} (1974), 118--122. \MR{330393}

\bibitem[MB22]{mijatovic2022limit}
Aleksandar Mijatovi{\'c} and Ger{\'o}nimo~Uribe Bravo, \emph{Limit theorems for
  local times and applications to sdes with jumps}, Stochastic Processes and
  their Applications \textbf{153} (2022), 39--56.

\bibitem[Mel07]{Melleray.07}
Julien Melleray, \emph{On the geometry of urysohn's universal metric space},
  Topology and its Applications \textbf{154} (2007), no.~2, 384--403.

\bibitem[Mie08a]{Miermont.08}
Gr{\'e}gory Miermont, \emph{Invariance principles for spatial multitype
  galton-watson trees}, Annales de l'IHP Probabilit{\'e}s et statistiques,
  vol.~44, 2008, pp.~1128--1161.

\bibitem[Mie08b]{MR2469338}
\bysame, \emph{Invariance principles for spatial multitype {G}alton-{W}atson
  trees}, Ann. Inst. Henri Poincar\'e Probab. Stat. \textbf{44} (2008), no.~6,
  1128--1161. \MR{2469338}

\bibitem[Mie09]{Miermont.09}
Gr\'{e}gory Miermont, \emph{Tessellations of random maps of arbitrary genus},
  Ann. Sci. \'{E}c. Norm. Sup\'{e}r. (4) \textbf{42} (2009), no.~5, 725--781.
  \MR{2571957}

\bibitem[Mil77]{MR0433606}
P.~W. Millar, \emph{Zero-one laws and the minimum of a {M}arkov process},
  Trans. Amer. Math. Soc. \textbf{226} (1977), 365--391. \MR{0433606}

\bibitem[MM11]{MR2829313}
Jean-Fran\c~cois Marckert and Gr\'egory Miermont, \emph{The {CRT} is the
  scaling limit of unordered binary trees}, Random Structures Algorithms
  \textbf{38} (2011), no.~4, 467--501. \MR{2829313}

\bibitem[MPB25]{mijatović2025criticalbranchingprocessesimmigration}
Aleksandar Mijatović, Benjamin Povar, and Gerónimo~Uribe Bravo,
  \emph{Critical branching processes with immigration: scaling limits of local
  extinction sets}, 2025.

\bibitem[Nev86]{MR850756}
J.~Neveu, \emph{Arbres et processus de {G}alton-{W}atson}, Ann. Inst. H.
  Poincar{\'e} Probab. Statist. \textbf{22} (1986), no.~2, 199--207.
  \MR{850756}

\bibitem[NP10]{NP.10}
Asaf Nachmias and Yuval Peres, \emph{Critical percolation on random regular
  graphs}, Random Structures \& Algorithms \textbf{36} (2010), no.~2, 111--148.

\bibitem[Pit06]{MR2245368}
J.~Pitman, \emph{Combinatorial stochastic processes}, Lecture Notes in
  Mathematics, vol. 1875, Springer-Verlag, Berlin, 2006, Lectures from the 32nd
  Summer School on Probability Theory held in Saint-Flour, July 7--24, 2002,
  With a foreword by Jean Picard. \MR{2245368}

\bibitem[PR04]{2004math.....10430P}
Yuval {Peres} and David {Revelle}, \emph{{Scaling limits of the uniform
  spanning tree and loop-erased random walk on finite graphs}}, arXiv
  Mathematics e-prints (2004), math/0410430.

\bibitem[RW18]{RW.18}
Franz Rembart and Matthias Winkel, \emph{Recursive construction of continuum
  random trees}, Ann. Probab. \textbf{46} (2018), no.~5, 2715--2748.
  \MR{3846837}

\bibitem[RY99]{MR1725357}
Daniel Revuz and Marc Yor, \emph{Continuous martingales and {B}rownian motion},
  third ed., Grundlehren der Mathematischen Wissenschaften [Fundamental
  Principles of Mathematical Sciences], vol. 293, Springer-Verlag, Berlin,
  1999. \MR{1725357}

\bibitem[S\'22]{Senizergues.20}
Delphin S\'{e}nizergues, \emph{Growing random graphs with a preferential
  attachment structure}, ALEA Lat. Am. J. Probab. Math. Stat. \textbf{19}
  (2022), no.~1, 259--309. \MR{4359793}

\bibitem[Sat99]{MR1739520}
Ken-iti Sato, \emph{L\'{e}vy processes and infinitely divisible distributions},
  Cambridge Studies in Advanced Mathematics, vol.~68, Cambridge University
  Press, Cambridge, 1999, Translated from the 1990 Japanese original, Revised
  by the author. \MR{1739520}

\bibitem[Sch09]{MR2496437}
Jason Schweinsberg, \emph{The loop-erased random walk and the uniform spanning
  tree on the four-dimensional discrete torus}, Probab. Theory Related Fields
  \textbf{144} (2009), no.~3-4, 319--370. \MR{2496437}

\bibitem[S{\'e}n19]{Senizergues.19}
Delphin S{\'e}nizergues, \emph{Random gluing of metric spaces}, The Annals of
  Probability \textbf{47} (2019), no.~6, 3812--3865.

\bibitem[Sko57]{MR94842}
Anatolii~Volodimirovich Skorokhod, \emph{Limit theorems for stochastic
  processes with independent increments}, Teor. Veroyatnost. i Primenen.
  \textbf{2} (1957), 145--177. \MR{94842}

\bibitem[vdH24]{vanderHofstad.24}
Remco van~der Hofstad, \emph{Random graphs and complex networks. {V}ol. 2},
  Cambridge Series in Statistical and Probabilistic Mathematics, Cambridge
  University Press, 2024.

\bibitem[Wan23]{Wang.23}
Minmin Wang, \emph{Large random intersection graphs inside the critical window
  and triangle counts}, 2023.

\bibitem[Whi80]{Whitt.80}
Ward Whitt, \emph{Some useful functions for functional limit theorems},
  Mathematics of operations research \textbf{5} (1980), no.~1, 67--85.

\end{thebibliography}
\bibliographystyle{amsalpha}

\end{document}